\documentclass[12pt]{amsart}

\usepackage[top=100pt,bottom=60pt,left=96pt,right=92pt]{geometry}
\usepackage{fullpage}

\usepackage[export]{adjustbox}
\usepackage[mathscr]{eucal}
\usepackage{hyperref}
\usepackage{amsmath,amsthm,amssymb,amsfonts,enumerate,xcolor,layout,tikz,subcaption,enumitem, verbatim} 
\usepackage[normalem]{ulem} 

\usetikzlibrary{shapes.geometric, arrows}

\usepackage{subcaption}
\allowdisplaybreaks

\newtheorem{theorem}{Theorem}[section]
\newtheorem{proposition}[theorem]{Proposition}
\newtheorem{corollary}[theorem]{Corollary}
\newtheorem{lemma}[theorem]{Lemma}

\theoremstyle{definition}
\newtheorem{definition}[theorem]{Definition}
\newtheorem{remark}[theorem]{Remark}
\newtheorem{example}[theorem]{Example}

\theoremstyle{plain}

\newtheorem{question}[theorem]{Question}

\definecolor{linkblue}{rgb}{0,0,.6}
\definecolor{citered}{rgb}{.7,0,0}
\hypersetup{colorlinks =true, linkcolor=linkblue, citecolor = citered, urlcolor=linkblue}

\usepackage{chngcntr}
\counterwithin{figure}{section}

\numberwithin{equation}{section}

\def\C{{\mathbb C}}

\def\P{{\mathbb P}}
\def\R{{\mathbb R}}
\def\mbS{{\mathbb S}} 
\def\CP{\C\P}
\def\T{{\mathbb T}}

\def\Z{{\mathbb Z}}

\newcommand{\De}{\Delta}
\newcommand{\om}{\omega}
\newcommand{\abs}[1]{\left| #1 \right|}
\newcommand{\mf}{{m_{\mathrm{f}}}}

\newcommand{\vungoc}{V\~u Ng\d{o}c}

\begin{document}

\title[Extending compact $\mbS^1$-actions]{Extending compact Hamiltonian $\mbS^1$-spaces to integrable systems with mild degeneracies in dimension four}
\author{Sonja Hohloch \quad $\&$ \quad Joseph Palmer}
\date{\today}

\begin{abstract}
Given any compact connected four dimensional symplectic manifold $(M,\om)$ and smooth function $J\colon M\to \R$
which generates an effective $\mbS^1$-action, we show that there exists a smooth function $H\colon M\to\R$ such that
$(M,\om,(J,H))$ is a completely (Liouville) integrable system of a type we call \emph{hypersemitoric} ---
these are systems for which all singularities are  non-degenerate, except possibly for a finite number of families of degenerate points of a relatively tame type called parabolic (also sometimes called cuspidal).
Such an $(M,\om,J)$ is often referred to as a Hamiltonian $\mbS^1$-space (classified by Karshon in 1999)
and we call any integrable system of the form $(M,\om,(J,H))$ an \emph{extension} of $(M,\om,J)$.
Using this terminology, our main result is that any
Hamiltonian $\mbS^1$-space can be extended to a hypersemitoric integrable system.

We also show that there exist Hamiltonian $\mbS^1$-spaces for which any extension must include at least one degenerate singular
point. 
Parabolic points are among the most common and natural degenerate points,
and 
thus hypersemitoric systems are in this sense
the `nicest' class of systems to which all Hamiltonian $\mbS^1$-spaces can be extended. We also prove several foundational results about these systems, such as the non-existence of loops of hyperbolic-regular
points and some properties about their fibers.\\

\end{abstract}

\maketitle


\section{Introduction}
\label{sec:intro}

Hamiltonian systems form an important and interesting class of dynamical systems since, on the one hand, many fundamental physical phenomena can be modeled as Hamiltonian systems and, on the other hand, questions of symplectic rigidity and conservation laws are naturally linked to Hamiltonian systems. Examples of Hamiltonian systems are common in mathematics, physics, and the other sciences, such as the n-body problem in celestial mechanics, certain Li\'{e}nard type and Van der Pol type equations
in biology, the Euler-Lagrange equations in calculus of variations, and the interaction of chaotic
and non-chaotic aspects in KAM theory. 

Roughly, if a Hamiltonian system has the maximal possible number of independent conserved quantities then it is called integrable.
Integrable systems form a field with a long tradition at the intersection of dynamical systems, ODEs, PDEs, symplectic geometry, Lie theory, algebraic geometry, classical mechanics, mathematical physics, and so on. Many classical, familiar systems are integrable, such as the spherical pendulum, coupled angular momenta system, Gelfand-Zeitlin systems, and the Lagrange, Euler, and Kovalevskaya spinning tops.


In fact, all of these examples share a special property: they have a circular symmetry.
Integrable systems which display such circular symmetries are not rare but rather a very natural and common occurrence.

The symplectic geometry of Hamiltonian $\mbS^1$-actions has been studied essentially since the invention of symplectic geometry, and the field of integrable systems in classical mechanics has been active for even longer. Yet the interaction between these two fields has not been thoroughly investigated until quite recently. The present paper aims to understand the relationship between them in dimension four.
More precisely, we will show that any effective Hamiltonian $\mbS^1$-action on a compact connected symplectic four dimensional manifold can be extended to a completely integrable Hamiltonian system on this manifold of a relatively ``nice'' type. Moreover, along the way, we prove several interesting properties of integrable systems with underlying $\mbS^1$ symmetries.

\subsection{Integrable systems and \texorpdfstring{$\mbS^1$-actions}{S1 actions}}
To state our results properly, we first need to fix some definitions and conventions: First of all, 

\centerline{\emph{throughout this paper we will assume that all manifolds $M$ are connected.}}

\noindent
Let $(M,\om)$ be a symplectic manifold. Then any smooth function $J\colon M\to\R$ determines a vector field $\mathcal{X}^J$ on $M$ via the equation $\om(\mathcal{X}^J,\cdot) = -\mathrm{d}J$. The function $J$ is then usually referred to as the {\em Hamiltonian}, $\mathcal{X}^J$ as the \emph{Hamiltonian vector field} of $J$, and $\dot{z} = \mathcal{X}^J(z)$ as Hamiltonian system induced by $J$ or $\mathcal{X}^J(z)$.

We are interested in particular in the following type of Hamiltonian systems:

\begin{definition}
\label{def:HamSspace}
We call $(M,\om,J)$ a \emph{Hamiltonian $\mbS^1$-space} if $(M,\om)$ is a compact four dimensional symplectic manifold with Hamiltonian $J:M \to \R$ and the flow of $\mathcal{X}^J$ is periodic of minimal period $2\pi$.
In other words, the Hamiltonian flow of such a $J$ generates an effective Hamiltonian action of $\mbS^1=\R/2\pi \Z$ on $M$.
\end{definition}

Such Hamiltonian $\mbS^1$-spaces were classified up to isomorphism by Karshon \cite{karshon} in terms of a labeled graph encoding information about the fixed points and isotropy groups ($\Z_k$-spheres) of the $\mbS^1$-action --- for the readers' convenience, we will recall the necessary details in Section~\ref{sec:S1-classification}.

This paper aims at extending effective Hamiltonian $\mbS^1$-actions to integrable systems, so we need to fix our notion of integrability: We will work with so-called `Liouville integrability' which is defined as follows.

\begin{definition}
A triple $(M,\om,F=(f_1,\ldots,f_n))$ is a \emph{$2n$-dimensional completely integrable system} (briefly an \emph{integrable system}) if $(M,\om)$ is a $2n$-dimensional symplectic manifold and $F\colon M\to\R^n$, called the \emph{momentum map}, satisfies:
\begin{enumerate}
 \item\label{cond:commute} $\om (\mathcal{X}^{f_i},\mathcal{X}^{f_j})=0$ for all $i,j\in\{1,\ldots, n\}$;
 \item\label{cond:independent} $\mathcal{X}^{f_1}(p), \ldots, \mathcal{X}^{f_n}(p)\in T_p M$ are linearly independent for almost all $p\in M$.
\end{enumerate}
\end{definition}
We say that an integrable system $(M,\om,F)$ is \emph{compact} if $M$ is compact. The points in $M$ at which Condition~\eqref{cond:independent} fails are called \emph{singular points}, and the other points
of $M$ are called \emph{regular points}.

\subsection{Extending \texorpdfstring{$\mbS^1$-actions}{S1 actions} to integrable systems}

Given a four dimensional compact integrable system $(M,\om,(J,H))$ such that $J$ generates an $\mbS^1$-action, it is clear that, by simply forgetting the function $H$, we obtain the Hamiltonian $\mbS^1$-space $(M,\om,J)$.

\begin{definition}
If $(M,\om,J)$ is a Hamiltonian $\mbS^1$-space and $H\colon M\to\R$ is such that $(M,\om,(J,H))$ is an integrable system, then $(M,\om,(J,H))$ is said to be an \emph{extension to an integrable system} of $(M,\om,J)$ and, conversely, $(M,\om,J)$ is the \emph{underlying} Hamiltonian $\mbS^1$-space of $(M,\om,(J,H))$.
\end{definition}

Whereas recovering an underlying Hamiltonian $\mbS^1$-space $(M,\om,J)$ by forgetting $H$ is trivial, the converse is everything but obvious:

\begin{question}\label{question}
 Given a compact four dimensional Hamiltonian $\mbS^1$-space $(M,\om,J)$, when and, if yes, how can we extend it to an integrable system $(M,\om,(J,H))$? What do these integrable systems look like, in particular the function $H: M \to \R$? What are the `nicest possible' resulting integrable systems $(M,\om,(J,H))$?
\end{question}

Looking at this problem from a different angle, one is in fact studying an `intermediate situation' in the sense that one component of the moment map induces an $\mbS^1$-action and the other one an $\R$-action on the manifold. If both components induce $\R$-actions, we are dealing with general integable systems which can exhibit extremely complicated behavior. If, on the other hand, both components induce $\mbS^1$-actions, we are facing an induced Hamiltonian 2-torus action on the underlying manifold which imposes quite a lot of restrictions on the system as a whole.

\begin{definition}
A completely integrable system $(M,\om,F=(f_1,\ldots,f_n))$ is \emph{toric} if each of the $f_1$, \dots, $f_n$ induces an $\mbS^1$-action of minimal period $2 \pi$.
\end{definition}

After classifying Hamiltonian $\mbS^1$-spaces, Karshon~\cite{karshon} studied and solved Question \ref{question} concerning four dimensional toric systems, i.e., she found the necessary and sufficient conditions under which a Hamiltonian $\mbS^1$-space $(M,\om,J)$ extends to a toric integrable system $(M,\om,F=(J,H))$, see Section \ref{sec:S1-toric} for more details.

\subsection{Main results}

As mentioned in the previous subsection, Karshon was already considering Question~\ref{question} in her paper~\cite{karshon} containing the original classification of Hamiltonian $\mbS^1$-spaces, 
where she proved exactly which Hamiltonian $\mbS^1$-spaces can be extended to a toric system. More 
recently\footnote{first announced at Poisson 2014 in a talk by Daniele Sepe.}
Hohloch \& Sabatini \& Sepe \& Symington~\cite{HSSS2} have exactly described which Hamiltonian $\mbS^1$-spaces can be extended to so-called \emph{semitoric systems} (cf.\ Definition \ref{def:semitoric}). 
Semitoric systems are a generalization
of toric systems in dimension four which can include all of the types of singularities that
appear in toric systems and additionally a type of non-degenerate singularity known
as a \emph{focus-focus} singular point.

We now define \emph{hypersemitoric systems}, which represent a substantial generalization of semitoric systems.
In particular, hypersemitoric systems allow all types of non-degenerate singularities, including singularities with hyperbolic components, and certain relatively mild degenerate singular points which naturally occur with hyperbolic-regular singularities in many cases
(see Section~\ref{sec:singularities} for an explanation of non-degenerate and degenerate singular points).
Singularities of hyperbolic-hyperbolic type are not explicitly ruled out in hypersemitoric systems, but cannot appear due to the presence of the
global $\mbS^1$-action. All other types of non-degenerate points can appear.

\begin{definition}\label{def:hypersemitoric}
 An integrable system $(M,\om,F=(J,H))$ is called
 a \emph{hypersemitoric system} if:
 \begin{enumerate}
  \item $J$ is proper and generates an effective $\mbS^1$-action;
  \item all degenerate singular points of $F$ (if any) are of parabolic type (as in Definition~\ref{def:parabolic}).
 \end{enumerate}
\end{definition}


Hypersemitoric systems form a significantly more general class than semitoric systems (cf.~Remark~\ref{rmk:generic}), which are in
turn significantly more abundant than toric ones.

In the first part of this paper we give several examples and prove certain properties of such systems.
For instance, we show that hypersemitoric systems do not admit `loops of hyperbolic-regular values' (see Corollary~\ref{cor:no-loops} for a more precise statement).


In the second part of this paper we show that \emph{any} compact four dimensional Hamiltonian $\mbS^1$-space can be extended to a hypersemitoric system.
That is, we prove:

\begin{theorem}\label{thm:extending}
 Let $(M,\om,J)$ be a $4$-dimensional Hamiltonian $\mbS^1$-space where $(M,\om)$ is a compact symplectic manifold.
 Then there exists a smooth function $H\colon M\to\R$ such that $(M,\om,(J,H))$ is a hypersemitoric system. 
\end{theorem}

We sketch the idea of the proof of Theorem~\ref{thm:extending} in Section~\ref{sec:sketch}.
The proof of Theorem~\ref{thm:extending}, given in Section~\ref{sec:proof}, actually gives a slightly more refined result about the properties of the resulting  hypersemitoric system, which we state as Corollary~\ref{cor:refined-extending}.

The resulting integrable system in Theorem~\ref{thm:extending} is not unique, and studying the collection of such systems is an interesting question in its own right, see Remark~\ref{rmk:generic}.


Singular fibers of hyperbolic-regular type always occur in one-parameter families, and in many cases (see Corollary~\ref{cor:no-loops})
the two endpoints of the family correspond to fibers that include degenerate singular points, which is why we allow certain degenerate singularities in hypersemitoric systems.
In fact, there exist systems which cannot be extended to an integrable system
with no degenerate singular points (Corollary~\ref{cor:forced-degen-2ndversion}), so to have any hope
to be able to extend all Hamiltonian $\mbS^1$-spaces to a class of integrable systems, we
must include some type of degenerate points.
Degenerate singularities of parabolic type are one of the simplest classes of degenerate singularities among the typical degenerate singularities discussed in Bolsinov \& Fomenko~\cite{BolFom},
and they are stable under perturbation~\cite{Giacobbe-stable}.
That is, parabolic points cannot be removed from a system by completely integrable
perturbations,
and they are therefore common in nature as well. 
For instance, they appear
in the Kovalevskaya top~\cite{BolRikFom-Ktop} and many other systems from rigid body dynamics, see
the references in~\cite{Bol-parabolic}.

Thus, the motivation of this paper is two-fold:
\begin{itemize}
 \item \textbf{Motivation 1:} non-degenerate singularities and parabolic singularities are the most common ones in nature. The known classifications (toric and semitoric) prohibit the existence of certain non-degenerate singular points (those containing hyperbolic blocks) and parabolic points, but the time is ripe to extend the toric and semitoric techniques to systems including these singularities;\\
 \item \textbf{Motivation 2:} there are Hamiltonian $\mbS^1$-spaces which cannot be extended to either a
  toric or a semitoric system, so to extend all Hamiltonian $\mbS^1$-spaces to a class of integrable systems
  we must consider a more general class. The class of hypersemitoric systems is in some sense the `easiest and smallest' class with this property.
\end{itemize}


\subsection{Background and interactions of this paper with other works}

For many years interactions between the classical field of integrable systems and the
more modern 
field of compact Hamiltonian group actions on symplectic manifolds
have lead to interesting results.
One of the earliest and best known examples of results in this direction is
the classification of effective Hamiltonian $\T^n$-actions on compact symplectic $2n$-manifolds, which
can equivalently be thought of as a classification of compact toric integrable systems:
Atiyah~\cite{Atiyah} and Guillemin \& Sternberg~\cite{GuilStern} showed that if $(M,\om,F)$ is a toric system
then the image $F(M)$ is a convex $n$-dimensional polytope,
and moreover Delzant~\cite{delzant} showed that toric systems are classified up to isomorphism
by this convex polytope, up to the action of the affine group $\mathrm{GL}(n,\R)\ltimes \R^n$. 
Karshon \& Lerman~\cite{karshon-lerman} generalized this to the non-compact case.
Given a fixed compact symplectic four manifold, the $\mbS^1$- and $2$-torus actions on that manifold have also been
classified by Holm \& Kessler~\cite{HolmKessler2019} and Karshon \& Kessler \& Pinsonnault~\cite{KKP-counting2015}.

Another important classification result is that of semitoric integrable
systems, which generalize toric integrable systems in dimension four.
\begin{definition}\label{def:semitoric}
A four dimensional integrable system $(M,\om,F=(J,H))$ is a \emph{semitoric integrable system}, or briefly a \emph{semitoric system}, if:
\begin{enumerate}
 \item $J$ is proper and generates an effective $\mbS^1$-action;
 \item\label{cond:no-hyp} all singular points of $F=(J,H)$ are non-degenerate and do not include hyperbolic blocks
  (i.e.~there are no singular points of hyperbolic-regular, hyperbolic-elliptic, or hyperbolic-hyperbolic type, as described
   in Section~\ref{sec:singularities}).
\end{enumerate}
\end{definition}
About ten years ago, semitoric systems were classified in terms of five invariants by Pelayo \& V\~{u} Ng\d{o}c~\cite{PVN09, PVN11}, generalizing Delzant's \cite{delzant} toric classification. Their original classification has the extra assumption that the systems must be \emph{simple} (see Section \ref{sec:semitoric}), but this assumption has been removed recently by Palmer \& Pelayo \& Tang~\cite{PPT}.
Semitoric systems are much more general than toric systems, and their behavior is much more complicated due to the presence of focus-focus singularities which cannot occur in toric systems.
Note that the semitoric classification includes both compact and non-compact systems.


A semitoric system
naturally comes with the structure of an underlying Hamiltonian $\mbS^1$-action. The relationship between the 
semitoric classification and Karshon's classification of Hamiltonian $\mbS^1$-actions 
on compact 4-manifolds was studied
by Hohloch \& Sabatini \& Sepe~\cite{HSS}.

Suppose that $(M,\om,J)$ is a $4$-dimensional Hamiltonian $\mbS^1$-space where $(M,\om)$ is a compact symplectic manifold.
Combining Theorem~\ref{thm:extending} with the results of Karshon~\cite{karshon} and those of Hohloch \& Sabatini \& Sepe \& Symington~\cite{HSSS2}, one obtains the following: 
 \begin{enumerate} [label=(\arabic*),font=\upshape]
  \item\label{item:combinedKarshon} There exists $H\colon M\to\R$ such that $(M,\om,(J,H))$ is a toric system if and only if each fixed surface (if any exists) has genus zero and each non-extremal level set of $J$ contains at most two non-free orbits of the $\mbS^1$-action (see Karshon~\cite[Proposition 5.21]{karshon}\,);
  \item\label{item:combinedHSSS} There exists $H\colon M\to\R$ such
  that $(M,\om,(J,H))$ is a semitoric system if and only if each fixed surface has genus zero (if any exists) and each non-extremal level set of $J$ contains at most two non-free
  orbits of the $\mbS^1$-action which are not fixed points (see Hohloch $\&$ Sabatini $\&$ Sepe $\&$ Symington~\cite{HSSS2}\,);
  \item\label{item:combinedHP} In all cases, there exists $H\colon M\to\R$ such that $(M,\om,(J,H))$ is a hypersemitoric system (see Theorem~\ref{thm:extending}). 
 \end{enumerate}

The relationship between these results is shown in the diagram in Figure~\ref{fig:diagram}.

\begin{figure}
\centering

\tikzstyle{top-process} = [rectangle, rounded corners, minimum width=4.5
cm, minimum height=2.5cm, text centered, text width=4.5cm, draw=black, fill=blue!5]
\tikzstyle{bottom-process} = [rectangle, rounded corners, minimum width=4.5
cm, minimum height=3cm, text centered, text width=4.5cm, draw=black, fill=blue!5]
\tikzstyle{arrow} = [thick,->,>=stealth]

\begin{tikzpicture}[node distance=2cm]

\node (toric) [top-process] {\textbf{Toric systems}\\[1em]\footnotesize{classified by Delzant~\cite{delzant}  \\[.5em] $(\mbS^1\times\mbS^1)$-action}};

\node (semitoric) [top-process, right of=toric, xshift=3.25cm] {\textbf{Semitoric systems}\\[1em]\footnotesize{classified by Pelayo \& V\~{u} Ng\d{o}c~\cite{PVN09,PVN11}  \\[.1em]
$(\mbS^1\times\R)$-action}};

\node (hypersemitoric) [top-process, right of=semitoric, xshift=3.25cm] {\textbf{Hypersemitoric systems}\\[1em]\footnotesize{not yet classified\\[.1em]
$(\mbS^1\times\R)$-action}};

\node (S1-toric) [bottom-process, below of=toric,yshift=-2.5cm] {Hamiltonian $\mbS^1$-spaces satisfying:\footnotesize{\begin{itemize}[noitemsep, leftmargin=.5cm] \item $g=0$ \item  at most two non-free orbits in each $J^{-1}(j_\mathrm{int})$
\end{itemize}
classified by Karshon\cite{karshon}}};

\node (S1-semitoric) [bottom-process, below of=semitoric,yshift=-2.5cm] {Hamiltonian $\mbS^1$-spaces satisfying:\footnotesize{\begin{itemize}[noitemsep, leftmargin=.5cm] \item  $g=0$ \item at most two non-free, non-fixed orbits in each $J^{-1}(j_{\mathrm{int}})$\end{itemize}
classified by Karshon\cite{karshon}}};

\node (S1-hypersemitoric) [bottom-process, below of=hypersemitoric,yshift=-2.5cm] {All Hamiltonian\\ $\mbS^1$-spaces\\[.5em]
\footnotesize{classified by Karshon\cite{karshon}}};

\draw [draw=none] (toric) -- node{$\subset$} (semitoric);
\draw [draw=none] (semitoric) -- node{$\subset$} (hypersemitoric);

\draw [draw=none] (S1-toric) -- node{$\subset$} (S1-semitoric);
\draw [draw=none] (S1-semitoric) -- node{$\subset$} (S1-hypersemitoric);

\begin{scope}[transform canvas={xshift=1em}]
\draw [arrow] (toric) -- node[anchor = west]{\cite{karshon}} (S1-toric);
\draw [arrow] (semitoric) -- node[anchor = west]{\cite{HSS}} (S1-semitoric);
\end{scope}

\begin{scope}[transform canvas={xshift=-1em}]
\draw [arrow] (S1-toric) -- node[anchor = east]{\cite{karshon}} (toric);
\draw [arrow] (S1-semitoric) -- node[anchor = east]{\cite{HSSS2}} (semitoric);
\draw [arrow] (S1-hypersemitoric) -- node[anchor = west]{Theorem~\ref{thm:extending}} (hypersemitoric);

\end{scope}

\end{tikzpicture}
\caption{Relationships between integrable systems and Hamiltonian $\mbS^1$-spaces:
$g$ refers to the genus of
all fixed surfaces which exist, and $j_{\mathrm{int}}$ refers to an element of the interior
of the image of $J$.
This diagram is commutative in the sense that
the inclusions and projections downwards are
compatible. 
The upwards arrows represent extending, and the 
downwards arrows represent using the classification of the
integrable system to recover the Karshon graph.
Since hypersemitoric systems are not yet symplectically 
classified, there is no downwards arrow from them in the diagram.}
\label{fig:diagram}
\end{figure}
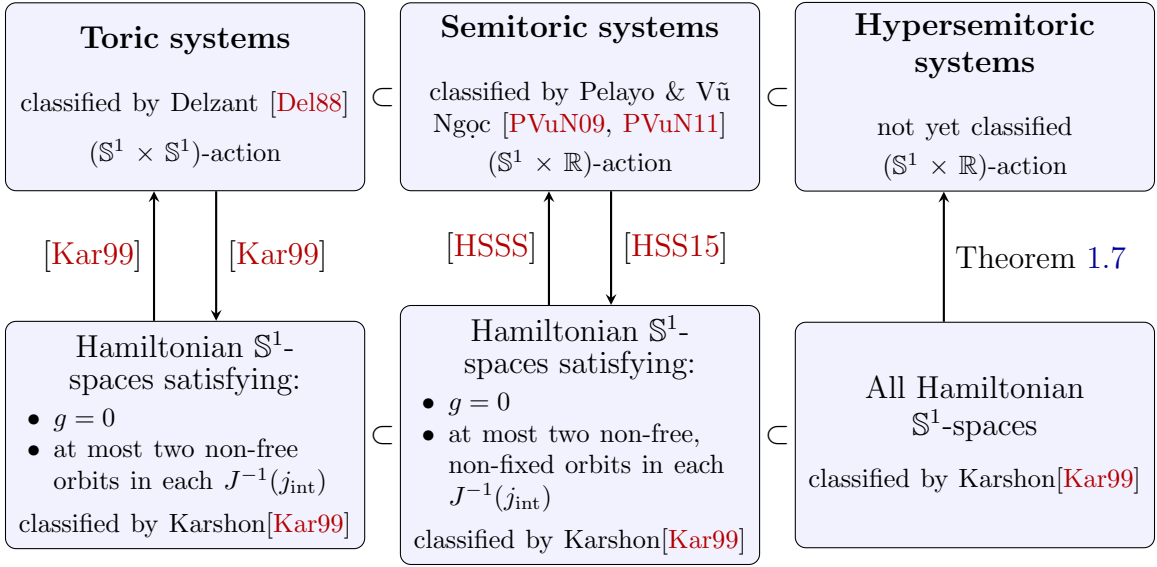

It is natural to consider extending the semitoric classification to hypersemitoric systems.
Pelayo \& V\~{u} Ng\d{o}c~\cite[Section 2.3]{PelVNsteps} discuss the expected difficulty in classifying
systems with hyperbolic singular points.
One of the main extra difficulties that make these systems more challenging, but also more interesting, is that the fibers of the momentum
map are often disconnected. 
The existence of disconnected fibers prevents the application of many standard
techniques, and therefore analyzing these systems is a non-trivial endeavor.
These extra difficulties are unavoidable though,
since hyperbolic singularities, and also disconnected fibers, are a common feature in natural
systems.
For instance, the Lagrange top and the two body problem are two fundamental physical systems
which include an $\mbS^1$-symmetry and also have singularities with hyperbolic components.

\begin{remark}
Following the appearance of the present paper on arXiv, which introduced hypersemitoric systems for the first time, there have been significant developments in the area.
    For details, see the recent survey by Palmer~\cite{joey-survey}; some significant recent milestones have been:
    \begin{itemize}
        \item Gullentops $\&$ Hohloch~\cite{GuHo} studied the appearance, disappearance, and number of hyperbolic-regular singularities by perturbing semitoric systems to hypersemitoric systems.
\item Gullentops~\cite{yannick-thesis} studied certain examples of hypersemitoric systems and the topology of hyperbolic-regular fibers.
        \item Palmer \& Le Floch~\cite{LFP-families2} studied the general theory of and constructed examples of one-parameter families of integrable systems that transition between being toric, semitoric, and hypersemitoric.
        \item A generalization of the semitoric polygon invariant to hypersemitoric systems was introduced and studied in Efstathiou $\&$ Hohloch $\&$ Santos~\cite{EHS-hyp-poly}.
        \item Sepe \& Tolman~\cite{SepeTolman} study the connectedness of fibers in integrable systems which have only non-degenerate singular points. The authors study these systems in all dimensions, but in dimension four such systems generalize semitoric systems, and are a special case of hypersemitoric systems.
    \end{itemize}
\end{remark}

\begin{remark}
In this paper, we extend Hamiltonian $\mbS^1$-actions to hypersemitoric systems, which we view as the most natural and well-behaved class of systems to which all such spaces can be extended. 
A more basic question would be if a Hamiltonian $\mbS^1$-action can be extended to an integrable system at all, with no regard as to how degenerate the system is. 
In fact \emph{any} effective Hamiltonian torus action on a compact manifold can be extended to an integrable system, see the discussion at the end of Section 5 of Bryant et al~\cite{BryantFoulonIvanovMatveevZiller} (which makes use of work of Fomenko~\cite[p 145]{Fomenko-symplgeom} and Bolsinov \& Jovanovi\'c~\cite{BolJov}).
\end{remark}

\begin{remark}
    A logical next step would be to generalize these extension results, such as Theorem~\ref{thm:extending}, to the case in which $M$ is non-compact but $J$ is still proper.
    One of the challenges of this task would be to deal with the potentially infinitely many features of the underlying Hamiltonian $\mbS^1$-space.
    It also remains to investigate if these results hold true in the analytic category or if there are any simplifications in the K\"ahler case.
\end{remark}



\begin{figure}
 \centering
 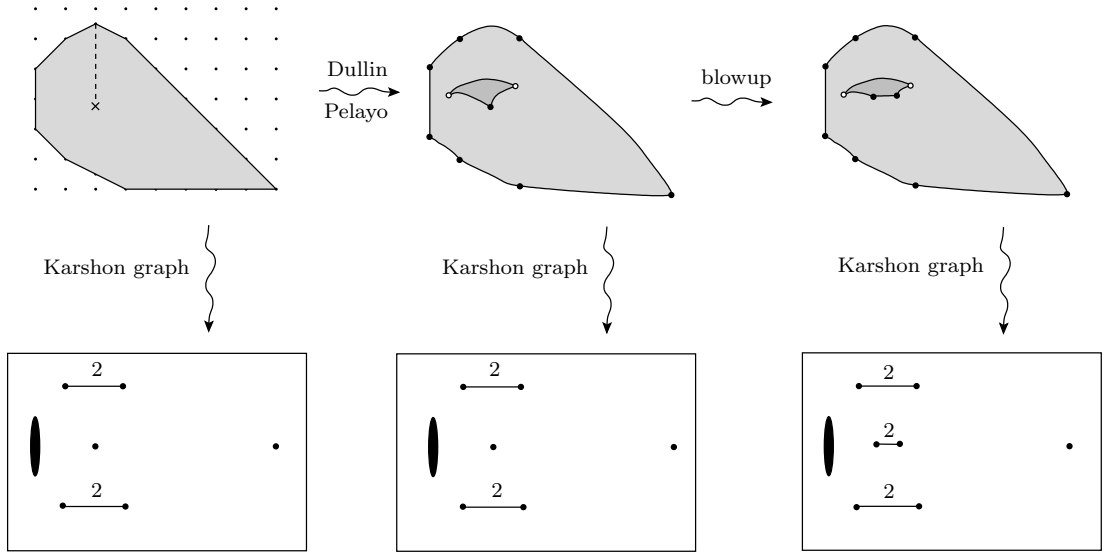
 \caption{Let $p$ be a fixed point in the Hamiltonian $\mbS^1$-space which corresponds to a focus-focus singular point in the extended integrable system. We first use the technique of Dullin \& Pelayo~\cite{dullin-pelayo} to induce a Hamiltonian-Hopf bifurcation which changes the system so that $p$ becomes an elliptic-elliptic point, and then we perform an $\mbS^1$-equivariant blowup at $p$.}
 \label{fig:trick}
\end{figure}

\subsection{Sketch of the proof of Theorem~\ref{thm:extending}}
\label{sec:sketch}

The proof of Theorem~\ref{thm:extending} makes use of results concerning minimal models in Karshon's classification~\cite{karshon}. 
Karshon proved that all Hamiltonian $\mbS^1$-spaces can be obtained from a list of certain \emph{minimal models} by a finite
sequence of $\mbS^1$-equivariant blowups.
We first show that all of the minimal 
models can be extended to hypersemitoric systems (Proposition~\ref{prop:minsystems}).
We then show that these hypersemitoric systems on the minimal models can be 
carried along
to construct a hypersemitoric system on any Hamiltonian $\mbS^1$-space which can be 
obtained from a minimal model via a sequence of
$\mbS^1$-equivariant blowups. 
Since all Hamiltonian $\mbS^1$-spaces can be produced in this way, we then
conclude that all Hamiltonian $\mbS^1$-spaces can be extended to hypersemitoric systems 
and the proof is complete.


The main difficulty is, roughly, that we must show that any hypersemitoric system on a given Hamiltonian $\mbS^1$-space induces a
hypersemitoric system on any $\mbS^1$-equivariant blowup of that Hamiltonian $\mbS^1$-space.
Such blowups occur at fixed points of the $\mbS^1$-action, which are necessarily singular points of the integrable system
which are of elliptic-elliptic, elliptic-regular, or focus-focus type.
If the blowup occurs at an elliptic-elliptic or elliptic-regular point of the integrable 
system then the argument follows similarly to that of Hohloch \& Sabatini \& Sepe \& Symington~\cite{HSSS2},
except that we cannot make use of a polygon invariant in our more general situation since so far such an invariant exists
only for semitoric systems.
Here we compensate for the lack of a polygon invariant by working locally around the fixed point.
For extending \emph{all} Hamiltonian $\mbS^1$-spaces, though, performing blowups only at elliptic-elliptic and elliptic-regular points is insufficient. 
We must also perform blowups at focus-focus points of the system, as explained in the following paragraph.

One of the main new ideas of the present paper is  performing an $\mbS^1$-equivariant blowup at a
focus-focus singular point $p\in M$. This is done by incorporating the technique described by Dullin \& Pelayo~\cite{dullin-pelayo} to
 use a supercritical Hamiltonian-Hopf bifurcation to replace
a neighborhood of the focus-focus singular value in $F(M)$ with a 
triangle of singular values known as a \emph{flap} (see Definition~\ref{def:flap}) while keeping the structure of the integrable
system and leaving the underlying $\mbS^1$-space unchanged. The flap includes two families of elliptic-regular points, one family of hyperbolic-regular
points, one elliptic-elliptic point, and two degenerate orbits of parabolic type. 
Once the Dullin \& Pelayo technique has been used to create the flap, the point $p$
is now an elliptic-elliptic singular point of the new system
and thus a usual toric-type blowup can be performed at $p$. 
This process is shown in Figure~\ref{fig:trick}.
One of the main technical difficulties of the paper is verifying that the flap
may be made large enough so that the entire blowup occurs at points contained `on the flap'.
Furthermore, we cannot proceed purely by induction on the number of blowups and make the flaps
one at a time, since each flap must be carefully positioned relative to all others to insure that none of the flaps created in this process interfere with each other.


\begin{remark}\label{rmk:generic}
 Hypersemitoric systems represent a very general class of systems with
 $\mbS^1$-symmetries. 
 To analyze this more deeply is beyond the scope of the present
 paper (we will consider this in a future project),
 but, given a Hamiltonian $\mbS^1$-space $(M,\om,J)$, it seems
 reasonable to conjecture that
 the set
 \[
  \{H\colon M\to \R \mid (M,\om,(J,H))\textrm{ is a hypersemitoric system}\}
 \] 
 is an open and dense subset of
 \[
  \{H\colon M\to \R \mid (M,\om,(J,H))\textrm{ is integrable}\}.
 \] 
 This seems plausible by the following sketched argument: such an integral $H$ can, roughly, be thought of as a one-parameter
 family of functions on the reduced space. By the work of Cerf~\cite{Cerf}, one parameter
 families of functions on smooth manifolds can be perturbed to be Morse at all but finitely many times, at which
 times the function takes a form which lifts to be a parabolic point.
 Lifting such a perturbed
 function would produce our candidate for a hypersemitoric system.
 One of the technical problems that would need to be overcome, is that
 in this situation the
 reduced space is in general not a smooth manifold but has a finite number
 of singular points.
\end{remark}


\subsection{Outline of paper:} In Section~\ref{sec:prelims} we recall various results we will need throughout the paper.
In Section~\ref{sec:examples} we give some motivating examples.
In Section~\ref{sec:properties} we prove some results about properties
of integrable systems for which one of the integrals generates an
$\mbS^1$-action.
In Section~\ref{sec:extending} we prove Theorem~\ref{thm:extending}.

\subsection{Acknowledgments:} We would like to thank in particular A.~Bolsinov and S.~V\~{u} Ng\d{o}c for helpful remarks
and references, and also various other colleagues for feedback on an earlier version of the paper. 
The first author was partially funded by the \emph{FWO-EoS project} G0H4518N,
and the second author was partially supported by the \emph{FWO senior postdoctoral fellowship} 12ZW320N.

\section{Preliminaries}
\label{sec:prelims}

In this section we briefly recall a number of results from the literature that we need. 
We will cover all of the relevant background partially to fix notation, and partially to provide an overview of the topics relevant for the present paper. These subjects range from relatively classical to extremely recent, and also include a few new results, and thus we believe a modern summary of this material has independent value.

Specifically, this section mainly summarizes the works
of Karshon~\cite{karshon}, Delzant~\cite{delzant}, Pelayo \& V\~u Ng\d{o}c~\cite{PVN09, PVN11},
Hohloch \& Sabatini \& Sepe~\cite{HSS}, Efstathiou \& Giacobbe~\cite{EG-cusps},
Bolsinov \& Guglielmi \& Kudryavtseva~\cite{Bol-parabolic}, and Kudryavtseva \& Martynchuk~\cite{KudMar-circle,KudMar-invariants}.


\subsection{Hamiltonian \texorpdfstring{$\mbS^1$}{S1}-spaces and their Karshon graphs}\label{sec:S1-classification}
Let $(M,\om,J)$ be a Hamiltonian $\mbS^1$-space (cf.\ Definition \ref{def:HamSspace}) and keep in mind that the underlying manifold is by definition compact and four-dimensional. 
Following~\cite{karshon}, we will
construct a labeled graph associated to this space.
\begin{lemma}[{\cite[Lemma 2.1]{karshon}}]\label{lem:karshon_fixedset}
Let $(M,\om,J)$ be a four dimensional compact Hamiltonian $\mbS^1$-space and denote by $M^{\mbS^1}$ the fixed point set of the $\mbS^1$-action.
Then $M^{\mbS^1}$ has finitely many components, each of which is either an isolated
point or a symplectic surface (i.e.~$\omega$ restricts to a symplectic form on the surface), and any such surface, if it exists, is exactly the preimage under $J$ of the maximum or minimum value of $J$.
Moreover, the preimages under $J$ of its maximum and minimum values are each connected.
\end{lemma}

For $k\in\Z_{>0}$ define $\Z_k := \{\lambda\in\mbS^1\mid k\lambda\in 2\pi\Z\}$. Connected components of the set of points on $M$ with isotropy subgroup $\Z_k$, $k>1$, are homeomorphic to cylinders, and the closure of each such component is an embedded sphere in $M$, called a \emph{$\Z_k$-sphere} (see Figure~\ref{fig:Zk-sphere}) on which the $\mbS^1$-action acts as rotation leaving the points in the closure (`poles') fixed.
Thus, $\Z_k$-spheres connect distinct components of $M^{\mbS^1}$.

\begin{figure}
 \centering
 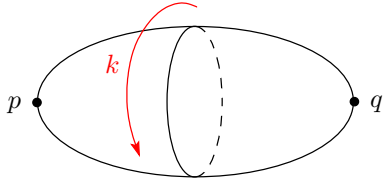 
 \caption{A $\Z_k$-sphere connecting the fixed points $p$ and $q$.}
 \label{fig:Zk-sphere}
\end{figure}

Now we will describe the graph used by Karshon to classify Hamiltonian $\mbS^1$-spaces. 
\begin{itemize}
 \item 
 The \emph{set of vertices} is the set of connected components of $M^{\mbS^1}$ each labeled by the value
of $J$ on that component. The fixed surfaces $\Sigma$ are represented by `fat vertices' (in figures drawn as large black ovals) that are labeled by the value of $J$ and, in addition, by the \emph{normalized symplectic area} of the surface 
\begin{equation}\label{eqn:norm-area}
 A(\Sigma):=\frac{1}{2\pi}\int_{\Sigma}\om 
\end{equation}
and its genus $g$. 
If the genus is 0 we often omit the label in figures, and we often omit the normalized area label as well.
\item
Two vertices are connected by an \emph{edge} if and only if there exists a $\Z_k$-sphere, $k>1$, connecting the two associated fixed points in $M$. Each edge is labeled by its $k$. The horizontal position of the vertices corresponds to their $J$-value (although we usually omit the coordinate axis with the values of $J$). 
\end{itemize}

Notice that our graphs are rotated by $90^\circ$ compared to Karshon's~\cite{karshon} since it is more natural for us to align the $J$-value direction in the graphs with the $J$-coordinate axis of the semitoric polygons.

Given any fixed point $p\in M^{\mbS^1}$, there exist integers $m,n\in\Z$
and complex coordinates $w,z$ around $p$ such that the $\mathbb S^1$-action by $t\in\mbS^1$ is given by $t\cdot (w,z)= (\mathrm{e}^{\mathrm{i} m t}w, \mathrm{e}^{\mathrm{i}n t}z)$
and the symplectic form is locally given by $\frac{\mathrm{i}}{2}(\mathrm{d}w\wedge\mathrm{d}\overline{w}+\mathrm{d}z\wedge\mathrm{d}\overline{z})$.
The integers $m$ and $n$ are called the \emph{weights} of the $\mbS^1$-action at $p$, and
they are also easy to see in the graph: for $k>1$, a fixed point has $-k$ as one
of its weights if and only if it forms the north pole of a $\Z_k$-sphere
and $k$ as one of its weights if and only if it forms the south pole of a $\Z_k$-sphere, where the north pole refers to the fixed point with higher $J$ value and the south pole the fixed point with the lower $J$-value.
The point $p$ has zero as one of its weights if and only if it lies in a fixed surface.
All weights not determined by these rules are either $+1$ or $-1$.
Furthermore,
if $p$ is in the preimage of the maximum value of $J$ then $p$ has two non-positive weights,
if $p$ is in the preimage of the minimum value of $J$ then it has two non-negative weights, and
otherwise $p$ has one positive and one negative weight.

\begin{example}\label{example:CP2_karshon}
Consider the usual action of $\mbS^1$ on $\CP^2$ given by $t\cdot [z_0:z_1:z_2] = [z_0:\mathrm{e}^{\mathrm{i} t}z_1:z_2]$ for $t\in\mbS^1$ with Hamiltonian $J([z_0:z_1:z_2]) = \abs{z_1}^2/(\abs{z_0}^2+\abs{z_1}^2+\abs{z_2}^2)$.
Then $J^{-1}(0) = \{z_1=0\}$ is a sphere which is fixed by the $\mbS^1$-action and $J^{-1}(1)= \{z_0=z_2=0\}$ is a point fixed by the $\mbS^1$-action. 
There are no other fixed points and the action is free away from these sets.
The fixed sphere is represented by a fat vertex at $J=0$ with normalized area $A=1$ (cf.~Equation~\eqref{eqn:norm-area}), the fixed point is represented by a regular vertex
at $J=1$, and there are no edges.
The graph is shown in Figure~\ref{fig:CP2_karshon}.
\end{example}

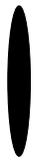
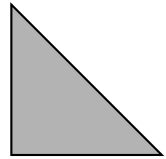
\begin{figure}
\centering
\begin{subfigure}[t]{.48\textwidth}
\centering
\begin{tikzpicture}
 \draw[black,fill=black] (0,0) ellipse (.15cm and 1cm);	
 \draw[black,fill=black] (2,0) circle (.1cm);
\end{tikzpicture}
\caption{The Karshon graph for $\C\P^2$.}
\label{fig:CP2_karshon}
\end{subfigure}\quad
\begin{subfigure}[t]{.48\textwidth}
\centering
\begin{tikzpicture}
\filldraw[thick, fill = gray!60] (0,0) node[anchor = north,color = black]{}
  -- (0,2) node[anchor = south,color = black]{}
  -- (2,0) node[anchor = north,color = black]{}
  -- cycle;
\end{tikzpicture}
\caption{The Delzant polygon of $\CP^2$.}
 \label{fig:CP2_delzant}
\end{subfigure}
\caption{The Karshon graph and Delzant polygon for the standard actions of $\mbS^1$ and $\T^2$ on $\CP^2$. Their geometric relation is described in Section~\ref{sec:S1-toric}.}
\label{fig:CP2_karshondelzant}
\end{figure}

An \emph{isomorphism} between two Hamiltonian $\mbS^1$-spaces $(M_1,\om_1,J_1)$ and $(M_2,\om_2,J_2)$ is
a symplectomorphism $\Psi\colon M_1\to M_2$ such that $J_2 \circ \Psi = J_1$, in which case $\Psi$ is also automatically
equivariant with respect to the $\mbS^1$-actions.
One of the main results of Karshon's work~\cite{karshon} is that the graphs contain all of the information of the isomorphism class
of the associated Hamiltonian $\mbS^1$-space.

\begin{theorem}[{\cite[Theorem 4.1]{karshon}}]
 Two four-dimensional compact Hamiltonian $\mbS^1$-spaces have the same Karshon graph if and only if
 they are isomorphic as Hamiltonian $\mbS^1$-spaces.
\end{theorem}

Karshon completes the classification by describing exactly 
which graphs occur as the invariant of an
$\mbS^1$-space. 
There are constraints on which graphs can occur, and Karshon specifies which graphs can occur by describing an algorithm that produces the set of all admissible graphs (by starting with an explicit list of minimal models and performing blowups).
We discuss this in Section~\ref{sec:min-models}.

\begin{remark}
\emph{Complexity-one spaces} are the higher dimensional analogue of Hamiltonian $\mbS^1$-spaces: they consist of a $2n$-dimensional symplectic manifold carrying a Hamiltonian action of the torus $\mathbb{T}^{n-1}$.
A complexity-one space is called \emph{tall} if all reduced spaces are two-dimensional.
Tall complexity-one spaces are classified by a series of papers by Karshon \& Tolman~\cite{KT-centered,KT-invariants,KT-classify},
which extends the above presented classification of Hamiltonian $\mbS^1$-spaces by including 
additional invariants.
Extending such Hamiltonian torus actions to integrable systems is unfortunately beyond the scope of the present paper and thus will be treated in a future work.
First steps in this direction have been made by Wacheux~\cite{Wacheux-thesis} who studied six dimensional integrable systems where two components of the momentum map are both periodic, and thus generate a Hamiltonian $\mathbb{T}^2$-action.
\end{remark}

\subsection{Classification of toric systems}\label{sec:toric_classification}
Let $(M,\om,F=(f_1,\ldots, f_n))$ be a $2n$-dimensional 
toric system, and recall from Section~\ref{sec:intro} that, in our convention, the underlying manifold $M$ of a  toric system is always compact. 
The flows of $f_1,\ldots, f_n$ are each periodic
of minimal period $2\pi$, and thus induce an effective action of $\T^n = \R^n/2\pi\Z^n$ so that one can interpret it as a Hamiltonian action of $\T^n$ on $M$.
Atiyah~\cite{Atiyah} and Guillemin \& Sternberg~\cite{GuilStern} showed that the image $F(M)$ is a convex $n$-dimensional polytope
and, moreover, that it is the convex hull of the images of the fixed points of the torus action on $M$.
Furthermore, Delzant~\cite{delzant} showed that the polytope $\De:=F(M)$ always satisfies three conditions:
\begin{enumerate}
 \item \emph{simplicity}: exactly $n$ edges meet at each vertex of $\De$ (note that this is automatic if $n=2$, the case we will consider in this paper);
 \item\label{item:rationality} \emph{rationality}: each face of $\De$ admits an integral normal vector (i.e. a normal vector in $\Z^n$);
 \item \emph{smoothness}:  given any vertex, the set of integral inwards pointing normal vectors (from item~\eqref{item:rationality}) of the faces adjacent
  to that vertex can be chosen such that their $\Z$-space is the entire
  lattice $\Z^n$.
\end{enumerate}
Delzant also showed that any $n$-dimensional polytope satisfying these conditions arises as the image of the momentum map for some toric system, and that two toric systems $(M_1,\om_1,F_1)$ and $(M_2,\om_2,F_2)$
have the same momentum map image if and only if there exists a symplectomorphism $\Phi\colon M_1\to M_2$ such that $F_2 \circ \Phi = F_1$,
called an \emph{isomorphism of toric systems}.
Thus, Delzant completed the classification of toric systems up to isomorphism in terms of a convex polytope given by the image of the momentum map. We will call this polytope the \emph{Delzant polytope} of the system, or, in case $n=2$, the \emph{Delzant polygon}.

\begin{example}\label{example:CP2_toric}
 Consider the toric system $(\CP^2,\om_{\mathrm{FS}},F=(J,H))$ where $\om_{\mathrm{FS}}$ is the usual Fubini-Study symplectic form, $J$ is as in Example~\ref{example:CP2_karshon},
 and $H([z_0,z_1,z_2]) =  \abs{z_2}^2/(\abs{z_0}^2+\abs{z_1}^2+\abs{z_2}^2)$. Then the associated Delzant polygon is the triangle with vertices at $(0,0)$, $(0,1)$, and $(1,0)$ drawn in Figure~\ref{fig:CP2_delzant}.
\end{example}

\subsection{Hamiltonian \texorpdfstring{$\mbS^1$}{S1}-spaces and toric systems}
\label{sec:S1-toric}

Let $(M,\om,F=(J,H))$ be a compact toric system with Delzant polygon $\De=F(M)$.
Then $(M,\om,J)$ is a Hamiltonian $\mbS^1$-space with
$\mbS^1$-action associated with the subgroup $\mbS^1\times\{0\}\subset\T^2$ of the torus acting on $M$.
The fixed surfaces (if any) of this action are the preimages of the (closed) vertical edges of $\De$ which
have normalized symplectic area (see Equation~\eqref{eqn:norm-area})
equal to the length of the edge and have always genus zero. The isolated
fixed points of the action are the vertices of $\De$ which are not on vertical edges.
Each $\Z_k$-sphere with $k \in \Z_{>1}$ associated with the action induced by $J$ is the preimage of an edge (including the limiting vertices) of $\De$ of which the slope can be written as $b/k$ for some $b\in\Z$ such that $k$ and $b$ are relatively prime.
Thus, it is straightforward to construct the Karshon graph from
the Delzant polygon, compare the Delzant polygon to the Karshon graph for the standard action on $\CP^2$ in Figure~\ref{fig:CP2_karshondelzant}.
Note that there are Hamiltonian $\mbS^1$-spaces that cannot be obtained from
a Hamiltonian $\T^2$-action in this way, as we will see now.
We call a level set \emph{non-extremal} if it is the preimage of any point in the image of $J$
except for its maximum or minimum values.
\begin{lemma}[{\cite[Proposition 5.21]{karshon}}]\label{lem:toric-extending}
 A Hamiltonian $\mbS^1$-space can be extended to a toric system if
 and only if each fixed surface (if any) of the Hamiltonian $\mbS^1$-space has genus zero
 and each non-extremal level set of $J$ contains at most
 two non-free orbits of the $\mbS^1$-action.
\end{lemma}

Actions without fixed surfaces can be extended:

\begin{lemma}[{\cite[Corollary 5.19]{karshon}}]\label{lem:isolated-points-extend}
 Let $(M,\om,J)$ be a Hamiltonian $\mbS^1$-space. If all fixed points of the $\mbS^1$-action are isolated
 then $(M,\om,J)$ extends to a toric system.
\end{lemma}

\subsection{Equivariant blowups}
\label{sec:S1-blowups}

Not every possible labeled graph can actually be obtained as the Karshon graph
associated to a Hamiltonian $\mbS^1$-space. Following Karshon~\cite{karshon}, we will describe, in terms of minimal models, the set of labeled graphs which do correspond to a Hamiltonian $\mbS^1$-space.
This involves in particular the effect of $\mbS^1$-equivariant blowups and blowdowns on the labeled graph and the notion of \emph{minimal Hamiltonian $\mbS^1$-spaces}, i.e., Hamiltonian $\mbS^1$-spaces which do not admit such blowdowns, and their Karshon graphs. Essential for the present paper is the fact that the set of graphs that can be obtained from Hamiltonian $\mbS^1$-spaces is equal to the set of graphs which can be produced from one of these minimal graphs via a finite sequence of $\mbS^1$-equivariant blowups.

\subsubsection{Symplectic and equivariant blowups and blowdowns}

Let $(M,\om)$ be a symplectic 4-manifold. Intuitively, a `blowup' of $(M,\om)$ amounts to removing an embedded $4$-ball and collapsing the boundary via the Hopf fibration to a 2-sphere. Since a `blowdown' is the inverse operation of a blowup we recall in the following only the construction and definition of (various types of) blowups but not blowdowns.

\begin{definition}\label{def:blowup}
Let $p\in M$, let $U\subset M$ be a neighborhood of $p$, and let $\phi\colon U\to V\subset\C^2$ be a symplectomorphism with $\phi(p)=(0,0)$. 
Then, given any $r>0$ such that the standard ball of radius $r$ in $\C^2$ is contained in $V$, we define the \emph{symplectic blowup at $p$ of size $\lambda := \frac{r^2}{2}$} by removing the preimage of the ball and collapsing the boundary to a 2-sphere via the usual Hopf fibration given by $\mbS^3\to \CP^1$, $(z_0,z_1)\mapsto [z_0:z_1]$ where the coordinates $(z_0,z_1)$ come from the inclusion $\mbS^3\subset\C^2$.
We will denote the manifold obtained by blowing up $M$ at $k$ points by $\mathrm{Bl}^k(M)$.
\end{definition}

The blowup and the blowdown and how to equip the resulting manifold with a symplectic form are described in detail in McDuff $\&$ Salamon~\cite[Section 7.1]{McDuffSalamon}.
Note that the diffeomorphism type of $\mathrm{Bl}^k(M)$ is independent of the choice
of points and sizes of the blowups.

Suppose that $(M,\om, J)$ is a
Hamiltonian $\mbS^1$-space. 
To obtain \emph{$\mbS^1$-equivariant coordinates}
we require that the map $\phi$ from Definition~\ref{def:blowup} is equivariant with respect to the given Hamiltonian $\mbS^1$-action on $(M,\om)$
and a Hamiltonian $\mbS^1$-action
on $\C^2$
given by $\lambda\cdot (z,w) = (\lambda^m z, \lambda^{n} w)$ for some $m,n\in\Z$, where $\lambda\in\mbS^1$.
Notice that such a $\phi$ only 
exists if $p$ is a fixed point of the action. 


\begin{definition}
Let $(M, \om, J)$ be a Hamiltonian $\mbS^1$-space and $p \in M$. 
Taking the blowup with respect to the above mentioned $\mbS^1$-equivariant coordinates and restricting the momentum map to the resulting space yields the \emph{$\mbS^1$-equivariant blowup at $p \in M$ of size $\lambda$}.
\end{definition}

Now consider a toric system and its induced Hamiltonian $\T^2$-action. 
To obtain \emph{$\T^2$-equivariant coordinates}
we require that the map $\phi$ from Definition~\ref{def:blowup} is equivariant with respect to the given Hamiltonian $\T^2$-action on $(M,\om)$ and a standard linear Hamiltonian $\T^2$-action
on $\C^2$ given by $(s,t)\cdot (z,w) = (s^m t^{m'}z, \, s^n t^{n'} w)$ for $(s,t)\in\T^2$ with weights
$m,m',n,n'\in\Z$.


\begin{definition}
Let $(M,\om,F)$ be a toric system and $p\in M$.
Taking the blowup with respect to the above mentioned
$\T^2$-equivariant coordinates and restricting the momentum map to the resulting space yields the \emph{$\T^2$-equivariant blowup at $p \in M$ of size $\lambda$}.
\end{definition}

The isomorphism class of the result of the $\mbS^1$-equivariant and $\T^2$-equivariant blowups described
above is independent of all choices. This can
be seen by simply noting that 
the resulting Karshon graph (for a Hamiltonian $\mbS^1$-space) or Delzant polygon (for a toric system)
is independent of all choices, and using the fact
that these objects classify 
Hamiltonian $\mbS^1$-spaces and toric systems, respectively,
up to isomorphisms (this is the same argument from~\cite[Proposition 6.1]{karshon}).

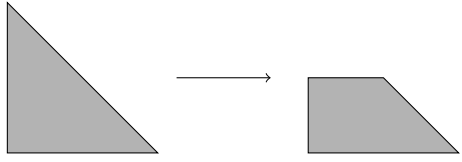
\begin{figure}
\begin{center}
\begin{tikzpicture}

\draw [->] (2.25,1) -- (3.5,1);

\filldraw[draw=black, fill=gray!60] (0,0) node[anchor=north,color=black]{}
  -- (0,2) node[anchor=south,color=black]{}
  -- (2,0) node[anchor=north,color=black]{}
  -- cycle;
  
\filldraw[draw=black, fill=gray!60] (4,0) node[anchor=north,color=black]{}
  -- (4,1) node[anchor=south,color=black]{}
  -- (5,1) node[anchor=south,color=black]{}
  -- (6,0) node[anchor=north,color=black]{}
  -- cycle;  
\end{tikzpicture}
\end{center}
\caption{Performing a blowup at one of the fixed points of $\CP^2$ corresponds to
performing a corner chop on the associated Delzant polygon.}
\label{fig:CP2blowup-delzant}
\end{figure}

\subsubsection{Blowups on 2-dimensional Delzant polytopes}
Now we analyze the impact of a $\T^2$-equivariant blowup of a toric system on its Delzant polygon.

A vector $v\in\Z^2$ is \emph{primitive} if $v = k u$ for $u\in\Z^2$ and $k\in\Z$ implies $k=\pm 1$, i.e., $v$ is not a multiple of a shorter integral vector.
Given primitive vectors $u_1,u_2\in\Z^2$ and $x\in\R^2$ set
\[
 \mathrm{Simp}_x^\lambda(u_1,u_2) := \{x + t_1 u_1 + t_2 u_2 \mid t_1, t_2 >0,\ t_1+t_2<\lambda\}.
\]
which parametrizes a simplex originating from $x$ between $\lambda u_1$ and $\lambda u_2$.
Thus, taking $x$ to be the origin of $\R^2$ and $u_1,u_2$ to be the standard basis vectors of $\Z^2$ yields a right triangle with two sides of length $\lambda$. Other choices of $x$ and primitive vectors $u_1,u_2\in\Z^2$ lead to all translations of images of that triangle under $\mathrm{GL}(2,\Z)$.

\begin{remark}
Let $(M,\om,F)$ be a compact, four dimensional, toric integrable system with Delzant polygon $F(M)=: \De$. Let $p \in M$ be a fixed point of the system and $\lambda>0$ such that the blowup at $p$ of size $\lambda$ is well-defined.  
Let $(M',\om',F')$ be the toric system resulting from a $\T^2$-equivariant blowup at $p\in M$ of size $\lambda>0$. Then the Delzant polygon $F'(M') =: \De'$ is related to $\De$ via
\[
 \De' = \De \setminus \mathrm{Simp}_{F(p)}^\lambda(v_1,v_2),
\]
cf.\ Figure~\ref{fig:CP2blowup-delzant}, where $v_1,v_2\in\Z^2$ are the primitive vectors aligned with the edges emanating from the corner $F(p)$.
This operation is usually referred to as a \emph{corner chop} and is well known in the literature, see for instance Cannas da Silva~\cite[Sections 3.4-3.5]{Can_toric}.

\end{remark}

Now we can reformulate the idea of the blowup entirely in terms of the Delzant polygon.

\begin{remark}
The toric system $(M,\om,F)$ associated to a Delzant polygon $\De$ admits a $\T^2$-equivariant blowup
of size $\lambda>0$ at a fixed point $p\in M$ if and only if $F(p)$ is the only vertex of $\De$ contained in the simplex $\mathrm{Simp}_{F(p)}^\lambda(v_1,v_2)$, where $v_1,v_2$ are the primitive vectors aligned with the edges emanating from $F(p)$.
\end{remark}

\subsubsection{Blowups on Karshon graphs} 
\label{sec:karshon-blowups}

Let $(M,\om,J)$ be a Hamiltonian $\mbS^1$-space on which we perform an $\mbS^1$-equivariant blowup at a fixed point $p \in M$ of (admissible) size $\lambda$. Denote the resulting Hamiltonian $\mbS^1$-space by $(M', \om', J')$ and flag all objects associated with $(M', \om', J')$ by a prime $'$.
Recall that a fixed point of the $\mbS^1$-action is either an isolated fixed point or lies in a fixed surface (which can only be located at the minimum and/or maximum of $J$). Denote by $j_\textrm{min},j_{\textrm{max}}\in\R$ the minimum and maximum values of $J(M)$.

We will now describe how the Karshon graph $\Gamma'$ of $(M', \om', J')$ relates to the Karshon graph $\Gamma$ of $(M, \om, J)$,
see Figure~\ref{fig:blowupCasesB}.


\begin{figure}
 \centering
 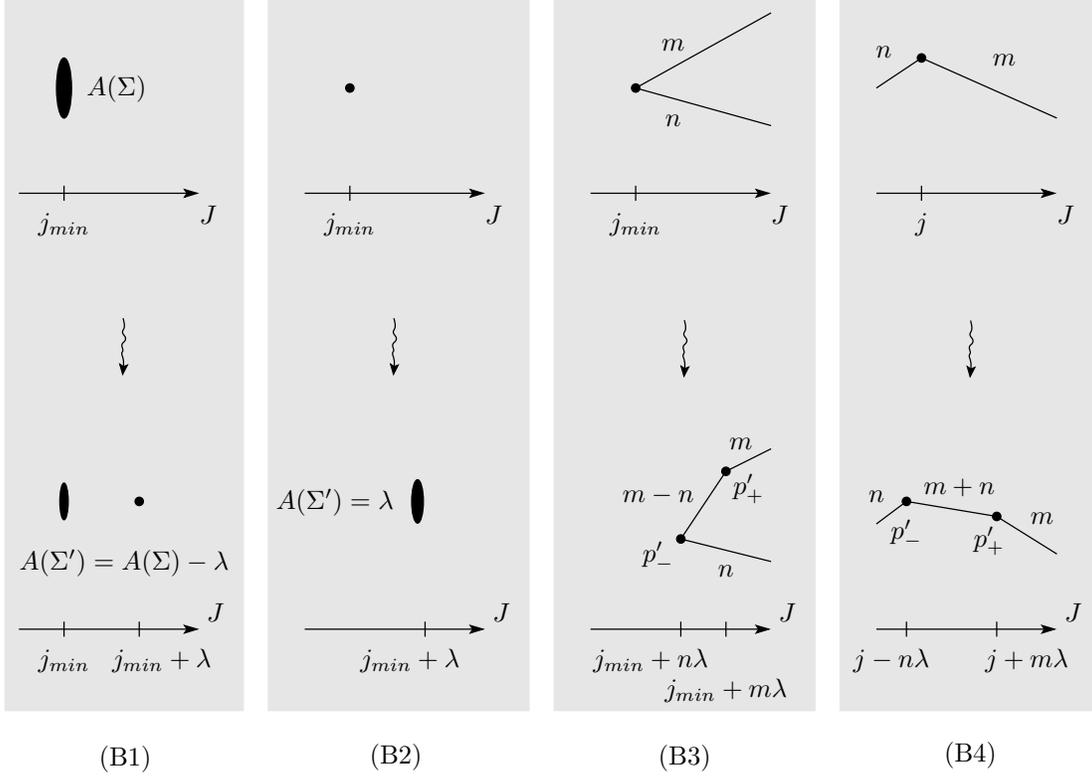
  \caption{The effect of a blowup on a Karshon graph in cases~\ref{case:blowup-surface}-\ref{case:interior-blowup} from Lemma~\ref{lem:casesGraphChange}.}
\label{fig:blowupCasesB}
\end{figure}

\begin{lemma}[{\cite[Section 6.1]{karshon}}]
\label{lem:casesGraphChange}

The effect of an $\mbS^1$-equivariant blowup at a fixed point $p\in M^{\mbS^1}$ of (admissible) size $\lambda$ on the
associated Karshon graph depends on the type of fixed point (isolated or not) and the value of $J(p)$:\\

\begin{enumerate}[label=(B\arabic*)]
\item \label{case:blowup-surface} 
If $\Sigma:=J^{-1}(j_\textrm{min})$ is a fixed surface and $p \in \Sigma$, then the normalized symplectic area $A(\Sigma)$
of $\Sigma$ is reduced by $\lambda$, i.e., $A(\Sigma')= A(\Sigma) - \lambda$. Moreover, there is a new isolated fixed point $p'$ in $\Gamma'$ with $J'(p')= J(p) + \lambda = j_{\mathrm{min}}+\lambda$. Other than this the graph is unchanged, and in particular there are no new edges. 
 
Replacing $j_\textrm{min}$ by $j_{\mathrm{max}}$ works similar, except that $J'(p') = J(p)- \lambda = j_{\mathrm{max}}-\lambda$.\\

\item \label{case:blowup-min11} 
If $p$ is an isolated fixed point with $J(p)=j_{\mathrm{min}}$ and weights both equal to $1$, then the vertex in $\Gamma$ corresponding to $p$ does not exist any more in $\Gamma'$, but is replaced in $\Gamma'$ by a fat vertex of normalized area $A=\lambda$ with genus zero and $J$-value $J=j_{\mathrm{min}}+\lambda$. There is no change in the edge set of the graph.

The case in which $p$ is an isolated fixed point with $J(p)=j_{\mathrm{max}}$ and weights both equal to $-1$ is similar, except that the new fat vertex has $J$-value $J=j_{\mathrm{max}}-\lambda$.\\

\item \label{case:blowup-min} 
If $p$ is an isolated fixed point with $J(p)=j_{\mathrm{min}}$ with weights $n$, $m$ satisfying $0<n<m$, then
 the vertex associated to $p$ in $\Gamma$ is removed and replaced in $\Gamma'$ by two new vertices associated with two fixed points $p_-'$ and $p_+'$ with $J'(p_-') = j_{\mathrm{min}}+n\lambda$ and $J'(p_+') = j_{\mathrm{min}}+m\lambda$. The edge with label $n$
 that was attached to the vertex corresponding to $p$ is now attached to the new vertex corresponding to $p_-'$ and the edge labeled by $m$ is now attached to the new vertex corresponding to $p_+'$. The two new vertices are connected to each other by an edge labeled by $m-n$. Recall that edges with labels equal to $1$ are not drawn in the graph.
 
 The case with $J(p) = j_{\mathrm{max}}$ is a reflection of this situation. That is, if the weights of $p$ are $-n$ and
 $-m$ satisfying $0<n<m$ then the new vertices have $J$-values
 $J'(p_-')=j_\mathrm{max}-n\lambda$ and
 $J'(p_+')=j_\mathrm{max}-m\lambda$, and the edge connecting
 them has label $m-n$.\\

\item \label{case:interior-blowup} 
If $p$ is an isolated fixed point with $J(p) \in\ ]j_{\mathrm{min}}, j_{\mathrm{max}}[ $ with weights $-n$ and $ m$ satisfying $n,m > 0$ then the vertex in $\Gamma$ corresponding to $p$ is removed
 from the graph and replaced in $\Gamma'$ with two new vertices corresponding to two isolated fixed points $p_-'$ and $p_+'$ satisfying $J'(p_-')=J(p)-n\lambda$ and $J'(p_+')=J(p)+m\lambda  $, so that $J'(p_-') \leq J'(p_+')$.
 The edge with label $n$ that was attached to the vertex corresponding to $p$ is now attached to the new vertex associated with $p_-'$. The edge labeled $m$ is now attached to the new vertex corresponding to $p_+'$. The two new vertices are connected to each other by an edge with label $m+n$.

\end{enumerate}
\end{lemma}

Above we have described the effect of the blowups on the graphs, but if $\lambda$ is too large, 
an $\mbS^1$-equivariant blowup of size $\lambda$ on $(M,\om,J)$ will not be possible. To describe which
blowups on the graphs can actually be realized by blowups on the $\mbS^1$-space (i.e.~to describe which
$\mbS^1$-equivariant blowups are possible on a given $\mbS^1$-space), we must 
first define a partial ordering of the vertices of a Karshon graph (again following
Karshon~\cite{karshon}).
If $v$ and $w$ are vertices of a Karshon graph, then we say that $v<w$ if and only if the $J$-value of $v$ 
is less than the $J$-value of $w$ and one of the following holds:
\begin{itemize}
 \item $v$ or $w$ are extremal, or
 \item $v$ and $w$ are connected by a chain of edges.
\end{itemize}
This partial ordering on the vertices of $\Gamma$ also induces a partial ordering on the vertices of any blowup of $\Gamma$ in the following way:
let $v$ be a vertex of a Karshon graph $\Gamma$ and let $\Gamma_{v,\lambda}$ be a graph obtained by performing a blowup at $v$
of some size $\lambda>0$. 
There is a value $\lambda_0>0$ such that
for all $0<\lambda<\lambda_0$ the partial ordering described above on the vertices of $\Gamma_{v,\lambda}$
is not changed by the size of $\lambda$.
We endow the vertices of $\Gamma_{v,\lambda}$ with this partial ordering,
no matter the size of $\lambda$. 

The following proposition essentially states that the sizes of blowups which are allowed are those for which $\lambda$ is sufficiently small so that the 
ordering on the vertices described above is in agreement with ordering of their $J$-values:
\begin{proposition}[{\cite[Proposition 7.2]{karshon}}]
 An $\mbS^1$-space admits an $\mbS^1$-equivariant blowup of size $\lambda>0$ at a given fixed point if
 and only if the blowup of size $\lambda$ of the corresponding Karshon graph at the corresponding vertex satisfies:
 \begin{enumerate}
  \item for all vertices $v,w$ in the blownup graph, if $v<w$ then $J(v)<J(w)$;
  \item the area label of each fat vertex (if any) is positive.
 \end{enumerate}
\end{proposition}


\begin{remark} 
For the goal of the present paper, case~\ref{case:interior-blowup} of Lemma~\ref{lem:casesGraphChange} is of interest for $m=n=1$, i.e., the situation where a single vertex $p$ with $J(p)=:j$ which is not the end or start point of any edge is replaced by a pair of vertices $p_-'$ and $p_+'$ with $J'(p_-')=j-\lambda$ and $J'(p_+')=j+\lambda$ that are connected by an edge with label $2$.
\end{remark}

\subsubsection{Toric blowups} 
If a fixed point of an integrable system consists only of elliptic components in the sense of Theorem \ref{thm:normal-form} it is said to be a \emph{completely elliptic point}.  Given a completely elliptic fixed point $p$ of an $2n$-dimensional integrable system, the coordinates of the local normal form in Theorem~\ref{thm:normal-form} yield an identification of a neighborhood of $p$ with $\C^n$. This allows to define a blowup at completely elliptic points, referred to as \emph{toric blowups}. Details can be found for instance in Le Floch \& Palmer~\cite[Section 4.2]{LFPfamilies}, in particular for the fact that the resulting system is independent of all choices.

\subsection{Minimal Hamiltonian \texorpdfstring{$\mbS^1$}{S1}-spaces}
\label{sec:min-models}

The exposition of this section closely follows the line of thoughts 
in Karshon~\cite[Section 6]{karshon}.

\begin{definition}
A Hamiltonian $\mbS^1$-space is \emph{minimal} if it does not admit an $\mbS^1$-equivariant blowdown.
\end{definition}

To formulate the list of all minimal Hamiltonian $\mbS^1$-spaces, we need the following definition:

\begin{definition}
\label{def:ruledMfd}
A {\em ruled surface} $(M, \om, J)$ is an $\mbS^2$-bundle $M$ over a closed surface $\Sigma$ endowed with a momentum map $J: M \to \R$ inducing an $\mbS^1$-action that fixes the base, and therefore rotates each fiber.
\end{definition}

Note that this definition of ruled surfaces may differ from the one used in a non-equivariant or complex setting. 
Since the $\mbS^1$-action is by rotation of the sphere component, there exist local
trivializations in which the Hamiltonian function $J$ is given by (an appropriate multiple of) the standard height function on the sphere
component. The $\mbS^1$-action on ruled surfaces fixes two surfaces, both diffeomorphic to the base $\Sigma$, corresponding to the north and south poles of the spheres which are fixed by the $\mbS^1$-action. There are no
other fixed points, and thus the Karshon graph of a ruled surface is composed of two fat vertices. The $J$-values of the fat vertices are $j_0$ and $j_0+s$, for some $s>0$, the area labels are $a>0$ and $a+ns$ (for some $n\in\Z$ such that $a+ns>0$),
and both are labeled with the same genus $g\in\Z_{\geq 0}$.
In the case that $n=0$ the ruled surface is simply the
product $\mbS^2\times\Sigma_g$, where $\Sigma_g$ is a surface of genus $g$.

Now recall that, given a toric system $(M,\om,F)$, any homomorphism $\mbS^1\hookrightarrow\T^2$ induces a Hamiltonian action of $\mbS^1$ on $(M,\om)$.

\begin{theorem}[{Karshon~\cite[Theorem 6.3]{karshon}}] \label{thm:karshon-minimal}
 A Hamiltonian $\mbS^1$-space $(M,\om,J)$ is minimal if and only if either:
 \begin{enumerate}
  \item $(M,\om,J)$ is induced by the standard toric system on $\CP^2$ with some multiple of the Fubini-Study form
   by a homomorphism $\mbS^1\hookrightarrow\T^2$;
  \item $(M,\om,J)$ is induced by the standard toric system on one of the scaled Hirzebruch surfaces by a homomorphism $\mbS^1\hookrightarrow\T^2$;
  \item \label{case:ruledmfd} $(M,\om,J)$ has two fixed surfaces and no other fixed points, in which case it is a ruled surface in the sense of Definition \ref{def:ruledMfd}.
 \end{enumerate}
\end{theorem}

\begin{definition}
 The minimal Karshon graphs are those that can be obtained from minimal Hamiltonian $\mbS^1$-spaces.
\end{definition}

Thus, Karshon's classification of Hamiltonian $\mbS^1$-spaces works in terms of `minimal models', i.e., the set of all possible Karshon graphs
is obtained by starting with the minimal Karshon graphs and also including in all graphs that can be obtained from those by a finite sequence of admissible $\mbS^1$-equivariant blowups.
Constructing the Karshon graph of a Hamiltonian $\mbS^1$-space
then defines a bijection between the set of all
Hamiltonian $\mbS^1$-spaces up to isomorphism
and the set of all Karshon graphs generated from the minimal models in this way.


\subsection{Singularities of integrable systems}
\label{sec:singularities}

Let $(M,\om,F=(f_1,\ldots,f_n))$ be a $2n$-dimen{-}sional integrable system. 
A point $p\in M$ is \emph{singular} if
\[
 \mathrm{rank}(p):= \dim( \mathrm{span}_\R \{ \mathcal{X}^{f_1}(p),\ldots,\mathcal{X}^{f_n}(p)\}) = \mathrm{rank}(dF(p))< n.
\]
Referring to Bolsinov $\&$ Fomenko~\cite{BolFom} for details, the space of quadratic forms on $T_p M$ can be endowed with a Lie algebra structure isomorphic to $\mathfrak{sp}(2n,\R)$. A rank zero singular point $p$ is \emph{non-degenerate} if and only if the Hessians of $f_1,\ldots, f_n$ span a so-called {\em Cartan subalgebra}. There is also a notion of non-degeneracy for singular points $p$ with $\mathrm{rank}(p)>0$: Here the idea is, loosely speaking, to take the symplectic quotient with respect to the non-singular part of the momentum map and apply the same condition.

Cartan subalgebras of $T_p M \cong \mathfrak{sp}(2n,\R)$ were classified by Williamson \cite{Wil}. This \emph{pointwise} classification was extended to a \emph{local} classification several special cases by a series of papers, such as Colin de Verdi\`ere-Vey~\cite{colinDeVerdiere-vey}, R\"ussmann~\cite{ruessmann}, Vey~\cite{vey}, Eliasson~\cite{Eliasson-thesis, eliasson-elliptic}, Dufour \& Molino~\cite{dufour-molino}, Miranda \& \vungoc~\cite{miranda-vungoc}, V\~{u} Ng\d{o}c \& Wacheux~\cite{vungoc-wacheux}, and Chaperon~\cite{chaperon}.
Furthermore, this result was extended to an equivariant version by Miranda \& Zung~\cite{miranda-zung}.
See Remark 4.6 and the discussion after Theorem 4.5 in Palmer~\cite{joey-survey} for details of proofs of the following statement in various special cases\footnote{To our knowledge, no complete proof of Theorem~\ref{thm:normal-form} exists in the literature.}.

\begin{theorem}[Local normal form for non-degenerate singularities]
\label{thm:normal-form}
 Let $p\in M$ be a non-degenerate singular point of 
 a $2n$-dimensional completely integrable system $(M,\om,F=(f_1,\ldots,f_n))$.
 Then:
 \begin{enumerate}
  \item there exist local symplectic coordinates $(x_1,\ldots, x_n, \xi_1,\ldots,\xi_n)$
 on an open neighborhood $U\subset M$ and smooth functions
 $q_1,\ldots,q_n\colon U\to\R$ where we have the following possibilities
 for the form of each $q_j$:
 \begin{itemize}
  \item Elliptic component: $q_j = (x_j^2+\xi_j^2)/2$,
  \item Hyperbolic component: $q_j = x_j \xi_j$,
  \item Focus-focus component: $q_j = x_j \xi_{j+1} - x_{j+1} \xi_j$ and $q_{j+1} = x_j \xi_j + x_{j+1} \xi_{j+1}$,
  \item Regular component: $q_j = \xi_j$,
 \end{itemize}
 such that $\{q_i,f_j\}=0$ for all $i,j\in\{0,\ldots,n\}$ and $p$ corresponds to the
 origin in these coordinates;
  \item\label{item:g} if there is no hyperbolic component then the system of equations $\{q_i,f_j\}=0$
 for all possible $i,j$ is equivalent to the existence of a local diffeomorphism
 $g\colon \R^n\to\R^n$ such that
 \[
  g \circ F = (q_1,\ldots, q_n) \circ (x_1,\ldots, x_n, \xi_1,\ldots, \xi_n).
 \]
 \end{enumerate}
\end{theorem}

The number of elliptic, hyperbolic, focus-focus, and regular components locally classifies a non-degenerate singular point and is referred to as its \emph{Williamson type}.
For instance, if $\mathrm{dim}(M)=4$, each non-degenerate singular point is of exactly one of the following six types:
\begin{itemize}
 \item 
 rank 0: 
 elliptic-elliptic,
 focus-focus,
 hyperbolic-hyperbolic,
 hyperbolic-elliptic;
 \item
 rank 1: 
 elliptic-regular,
 hyperbolic-regular.
\end{itemize}

%
%
%

\begin{remark}\label{rmk:morse}
 Suppose that $M = \Sigma_1\times\Sigma_2$ where $\Sigma_1$
 and $\Sigma_2$ are surfaces. For $i\in\{1,2\}$ 
 let $\pi_i\colon M\to \Sigma_i$ be the projection map,
 let $\om_i$ be a symplectic form on $\Sigma_i$, and let
 $f_i\colon \Sigma_i\to\R$ be a Morse function.
 Then $(M,\om_1\oplus\om_2,F=(f_1\circ\pi_1,f_2\circ\pi_2))$ is
 an integrable system and all singular points of this system are non-degenerate.
 Moreover, $p=(p_1,p_2)$ is a singular point of $F$ if and only if at least one of the $p_i$
 is a critical point of the corresponding $f_i$, and the Williamson
 type of $p$ is determined by the Morse indices of $p_1$ and $p_2$.
 This is because given Morse charts $U_i\subset \Sigma_i$ around $p_i$ with coordinates putting
 each $f_i$ into the standard form (as in the Morse lemma) the
 product $U_1\times U_2\subset M$ forms one of the charts
 discussed in Theorem~\ref{thm:normal-form}.
 Regular points of one of the Morse functions correspond to regular components in the local normal form,
 index 1 critical points correspond to hyperbolic components, and index 0 or 2
 critical points correspond to elliptic components.
 For instance, if $p_1\in\Sigma_1$ is an index 1 critical point with Morse coordinates
 $a,b$ then locally $f_1 = a^2-b^2 = (a-b)(a+b)$ so taking $x = a-b$
 and $\xi = a+b$ gives the coordinates for a hyperbolic block from Theorem~\ref{thm:normal-form}, $q_1=x\xi$.
 Alternatively, one can compare the eigenvalues of the Hessians of the Morse functions with the eigenvalues of the block diagonal Hessian of $F$, which is not the way that we will approach this in the present paper, but is closer to the pointwise classification of Williamson.
\end{remark}

\subsection{Integrable systems and singular Lagrangian fibrations}
\label{sec:fibration}

Let the triple $(M,\om, F)$ be a $2n$-dimensional completely integrable system.
A connected component of a fiber of $F\colon M\to \R^n$ is called \emph{singular}
if it contains a singular point of the integrable system and called
\emph{regular} otherwise. Every regular
component of a fiber of $F$ is a Lagrangian submanifold of $M$. If $M$ is compact then, according to the Arnold-Liouville theorem, all regular components are diffeomorphic to $n$-tori.
Let $B$ be the topological space obtained as the quotient space of $M$ by identifying two points if and only if they are in the same component of the same level set of $F$ and denote by $\tau\colon M\to B$ the associated quotient map.
The map $\tau\colon M\to B$ is in fact a so-called \emph{singular
Lagrangian fibration} of $M$ with \emph{base} $B$ since the union of the fibers of $\tau$ which are Lagrangian submanifolds forms an open dense set in $M$. Thus, we say that the singular Lagrangian fibration $\tau\colon M\to B$ is \emph{induced} by the momentum map $F$.

In the case that all fibers of $F$ are connected, the base $B$ can be naturally
identified with the image of the momentum map $F(M)$. This is the case for toric and semitoric systems. For more general classes of integrable
systems, such as hypersemitoric systems, there is a natural
surjection $B\to F(M)$, but this is not necessarily a bijection.

The next result follows from Bolsinov $\&$ Fomenko~\cite[Proposition 1.16]{BolFom}.

\begin{lemma}\label{lem:image_imm_curve}
Let $(M,\om,F=(J,H))$ be a four dimensional integrable system 
and let $M^{\textrm{HR}}\subset M$ be the set
of hyperbolic-regular singular points. 
If $C\subset M$ is a connected component of $M^{\mathrm{HR}}$
then $F|_C$ is an immersion whose image is a smooth immersed submanifold of dimension 1.
Thus, for any $p\in C$ there exists a set $U\subset C$ which is an
open (as a subset of $C$) neighborhood of $p$ such that
$F(U)$ is a one-dimensional submanifold of $\R^2$.
\end{lemma}

\begin{remark}
 Concerning Lemma~\ref{lem:image_imm_curve}, notice that the connected components 
 of $F(M^{\mathrm{HR}})$ are not always embedded curves in $\R^2$.
 This is because given two components $C$ and $C'$ of $M^{\mathrm{HR}}$
 the curves $F(C)$ and $F(C')$ may pass through each other in $\R^2$, although 
 their images $\tau(C)$ and $\tau(C')$ cannot intersect in the base of the fibration induced by $F$.
\end{remark}

\subsection{Semitoric systems and marked polygons}
\label{sec:semitoric}

Due to Theorem~\ref{thm:normal-form}, semitoric systems (as in Definition~\ref{def:semitoric}) can have singular points
of three types: 
elliptic-elliptic, focus-focus, and elliptic-regular.
A semitoric system $(M,\om,F=(J,H))$ is \emph{simple} if there is
at most one focus-focus point in each fiber of $J$.
Simple semitoric systems were classified
by Pelayo \& V\~{u} Ng\d{o}c \cite{PVN09, PVN11} in terms of five invariants: the number of focus-focus points,
the semitoric polygon, the height invariant, the Taylor series invariant, and the
twisting index invariant.
This classification was extended to non-simple systems by Palmer $\&$ Pelayo $\&$ Tang~\cite{PPT}.
We will focus our attention on the first three invariants, which were already developed in \vungoc\ \cite{VN2007} before the full classification and which extended nearly unchanged to the non-simple case. 
As in Le Floch $\&$ Palmer~\cite{LFPfamilies}, we will collect these three invariants together
into a single invariant called the \emph{marked semitoric polygon invariant}.

Given a semitoric system $(M,\om, F=(J,H))$ there are at most finitely
many singular points of focus-focus type,  according to \vungoc\ \cite[Corollary 5.10]{VN2007}. Denote their number by $\mf\in\Z_{\geq 0}$ and the set of focus-focus singular points by $M^{\mathrm{FF}}=\{p_1,\ldots, p_\mf\}\subset M$.
The images $F(p_1),\ldots,$ $F(p_\mf)\in\R^2$
of these points lie, due to Theorem \ref{thm:normal-form}, in the interior of the momentum map image $F(M)$. Without loss of generality, we may assume that their images are in lexicographic order in $\R^2$ with respect to the `coordinates' induced by $(J,H)$, i.e.\ $J(p_1)\leq \ldots \leq J(p_\mf)$ and if $J(p_\ell)=J(p_{\ell'})$ with $\ell<\ell'$ then $H(p_\ell)\leq H(p_{\ell'})$.
Given a point $c\in\R^2$ and $\epsilon=\pm 1$, let
$\ell_c^\epsilon\subset \R^2$ be the closed ray starting at $c$ directed along a vector in the positive
$H$-direction if $\epsilon=1$ and the negative $H$-direction if $\epsilon=-1$.
Given $\vec{\epsilon}\in\{\pm 1\}^{\mf}$ let $\vec{\ell}^{\vec{\epsilon}}=\cup_{j=1}^\mf \ell_{F(p_j)}^{\epsilon_j}$.
Let $M' = M\setminus F^{-1}(\vec{\ell}^{\vec{\epsilon}})$.
Then, as in \vungoc~\cite[Theorem 3.8]{VN2007}, there exists a homeomorphism $g\colon F(M)\to\R^2$ preserving the first
component which is 
smooth on $F(M)\setminus\vec{\ell}^{\vec{\epsilon}}$
such that each component of $g\circ F|_{M'}\colon M'\to\R^2$ generates an effective $\mbS^1$-action.
We will call such a $g$ a \emph{straightening map}.
The closure of the image of this toric momentum map is a polygon as sketched in Figure~\ref{fig:straightening-poly}. It is unique up to the freedom
in the choices of $\vec{\epsilon}$ and $g$, which we will below encode in a group action on the triple
\begin{equation}\label{eqn:markedpoly_noquotient}
 \big(\De := g\circ F(M), g\circ (F(p_1),\ldots,F(p_\mf)),\vec{\epsilon}\,\big).
\end{equation}

\begin{remark}
Symington developed a similar technique to obtain a polygon, called an almost toric base diagram, from an almost toric fibration, see~\cite{Sym}.
\end{remark}

\begin{remark}
Notice that we have allowed for the case of multiple focus-focus points in the same fiber of $F$, in which case the ordering
of the labeling of the focus-focus points is not unique. This will not cause any problems in constructing a unique
invariant because changing the order of points in the same fiber does not change the resulting marked semitoric polygon.
However, the non-uniqueness of the ordering of the labels
does cause complications if the Taylor series labels are
included (which are not relevant in the present paper), in which case the focus-focus points in
the same fiber have a cyclic ordering as in Palmer $\&$ Pelayo $\&$ Tang~\cite{PPT}.
\end{remark}

\begin{figure}
 \centering
 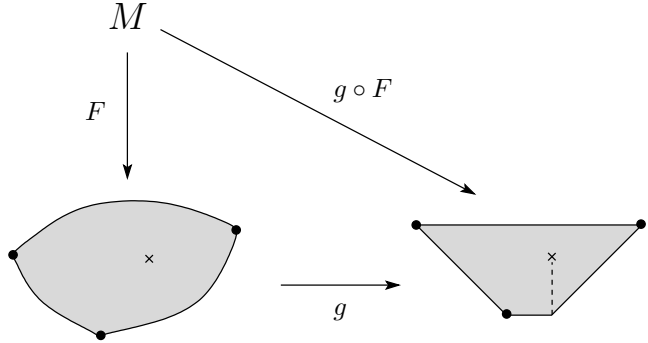 
 \caption{The momentum map image of a semitoric system can be `straightened out' to recover a polygon.}
 \label{fig:straightening-poly}
\end{figure}

For $k=1,2$ let $\pi_k\colon \R^2\to\R$ denote the projection onto the $k^{\mathrm{th}}$ coordinate
and let $s\in\Z_{\geq 0}$.
In general,  we call a triple
\[
 \big(\De,\vec{c} = (c_1,\ldots,c_s) ,\vec{\epsilon} = (\epsilon_1,\ldots, \epsilon_s) \big)
\]
a \emph{marked weighted polygon} if $\De\subset\R^2$ is a polygon, the $\vec{c}$ is a list of points in lexocographic order satisfying 
$c_i\in\mathrm{int}(\De)$ for all $i$, and $\varepsilon_i\in\{\pm 1\}$ for all $i$. Thus, the triple from Equation~\eqref{eqn:markedpoly_noquotient}
obtained from a semitoric system is an example of a marked weighted polygon.
For $j\in\R$, let
\begin{equation}\label{eqn:Tt}
 T = \begin{pmatrix} 1 & 0 \\ 1 & 1 \end{pmatrix} \,\,\,\textrm{ and }\,\,\, \mathfrak{t}_j \colon \R^2\to\R^2, \quad (x,y)\mapsto \begin{cases}(x, y + x - j) & \textrm{ if } x\geq j, \\ (x, y) & \textrm{ otherwise.} \end{cases}
\end{equation}
Let $\mathcal{T}$ be the group generated by powers of $T$ and vertical translations. This is in fact the subgroup of the group of integral affine maps on the plane which preserve the first component. Let $G_s := \{\pm 1\}^s$.
Then $\mathcal{T}\times G_s$ acts on the set of marked weighted polygons by
\[
 (\tau,\vec{\epsilon'})\cdot (\De, \vec{c}, \vec{\epsilon}\,) = \big(\sigma(\De), (\sigma(c_1), \ldots, \sigma(c_s)), (\epsilon_1' \epsilon_1,\ldots, \epsilon_s' \epsilon_s)\big)
\]
where $\sigma = \tau \circ \mathfrak{t}_{\pi_1(c_1)}^{u_1}\circ \cdots \circ \mathfrak{t}_{\pi_1(c_s)}^{u_s}$ and
$u_k = \epsilon_k (1-\epsilon_k')/2$.
Denote by $[\De,\vec{c}, \vec{\epsilon}\,]$ the orbit of $(\De,\vec{c}, \vec{\epsilon}\,)$ under this action.

This group action represents exactly the effect of the choice of straightening map $g$
and cut directions $\vec{\epsilon}$ on the
triple in Equation~\eqref{eqn:markedpoly_noquotient} above: if $(\De,\vec{c},\vec{\epsilon}\,)$ is
the result of that construction for one choice of $g$ and $\vec{\epsilon}$ then the set of all
possible triples produced in this way is exactly the orbit $[\De,\vec{c},\vec{\epsilon}\,]$.
This orbit is the \emph{marked semitoric polygon invariant} of the system $(M,\om,F)$.

\begin{remark}
Note that the marked semitoric polygon invariant contains more information than the semitoric polygon invariant 
introduced in \vungoc~\cite{VN2003} and used in the classification by Pelayo $\&$ \vungoc~\cite{PVN09, PVN11}, since it also
includes the marked points corresponding to the focus-focus values of the system.
The \emph{height invariant} of the semitoric system is encoded in the marked semitoric polygon
as the vertical distance from $c_k$ to the bottom of $\De$. It does not depend on the choice of representative.
\end{remark}

The Taylor series and twisting index invariants are not encoded in the marked
semitoric polygon, so the marked semitoric polygon is not a complete invariant of semitoric systems. The Taylor series, developed by \vungoc~\cite{VN2003}, and extended to the case of multiple focus-focus points in the same fiber by Pelayo $\&$ Tang~\cite{PT}, describes the semi-local (i.e.~in a neighborhood of the fiber) structure around a focus-focus singular point, and the twisting index, introduced in Pelayo $\&$ \vungoc~\cite{PVN09}, roughly, takes into account an additional degree
of freedom when gluing the neighborhood of a focus-focus fiber into the global system.
Since they are not related to the structure of the underlying Hamiltonian $\mbS^1$-space, these two invariants are not be as important as the other three for this paper.

In the remainder of this section, we will describe which marked weighted polygons are obtained from semitoric systems, summarizing results from Pelayo $\&$ \vungoc~\cite{PVN11} adapted to the case of marked polygons that may have multiple marked points in the same vertical line $J=\mathrm{const}$.
We say that a polygon is \emph{rational} if the slope of each 
non-vertical edge is rational.
Let $q$ be a vertex of a rational polygon and let $v,w\in\Z^2$ be the
primitive vectors directing the edges adjacent to $q$. Then we say
that $q$ satisfies:
\begin{itemize}
 \item the \emph{Delzant condition} if $\mathrm{det}(v,w)=1$;
 \item the \emph{hidden Delzant condition for $m$ cuts} if $\mathrm{det}(v,T^m w)=1$;
 \item the \emph{fake condition for $m$ cuts} if $\mathrm{det}(v,T^m w)=0$,
\end{itemize}
where in each case $\mathrm{det}(v,w)$ denotes the determinant
of the matrix with first column $v$ and second column $w$.
The orbit $[\De,\vec{c},\vec{\epsilon}\,]$ can be obtained from a simple
semitoric system if and only if one, and hence all, representatives
$(\De,\vec{c},\vec{\epsilon}\,)$ satisfy the following three conditions:
\begin{enumerate}
 \item $\De$ is rational and convex;
 \item each point of $\partial \De\cap \left(\cup_k \ell^{\epsilon_k}_{c_k}\right)$ is a vertex of $\De$ and satisfies either the fake or hidden
  Delzant condition for $m$ cuts (in which case it is known as a \emph{fake corner} or \emph{hidden corner} respectively) where
  $m>0$ is the number of distinct $k$ such that the vertex in question lies on a ray $\ell^{\epsilon_k}_{c_k}$;
 \item each other vertex of $\De$ satisfies the Delzant condition (and is known
  as a \emph{Delzant corner)}.
\end{enumerate}
Such marked semitoric polygons are known as \emph{marked Delzant semitoric polygons}.
We will represent these by drawing each ray $\ell^{\epsilon_k}_{c_k}$ as a
dotted line (known as a \emph{cut}) and the image under $g\circ F$ of each focus-focus
point will be indicated by a small cross, see Figure~\ref{fig:CP2-5blowups-polygon} for an example.

\subsection{Hamiltonian \texorpdfstring{$\mbS^1$}{S1}-spaces and semitoric systems}
\label{sec:HSS}

In this section we briefly recall the results of Hohloch $\&$ Sabatini $\&$ Sepe~\cite{HSS}.
Given a compact semitoric system $(M,\om, F=(J,H))$ 
the Karshon graph of the underlying Hamiltonian $\mbS^1$-space $(M,\om,J)$ can be obtained from
any representative of the associated marked semitoric polygon $[\De,\vec{c}, \vec{\epsilon}\,]$
much in the same way as obtaining the Karshon graph of a toric system
from the associated Delzant polygon.
Let $(\De,\vec{c},\vec{\epsilon}\,)$ be any representative of $[\De,\vec{c}, \vec{\epsilon}\,]$.
\begin{itemize}
 \item 
 The fixed surfaces of the $\mbS^1$-action are the preimages of the closure of the vertical edges of $\De$ (if any). The value of the normalized symplectic area of a fixed surface equals the length of the associated edge. The genus is always zero. Thus, for each vertical edge of $\De$ the Karshon graph includes a fat vertex labeled by $g=0$ and normalized area equal to the length of the edge.
 \item
 The isolated fixed points of the $\mbS^1$-action are the focus-focus points and also the preimages of any vertices of $\De$ which lie not on vertical edges and are not fake corners.
 Thus, the Karshon graph includes a regular vertex for each Delzant corner, each hidden Delzant corner, and each focus-focus singular point.
 \item
 The $\Z_k$-spheres of the $\mbS^1$-action are determined by the edges of the polygon: Let $e_1, \ldots, e_m$ with $m\geq1$ be a collection of adjacent edges of $\De$ such that the vertex joining $e_\ell$ to $e_{\ell+1}$, for $\ell=1,\ldots, m-1$, is a fake vertex and the remaining two endpoints (of $e_1$ and $e_m$) are not fake.
Then each of these edges has slope $b_\ell / k$ for certain distinct integers $b_\ell$ with $1 \leq \ell \leq m$ and a common integer $k> 0$, such that $b_\ell$ and $k$ are relatively prime. Moreover, the closure of the preimage of the union $e_1\cup\cdots\cup e_m$ is a $\Z_k$-sphere, which, if $k>1$, is represented by an edge labeled by $k$ in the Karshon graph between the vertices corresponding to the endpoints of the piecewise linear curve consisting of $e_1\cup\cdots\cup e_m$.
\end{itemize}

\begin{example}
\label{ex:CP2-45blowups}
We now produce a Hamiltonian $\mbS^1$-space on $\CP^2$ blown up five times which cannot be extended to a toric system but can be extended to a semitoric system, cf.\ Figure \ref{fig:CP2-blowupExample}. This construction will depend on two parameters, $\lambda_1\in\ ]0,1/3[$ and $\lambda_2\in\ ]0,\lambda_1[$.
Starting with $\CP^2$ with the Fubini-Study symplectic form and the (usual) $\mbS^1$-action as in Example~\ref{example:CP2_karshon}, perform two $\mbS^1$-equivariant blowups of the same size $\lambda_1$ on the fixed surface.
This produces two new fixed points, and next we perform one blowup of size
$\lambda_2$ at each of these fixed points. This gives us the Karshon graph as in Figure~\ref{fig:CP2-4blowups-karshon},
which can be extended to the toric system corresponding to the Delzant polygon in Figure~\ref{fig:CP2-4blowups-delzant}.
Now, performing another $\mbS^1$-equivariant blowup of size $\lambda_1$ on the fixed surface produces the Hamiltonian $\mbS^1$-space corresponding
to the Karshon graph in Figure~\ref{fig:CP2-5blowups-karshon} which cannot be extended to a toric system according to Lemma~\ref{lem:toric-extending},
but can be extended to a semitoric system. A representative of the semitoric polygon of such a semitoric system
is shown in Figure~\ref{fig:CP2-5blowups-polygon}.
Figures~\ref{fig:CP2-4blowups-delzant} and~\ref{fig:CP2-5blowups-polygon} represent choices of extensions for the $\mbS^1$-spaces represented by Figures~\ref{fig:CP2-4blowups-karshon} and~\ref{fig:CP2-5blowups-karshon} below them, respectively.
Later, in Example~\ref{ex:CP2_6blowups} we will perform another blowup on this graph, to obtain the Karshon graph shown in Figure~\ref{fig:CP2-6blowups}, and explain that the resulting $\mbS^1$-space cannot be lifted to a semitoric system, but we
show that it can be lifted to a hypersemitoric system.
\end{example}

\begin{figure}
\centering
\begin{subfigure}[t]{.45\textwidth}
\centering
\begin{tikzpicture}[scale = .6]
\foreach \x in {0,1,...,8}{
      \foreach \y in {0,1,...,8}{
        \node[draw,circle,inner sep=.25pt,fill] at (\x,\y) {};
      }
    }
\filldraw[thick, fill = gray!60] (1,1) node[anchor = north,color = black]{}
  -- (3,0) node[anchor = south,color = black]{}
  -- (8,0) node[anchor = north,color = black]{}
  -- (3,5) node[anchor = south,color = black]{}
  -- (1,6) node[anchor = south,color = black]{}
  -- (0,6) node[anchor = south,color = black]{}
  -- (0,2) node[anchor = south,color = black]{}
  -- cycle;
\end{tikzpicture}
\caption{The Delzant polygon of $\mathrm{Bl}^4(\CP^2)$.}
\label{fig:CP2-4blowups-delzant}
\end{subfigure}\quad
\begin{subfigure}[t]{.45\textwidth}
\centering
\begin{tikzpicture}[scale=.6]
\foreach \x in {0,1,...,8}{
      \foreach \y in {0,1,...,8}{
        \node[draw,circle,inner sep=.25pt,fill] at (\x,\y) {};
      }
    }
\filldraw[thick, fill = gray!60] (1,1) node[anchor = north,color = black]{}
  -- (3,0) node[anchor = south,color = black]{}
  -- (8,0) node[anchor = north,color = black]{}
  -- (3,5) node[anchor = south,color = black]{}
  -- (2,5.5) node[anchor = south,color = black]{}
  -- (1,5) node[anchor = south,color = black]{}
  -- (0,4) node[anchor = south,color = black]{}
  -- (0,2) node[anchor = south,color = black]{}
  -- cycle;
\draw [dashed] (2,3) -- (2,5.5);
\draw (2,3) node {$\times$}; 
\end{tikzpicture}
\caption{A semitoric polygon of $\mathrm{Bl}^5(\CP^2)$.}
\label{fig:CP2-5blowups-polygon}
\end{subfigure}\\[1em]
\begin{subfigure}[t]{.45\textwidth}
\centering
\begin{tikzpicture}[scale=.6]
 \draw[black,fill=black] (0,0) ellipse (.15cm and 2cm);	
 \draw (1,1.5) -- (3,1.5) node[midway,above]{2};
 \draw[black,fill=black] (1,1.5) circle (.1cm);
 \draw[black,fill=black] (3,1.5) circle (.1cm);
 \draw (1,-1.5) -- (3,-1.5) node[midway,above]{2};
 \draw[black,fill=black] (1,-1.5) circle (.1cm);
 \draw[black,fill=black] (3,-1.5) circle (.1cm);
 \draw[black,fill=black] (8,0) circle (.1cm);
\end{tikzpicture}
\caption{The Karshon graph of the Hamiltonian $\mbS^1$-space induced by the toric system from Figure~\ref{fig:CP2-4blowups-delzant}.}
\label{fig:CP2-4blowups-karshon}
\end{subfigure}\quad
\begin{subfigure}[t]{.45\textwidth}
\centering
\begin{tikzpicture}[scale = .6]
 \draw[black,fill=black] (0,0) ellipse (.15cm and 1cm);	

 \draw (1,1.5) -- (3,1.5) node[midway,above]{2};
 \draw[black,fill=black] (1,1.5) circle (.1cm);
 \draw[black,fill=black] (3,1.5) circle (.1cm);

 \draw[black,fill=black] (2,0) circle (.1cm);
 
 \draw (1,-1.5) -- (3,-1.5) node[midway,above]{2};
 \draw[black,fill=black] (1,-1.5) circle (.1cm);
 \draw[black,fill=black] (3,-1.5) circle (.1cm);

 \draw[black,fill=black] (8,0) circle (.1cm);
\end{tikzpicture}
\caption{The Karshon graph of the Hamiltonian $\mbS^1$-space induced by the semitoric system from Figure~\ref{fig:CP2-5blowups-polygon}.}
\label{fig:CP2-5blowups-karshon}
\end{subfigure}
\caption{The Delzant polygons, semitoric polygons, and Karshon graphs of $\CP^2$ blown up four and five times as in Example~\ref{ex:CP2-45blowups} with parameters $\lambda_1 = 1/4$ and $\lambda_2=1/8$.
The Delzant and semitoric polygons are drawn over the lattice $\frac{1}{8}\Z^2\subset \R^2$. 
Since the Delzant polygon for $\CP^2$ is the convex hull of $(0,0)$, $(1,0)$, and $(0,1)$, performing the four blowups of the specified sizes
produces the Delzant polygon with vertices at $(1/8, 1/8)$, $(0, 2/8)$, $(0,6/8)$, $(1/8, 6/8)$, $(3/8,5/8)$, $(1,0)$, and $(3/8, 0)$
shown in Figure~\ref{fig:CP2-4blowups-delzant}. Performing one more blowup produces a semitoric system
with the semitoric polygon shown in Figure~\ref{fig:CP2-5blowups-polygon}.} 
\label{fig:CP2-blowupExample}
\end{figure}
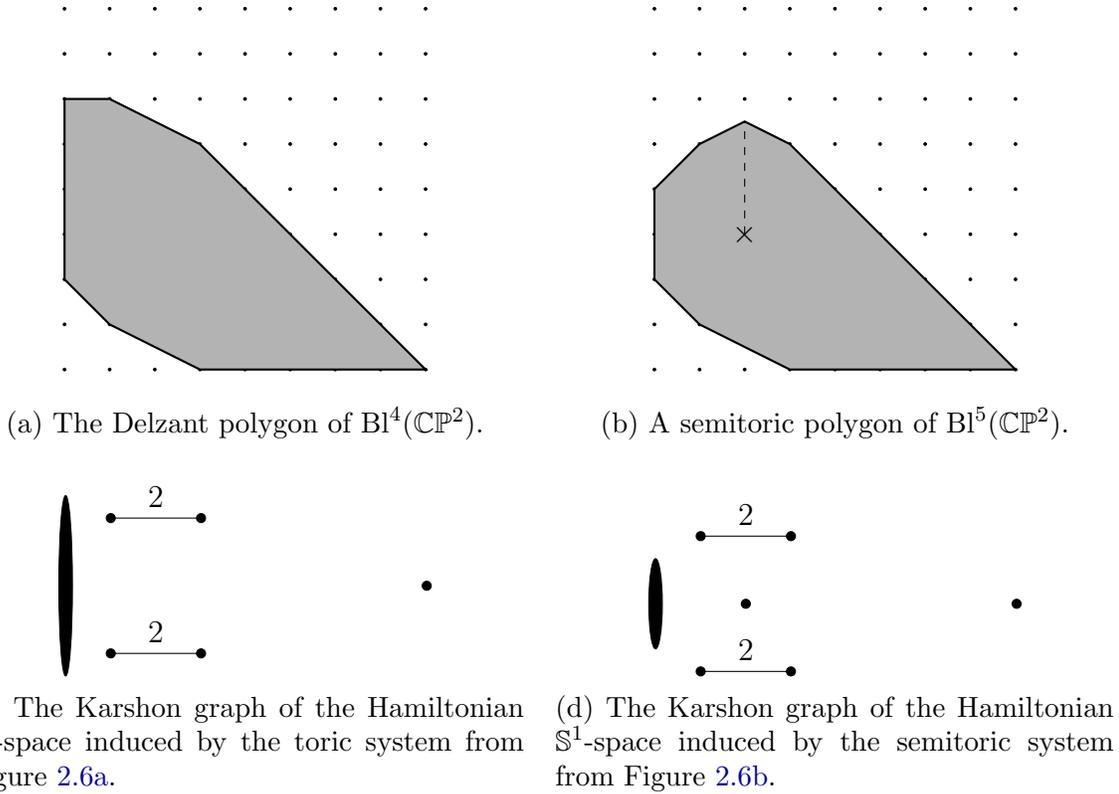

\subsection{Blowups of semitoric systems}
In Section \ref{sec:S1-blowups}, we already encountered blowups for Karshon graphs, Delzant polygons, and toric systems. We will now describe two types of blowups of semitoric systems.

\subsubsection{Toric blowups} 
This type of blowup is motivated by the blowup for toric systems described in Section \ref{sec:S1-blowups}:

\begin{definition}
 Let $(M,\om,F=(J,H))$ be a semitoric system and let $p\in M$ be an elliptic-elliptic
singular point.
Then one can perform a \emph{toric blowup of $(M,\om,F=(J,H))$ at $p$ of size $\lambda>0$} 
if there exists a straightening map $g\colon \R^2\to\R^2$ as in Section~\ref{sec:semitoric}
such that 
\begin{enumerate}
 \item 
 $g\circ F(p)$ is a Delzant corner of $\De := g\circ F(M)$,
 \item
 $g\circ F(p)$ is the unique vertex of $\De$ contained in
$\mathrm{Simp}_{F(p)}^\lambda(v_1,v_2)$, where $v_1,v_2$ are primitive vectors
directing the edges of $\De$ adjacent to $g\circ F(p)$,
\item
$\mathrm{Simp}_{F(p)}^\lambda(v_1,v_2)$ does not intersect any of the cuts or marked points in $\De$.
\end{enumerate}
\end{definition}

The result of a such a toric blowup is a semitoric system which has all of the same invariants as the original system except that its marked polygon is the result of performing a corner chop on the marked polygon invariant of the original system. This process amounts to performing a usual $\T^2$-equivariant blowup with respect to the toric momentum map $g\circ F$. It can be shown that the result of this operation does not depend on any of the choices made, see for instance Le Floch $\&$ Palmer~\cite[Section 4.3]{LFPfamilies}.

\begin{remark}
 Note that we only need to find \emph{one} representative of the marked semitoric polygon
 invariant which admits a corner chop at the vertex corresponding to $p$ in order to perform
 a toric blowup. Also note that given any elliptic-elliptic point there is at least
 one representative such that this point corresponds to a Delzant corner and thus for
 sufficiently small size $\lambda>0$ a toric blowup can always be performed.
\end{remark}

\subsubsection{Wall chops of marked polygons}
\label{sec:semitoric-blowup}
In a semitoric system $(M,\om,F)$, with $F=(J,H)$, a blowup can also be performed at
an elliptic-regular point $p$ if it lies in a surface $\Sigma$ which is fixed by the $\mbS^1$-action generated
by $J$. Such blowups, often referred to as \emph{semitoric blowups}, and their effect on the semitoric system, are described in detail in the upcoming work by Hohloch $\&$ Sabatini $\&$ Sepe $\&$ Symington~\cite{HSSS2}\footnote{first announced at Poisson 2014 in a talk by Daniele Sepe.}, but here we will
simply describe \emph{their effect on the marked semitoric polygon invariant}, and then use
the result of Pelayo $\&$ \vungoc~\cite{PVN09,PVN11} which states there exists a system (actually infinitely many) associated to the resulting marked semitoric polygon.
In the context of almost toric manifolds,
this type of blowup already appeared in the work of Symington~\cite[Section 5.4]{Sym}, and, furthermore, Auroux described this operation in a very explicit example in the context of Lagrangian fibrations when studying wall crossing phenomenon~\cite[Example 3.1.2]{Auroux2009}.
Such blowups are also discussed in the lecture notes by Evans~\cite[Section 9.1]{Evans-notes}.
Sometimes the corresponding operation on the polygon is also called a semitoric blowup, but to avoid confusion and to be consistent with more recent works, such as Le Floch \& Palmer~\cite{LFP-families2}, we call the operation on the polygon a \emph{wall chop} and reserve the term semitoric blowup for the operation on the semitoric system itself.

For any set $B\subset \R^2$ define 
\[
 \partial^+ B := \{(x,y)\in\partial B\mid y\geq y'\textrm{ for all }y'\textrm{ such that }(x,y')\in\partial B\}
\] 
and call it the \emph{upper boundary}
of $B$. Analogously define the \emph{lower boundary} $\partial^- B$.
Note that $\partial^+ B \cup \partial^- B$ is not in general the entire boundary of $B$, but if $B$ is convex specifying $\partial^+ B$ and $\partial^- B$ nevertheless completely determines $B$. In fact, in this case $B$ is equal to the convex hull of $\partial^+ B\cup \partial^-B$.

\begin{definition}\label{def:semitoric_blowup}
Suppose that $[\De,\vec{c},\vec{\epsilon}\,]$ is a marked Delzant semitoric
polygon, and let $j_\mathrm{min},j_\mathrm{max}\in\R$ be such that $\pi_1(\De)= [j_\mathrm{min},j_\mathrm{max}]$. Suppose that $e_{\mathrm{max}}:=\De\cap\pi_1^{-1}(j_\mathrm{max})$ is a vertical edge of $\De$ of length $A>0$. Let $\lambda>0$ be any real number such that $\lambda<A$ and $\lambda < j_{\mathrm{max}}-j_{\mathrm{min}}$.
We now let $(\De',\vec{c'},\vec{\epsilon'} \,)$ be
any marked weighted polygon such that:
\begin{enumerate}
  \item $\De'$ is the unique convex polygon with $\partial^-\De' = \partial^-\De$
   and $\partial^+\De' = \mathfrak{t}_{j_{\mathrm{max}}-\lambda}^{-1} (\partial^+\De)$,
   where $\mathfrak{t}_j$ is as in Equation~\eqref{eqn:Tt}.
  \item $\vec{c'}$ is a list of points in $\mathrm{int}(\De')$ in lexicographic order which has one entry more than $\vec{c}$ such that $\pi_1(\vec{c'})$, the list of $x$-coordinates of the marked points, is equal to the list $\pi_1(\vec{c})$ with one additional value, namely $j_{\mathrm{max}}-\lambda$.
  \item Let the new point be the $\ell^{\mathrm{th}}$ entry in the list. Then set 
  $\vec{\epsilon'} = (\epsilon_1,\ldots,\epsilon_{\ell-1}, 1, \epsilon_{\ell},\ldots,\epsilon_{s}).$
\end{enumerate}
The resulting marked Delzant semitoric polygon $[\De',\vec{c'}, \vec{\epsilon'} \,]$ is called
a \emph{wall chop of $[\De,\vec{c},\vec{\epsilon} \,]$ of size $\lambda$ on $e_{\mathrm{max}}$}.
The case of performing the wall chop on $e_\mathrm{min}:=\De\cap\pi_1^{-1}(j_\mathrm{min})$ proceeds analogously.
\end{definition}


\begin{remark}\label{rmk:stblowup-system}
Note that in Definition~\ref{def:semitoric_blowup} the change in the top boundary 
of $\De'$ is designed so that the new upward cut intersects the top boundary at either
a fake corner or a hidden Delzant corner.
Thus, the orbit $[\De', \vec{c'},\vec{\epsilon'}\,]$
is indeed a marked Delzant semitoric polygon, and thus
there exists infinitely many semitoric systems with this as its marked semitoric polygon invariant, guaranteed by the semitoric classification theorem of Pelayo \& \vungoc~\cite{PVN09,PVN11}. The resulting marked polygon, forgetting about the height of each marked point, does not depend on the choice of representative for $[\De,\vec{c},\vec{\varepsilon}\,]$, since different choices produce different representatives of the same class
$[\De',\vec{c'},\vec{\varepsilon'}\,]$ because $\mathfrak{t}_j$, $\mathfrak{t}_i$, and $T$ all commute.
The height of each marked point in the wall chop can be freely chosen, representing another degree of freedom when performing a wall chop, but it does not change the Karshon graph of the underlying $\mbS^1$-space of the resulting marked polygon.
\end{remark}

\begin{remark}
 We have only specified how this operation, the wall chop, acts on marked semitoric polygon invariants,
 and do not discuss any details of the corresponding operation on semitoric systems themselves, the semitoric blowup.
 Thus the resulting system of the wall chop discussed below in Lemma~\ref{lem:semitoric-blowup}
 is not unique.
 Furthermore we make no attempt to describe
 an operation on the semitoric system itself, instead depending on a description of
 its effect on the invariants since we need no more detailed version for the present paper.
 A description of the semitoric blowup operation on $(M,\om,(J,H))$ will be discussed in detail in the upcoming~\cite{HSSS2}.
\end{remark}

\begin{example}
Performing a wall chop of size $\lambda = 1/4$ on the fixed surface of the polygon in Figure~\ref{fig:CP2-4blowups-delzant} 
results in the polygon in Figure~\ref{fig:CP2-5blowups-polygon}. 
\end{example}

Definition~\ref{def:semitoric_blowup} was made with the following statement in mind.

\begin{lemma}\label{lem:semitoric-blowup}
 Let $(M,\om,F=(J,H))$ be a semitoric system and let $[\De,\vec{c}, \vec{\epsilon}\,]$ be its marked semitoric polygon invariant. 
 Let $[\De',\vec{c'}, \vec{\epsilon'}\,]$ be a choice of wall chop of $[\De,\vec{c}, \vec{\epsilon}\,]$ of size $\lambda$ on the vertical edge $e$. 
Then there exists a semitoric system $(M',\om',(J',H'))$ such that:
 \begin{enumerate}
 \item The semitoric polygon invariant of $(M',\om',(J',H'))$ is $[\De',\vec{c'}, \vec{\epsilon'}\,]$;
 \item $(M',\om',J')$ is the $\mbS^1$-equivariant blowup of $(M,\om,J)$ at a point $p\in F^{-1}(e)$ of size $\lambda$.
 \end{enumerate}
\end{lemma}

\begin{proof}
First of all, notice that by Remark~\ref{rmk:stblowup-system} the result of performing a wall chop on a marked Delzant semitoric polygon is another marked Delzant semitoric polygon.
Let $[\De,\vec{c}, \vec{\epsilon}\,]$ be the marked polygon invariant of $(M,\om,F)$ and let $[\De',\vec{c'}, \vec{\epsilon'}\,]$ be a choice of wall chop of $[\De,\vec{c}, \vec{\epsilon}\,]$.
As recalled in Section \ref{sec:HSS}, from the work by Hohloch $\&$ Sabatini $\&$ Sepe~\cite{HSS}, we obtain now the Karshon graphs underlying $[\De,\vec{c},\vec{\epsilon}\,]$
and $[\De',\vec{c'},\vec{\epsilon'}\,]$: the Karshon graph of the new marked semitoric polygon is related to the Karshon graph of the original one by reducing the area label on the fat vertex corresponding to $\Sigma$ by $\lambda$ and adding a new isolated fixed point with $J$-value label given by
\[ 
J = \begin{cases}
j_{\mathrm{max}}-\lambda, & \textrm{ if } \Sigma = J^{-1}(j_{\mathrm{max}}),\\
j_{\mathrm{min}}+\lambda, & \textrm{ if } \Sigma = J^{-1}(j_\mathrm{min}).
\end{cases}
\]
This is exactly Case~\ref{case:blowup-surface} from Section~\ref{sec:karshon-blowups}
describing the effect of a blowup at a point in a fixed surface on the Karshon graph.
Thus, the underlying $\mbS^1$-space of any system $(M',\om',F')$ associated to the marked Delzant semitoric polygon
$[\De',\vec{c'},\vec{\epsilon'}\,]$ is $\mbS^1$-equivariantly symplectomorphic to the
$\mbS^1$-equivariant blowup of $(M,\om,J)$, since they have the same Karshon graph.
This also proves the last statement of the lemma, since the resulting Karshon graph is the same for any such system.
Moreover, the conditions for a marked Delzant semitoric polygon to admit a wall chop
of size $\lambda$ at $p$ are exactly the same as the conditions for the associated
Hamiltonian $\mbS^1$-space to admit a blowup of size $\lambda$ at $p$.
\end{proof}

\subsection{Symplectic reduction and integrable systems}
\label{sec:reduction}

Given a Hamiltonian action of an abelian compact Lie group $T$ on a symplectic
manifold $(M,\om)$ with momentum map $\mu\colon M\to \mathfrak{t}^*\cong \R^{\mathrm{dim}(T)}$, 
where $\mathfrak{t}^*$ is the dual of the Lie algebra $\mathfrak t$ of $T$,
the \emph{symplectic reduction} or \emph{symplectic quotient at level $j\in \mu(M)$} is defined
to be
\[
 (M/\!\!/T)_j := \mu^{-1}(j)/T.
\]
Let $\pi_j\colon \mu^{-1}(j)\to (M/\!\!/T)_j$ be the projection on the quotient, and denote by $i_j\colon \mu^{-1}(j)\to M$ the inclusion map.
The Marsden-Weinstein-Meyer Theorem~\cite{marsden-weinstein, meyer}
describes symplectic reduction for general compact Lie groups
$G$. 
We will use the following version specialized to the situation that the Lie group is a torus.

\begin{theorem}[{\cite[Proposition III.2.15]{Audin-torus}}]
Let $T$ be a compact abelian Lie 
group acting on $(M,\om)$ in a Hamiltonian fashion with momentum map $\mu\colon M\to \mathfrak{t}^*$. 
If $T$ acts freely on $\mu^{-1}(j)$ then $(M/\!\!/T)_j$
is a smooth manifold and there exists a symplectic form $\om_j$ on $(M/\!\!/T)_j$
satisfying $\pi_j^*\om_j = i_j^*\om$.
\end{theorem}

More generally, if $T$ does not act freely on $\mu^{-1}(j)$ then $(M/\!\!/T)_j$ is
a type of singular space called a stratified symplectic space, but it still inherits 
smooth and symplectic
structures on the set of points $\pi_j(x)\in(M/\!\!/T)_j$ 
such that $T$ acts freely on the fiber 
$\pi_j^{-1}(\pi_j(x))$.
For a detailed study of singular reduction and stratified symplectic
spaces, see Cushman $\&$ Bates~\cite{cushman-bates}, Sjamaar $\&$ Lerman~\cite{sjamaar-lerman}, and Alonso~\cite{jaume-thesis}.

Now let $(M,\om,J)$ be a Hamiltonian $\mbS^1$-space.
Consider the quotient 
\[\hat{M} := M/\mbS^1\]
 with quotient map $\pi\colon M\to\hat{M}$ and let 
\[\hat{M}_j := (J^{-1}(j))/\mbS^1\]
 be the symplectic quotient for $j\in J(M)$.
Denote by $\textrm{Non-free}(J)\subset M$ the set of points on which the $\mbS^1$-action generated by $J$ does not act freely, i.e.\ the points in $\textrm{Non-free}(J)\subset M$ are those with non-trivial stabilizer.
Then $\hat{M}$ inherits a smooth structure away from $\mathrm{sing}(\hat{M}):=\pi(\textrm{Non-free}(J))$. Set $\mathrm{smooth}(\hat{M}) := \hat{M}\setminus \mathrm{sing}(\hat{M})$. Note that $\hat{M}_j = \pi(J^{-1}(j))\subset \hat{M}$ and let $\hat{H}_j := \hat{H}|_{\hat{M}_j}$.
We now describe the relationship between the Morse classification of the
critical points of $\hat{H}_j$ on (the smooth parts of) $\hat{M}_j$
and the classification of singular points of the integrable system $(M,\om,F)$.


\begin{lemma}\label{lem:reduction-type}
 Let $(M,\om,(J,H))$ be an integrable system such
 that $J$ generates an $\mbS^1$-action, 
 let $p \in M$ be such that $\mbS^1$ acts freely on $p$, and
 let $c = \pi(p) \in \hat{M}_j$ where $j = J(p)$.
 Then $c\in \mathrm{smooth}(\hat{M})$.
 Furthermore:
 \begin{enumerate}[label = \textup{(\alph*)}]
  \item \label{item:reduction-a} $c$ is a regular point of $\hat{H}_j$ 
  if and only if
 $p$ is a regular point
 of $(J,H)$, 
  \item \label{item:reduction-b} $c$ is a non-degenerate critical point
 of $\hat{H}_j$ with index $0$ or $2$ 
 if and only if $p$ is an elliptic-regular singular point of $(J,H)$,
  \item \label{item:reduction-c} $c$ is a non-degenerate critical point
 of $\hat{H}_j$ with index $1$ if and only if $p$ is
 a hyperbolic-regular singular point of $(J,H)$.
 \end{enumerate}
 In particular, the last two items imply that $c$ is non-degenerate in the sense of Morse theory
 if and only if $p$ is non-degenerate in the sense of integrable systems.
\end{lemma}

\begin{proof}
Since the $\mbS^1$-action is proper, smooth, and free at $p$ the statement $c\in\mathrm{smooth}(\hat{M})$ is a standard result from the theory of group actions. Moreover, keep in mind that $J$ inducing a free $\mbS^1$-action at $p$ implies $\mathrm{d}J(p) \neq 0$.
 
 \ref{item:reduction-a}: $c$ is a regular point of $\hat{H}_j$ if and only if $\mathrm{d}\hat{H}_j(c)\neq 0$ if and only if $\mathrm{d}H(p)$ and $\mathrm{d}J(p)$ are linearly independent, which is the definition of $p$ being a regular point of $(J,H)$.
 
\ref{item:reduction-b} $\&$ \ref{item:reduction-c}: According to Hohloch $\&$ Palmer~\cite[Lemma 2.4]{HohPal}\footnote{In the published version of~\cite{HohPal} there is a small error in Lemma 2.4, it requires that $\mathrm{d}J(p)\neq 0$ but it should actually require that the $\mbS^1$-action generated by $J$ acts freely on $p$.}, $c$ is a non-degenerate critical point of $\hat{H}_j$ in the sense of Morse theory if and only if $p$ is a non-degenerate singular point of $(J,H)$. Due to $\mathrm{d}J(p) \neq 0$, the point $p$ is not a fixed point of $(J,H)$. Thus, Theorem~\ref{thm:normal-form} only leaves two possibilities, namely $p$ is elliptic-regular or $p$ hyperbolic-regular.
On the other hand, again due to Hohloch \& Palmer~\cite[Lemma 2.4]{HohPal}, if $p$ is an elliptic-regular or hyperbolic-regular point then $c$ must be a non-degenerate
critical point of $\hat{H}_j$ in the sense of Morse theory.
We conclude that $c$ is a non-degenerate point of $\hat{H}_j$ if and only if $p$ is either an elliptic-regular or hyperbolic-regular
point of $(J,H)$.

All that remains is to associate the index of the critical point $c$ with the type of the singular point $p$. 
We can do this by examining the topology of the fibers.
First, notice that if $p$ is an elliptic-regular point then the local normal
form (cf.~Theorem~\ref{thm:normal-form}) sends a neighborhood of $p$ in the level set $F^{-1}(F(p))$
diffeomorphically to an open set in $\R^1$. Since all of the points in the orbit of $p$
under the $\mbS^1$-action are also elliptic-regular, these local charts cover the connected
component of $F^{-1}(F(p))$ which contains $p$,
and thus this component of $F^{-1}(F(p))$ is a compact one dimensional manifold,
i.e. it is diffeomorphic to $\mbS^1$. If $p$ is a hyperbolic regular point,
the local normal form shows that nearby $p$ the level set $F^{-1}(F(p))$ is homeomorphic to two
planes intersecting at a line, and therefore in this case the connected component of $F^{-1}(F(p))$
containing $p$ cannot be homeomorphic to $\mbS^1$.
Thus, in our situation $p$ is an elliptic-regular point if and only if 
the connected component of $F^{-1}(F(p))$ which
contains $p$ is homeomorphic to $\mbS^1$.
Now we are prepared to prove the claim.
Supposing that $p$ is an elliptic-regular or hyperbolic-regular point, $p$ is an elliptic-regular point if
and only if the connected component of $F^{-1}(F(p))$ which
contains $p$ is homeomorphic to $\mbS^1$, if and only if the connected component of $\hat{H}_j^{-1}(\hat{H}_j(c))$
which contains $c$ is a single point, if and only if the index of $c$ is 0 or 2 (since we already know that $c$ is a non-degenerate
critical point in this situation).
Therefore, we also conclude that $p$ is hyperbolic regular if and only if $c$ has index 1.
\end{proof}


\subsection{Parabolic degenerate points}
\label{sec:parabolic_points}
Now we  motivate and define the type of degenerate points that appear in the definition of hypersemitoric systems. These degenerate singularities were referred to by Colin De Verdiere~\cite{CdV} as `the simplest non-Morse [i.e.~degenerate] example' of singular points in integrable systems, and moreover, we will see that they naturally occur in many systems with hyperbolic-regular points.

\begin{definition}[{Bolsinov \& Guglielmi \& Kudryavtseva~\cite[Definition 2.1]{Bol-parabolic}}]\label{def:parabolic}
Let $(M,\om,F)$ be a $4$-dimensional integrable system and $p\in M$ a singular point such that $\mathrm{d}f_1(p)\neq 0$ where $(f_1,f_2)=g\circ F$ for some local diffeomorphism $g$ of $\R^2$ defined in a neighborhood of $F(p)$. Define 
$$ \tilde{f_2} := \tilde{f}_{2,p} := (f_2)|_{f_1^{-1}(f_1(p))} \colon f_1^{-1}(f_1(p))\to \R.$$
The point $p$ is a \emph{parabolic degenerate singular point}, briefly a {\em parabolic point}, if:
 \begin{enumerate}
  \item \label{item:rnk1} $p$ is a critical point of $\tilde{f_2}$,
  \item \label{item:rankD2f} $\text{rank}(\mathrm{d}^2\tilde{f_2}(p))=1$,
  \item \label{item:v3} there exists $v\in\text{ker}(\mathrm{d}^2\tilde{f_2}(p))$ such that 
  \[
 v^3(\tilde{f_2}) := \frac{\mathrm{d}^3}{\mathrm{d}t^3} \tilde{f_2}(\gamma(t))|_{t=0}
\]
 is nonzero, where $\gamma\colon \ ]-\varepsilon,\varepsilon [ \ \to f_1^{-1}(f_1(p))$ is a curve satisfying $\gamma(0)=p$ and $\dot{\gamma}(0) = v$.  
  \item \label{item:extra-cond} $\mathrm{rank}\big(\mathrm{d}^2(f_2 - k f_1)(p)\big)=3$, where $k\in\R$ is determined by $\mathrm{d}f_2(p) = k\mathrm{d}f_1(p)$.
 \end{enumerate}
 We call the image of a parabolic singular point a \emph{parabolic singular value of $F$}, briefly a {\em parabolic value}.
\end{definition}

\begin{remark}\label{rmk:v3}
That the definition of $v^3(\tilde{f_2})$ does not depend on the choice of the the curve $\gamma$ is proven in Bolsinov $\&$ Guglielmi $\&$ Kudryavtseva~\cite[Remark 2.1]{Bol-parabolic}.
\end{remark}

Intuitively, a parabolic degenerate point can be thought of as a singular point where the rank of all relevant operators is as maximal as possible without the point being non-degenerate:
in particular, to motivate items (2) and (4) in the definition note that for a degenerate point $\text{rank}(\mathrm{d}^2\tilde{f_2}(p))<2$, and that $\mathrm{rank}\big(\mathrm{d}^2(f_2 - k f_1)(p)\big)\leq 3$ is implied by conditions (1)-(3).
Regarding item (3), there is degeneracy up to order two so item (3) is a way of requiring something of order (3) to be non-degenerate.

Being a parabolic point is invariant under the following changes of the involved integrals:

\begin{proposition}[{Bolsinov \& Guglielmi \& Kudryavtseva~\cite[Proposition 7.1]{Bol-parabolic}}]
\label{prop:parabolic-welldef}
 Let $p\in M$ be a parabolic point of $(M,\om,(f_1,f_2))$ in the
 sense of Definition~\ref{def:parabolic} (thus, in particular, $\mathrm{d} f_1(p)\neq 0$) and let $(\mathcal{J},\mathcal{H}) = g(f_1,f_2)$ where $g$ is a local diffeomorphism of $\R^2$ such that $\mathrm{d}\mathcal{J}(p)\neq 0$. Then $p$ is also parabolic with respect to $(\mathcal{J},\mathcal{H})$, i.e.~$(\mathcal{J}, \mathcal{H})$ satisfies conditions \eqref{item:rnk1}-\eqref{item:extra-cond} of Definition~\ref{def:parabolic}.
\end{proposition}

Sometimes parabolic singular points are also called {\em cuspidal singular points}. Parabolic/ cuspidal points were studied for instance by Efstathiou $\&$ Giacobbe~\cite{EG-cusps}, Bolsinov \& Guglielmi \& Kudryavtseva~\cite{Bol-parabolic}, Kudryavtseva \& Martynchuk~\cite{KudMar-circle,KudMar-invariants}, and Bolsinov $\&$ Fomenko~\cite{BolFom}. We used the definition by Bolsinov \& Guglielmi \& Kudryavtseva~\cite{Bol-parabolic} since it is the best adapted to our situation.

\begin{remark}
The complete symplectic invariants of parabolic points and parabolic orbits are described in
the analytic case by Bolsinov $\&$ Guglielmi $\&$ Kudryavtseva~\cite{Bol-parabolic}
and later extended to the smooth case by Kudryavtseva \& Martynchuk~\cite{KudMar-invariants}. These invariants are non-trivial and were found to be encoded in the affine structure of the base of the Lagrangian fibration near the parabolic values. 
\end{remark}

Parabolic points do not admit a \emph{symplectic} normal form, but they do admit a \emph{smooth} normal form:

%

\begin{proposition}[{Kudryavtseva \& Martynchuk~\cite[Theorem 3.1]{KudMar-circle}}]\label{prop:cuspnormalform}
 Let $p\in M$ be a parabolic singular point of an integrable system $(M,\om, F=(f_1,f_2))$ for which $\mathrm{d}f_1(p)\neq 0$.
 Then there exists a neighborhood $U$ of $p$ equipped with coordinates $(x,y,t,\theta)$ centered
 at $p$, and a local diffeomorphism $g=(g_1,g_2)$ of $\R^2$ around the origin with
 $g_1(x_1,x_2) = \pm x_1 + \mathrm{const}$ and $\frac{\mathrm{d}}{\mathrm{d}x_2}(g_2)\neq 0$ such that
 \[
  g\circ F|_U = (t, x^3+tx+y^2). 
 \]
\end{proposition}


The analytic case of Proposition~\ref{prop:cuspnormalform} was first proved by Bolsinov \& Guglielmi \& Kudryavtseva~~\cite[Proposition 2.1]{Bol-parabolic}
before being extended to the smooth case in Kudryavtseva \& Martynchuk~\cite[Theorem 3.1]{KudMar-circle}.


\begin{example}
\label{ex:parabolic}
The origin is a parabolic point for the integrable system given by the local normal form $F\colon \R^4\to\R^2$, $(x,y,t,\theta)\mapsto (t, x^3+tx+y^2)$ equipped with the symplectic form $\om=\mathrm{d}x\wedge\mathrm{d}y + \mathrm{d}t\wedge\mathrm{d}\theta$.
\end{example}

Here it is important to keep in mind that the coordinates $(x, y, t, \theta)$ in Proposition~\ref{prop:cuspnormalform} are in general not canonical, i.e.\ the symplectic form does not always take the standard form as in Example~\ref{ex:parabolic} after change of coordinates.


\subsection{Parabolic points in presence of an effective \texorpdfstring{$\mbS^1$-action}{S1 action}}
\label{sec:parabolic_points_action}

Now we want to show in two steps that for compact systems $(M, \omega, (J,H))$ where $J$ induces an effective $\mbS^1$-action we always have $\mathrm{d}J\neq 0$ at parabolic points, i.e., we can take the local diffeomorphism $g$ from Definition~\ref{def:parabolic} to be the identity.

\begin{lemma}\label{lem:dJ-nonzero-interior}
Let $(M,\om, F=(J,H))$ be a compact integrable system for which the flow of $J$ generates an effective $\mbS^1$-action and let $p\in M$ be a rank 1 singular point. 
Then, if $J(p)\in\mathrm{interior}(J(M))$, we have $\mathrm{d}J(p)\neq 0$. Therefore any rank 1 singular point $q$ with $F(q)\in\mathrm{interior}(F(M))$ has $\mathrm{d}J(q)\neq 0$.
\end{lemma}

\begin{proof}
 Assume that $\mathrm{d}J(p)=0$. Then $p$ is a fixed
 point of the $\mbS^1$-action generated by $J$ and all points in the orbit of $p$ under the flow generated by $H$ are also fixed points of the flow generated by $J$. Since $p$ is a rank 1 point it is not fixed by the flow generated by $H$, so $p$ is a non-isolated fixed point of the $\mbS^1$-action. 
 By Lemma~\ref{lem:karshon_fixedset}, this means that $p$ belongs to a fixed surface of the $\mbS^1$-action and $J(p)$ is in the boundary of the interval $J(M)$, and thus $F(p)$ is in the boundary of $F(M)$.
\end{proof}

\begin{corollary}
\label{cor:para-dJ}
 Let $(M,\om,(J,H))$ be a compact integrable system such that $J$
 generates an $\mbS^1$-action and let $p\in M$ be a parabolic point. Then $\mathrm{d}J(p)\neq 0$.
 Therefore, the local diffeomorphism $g$ from Definition~\ref{def:parabolic} can be taken to be the identity.
 That is, one can work directly with the given integrals $(J,H)$.
\end{corollary}

\begin{proof}
The image of the parabolic point in the normal form given in Proposition~\ref{prop:cuspnormalform} is the point $(0,0)$ and $g^{-1}(0,0)$ is an interior point of $F(M)$,
 so $F(p)$ lies in the interior of $F(M)$.
 By Lemma~\ref{lem:dJ-nonzero-interior} this implies that $\mathrm{d}J(p)\neq 0$.
\end{proof}

Note that Corollary \ref{cor:para-dJ} applies in particular to hypersemitoric systems.

%



\subsection{Whereabouts of parabolic points in integrable systems}

Considering $\theta$ in Proposition~\ref{prop:cuspnormalform} as a parameter, we see that parabolic points come in one parameter families in $M$, which project to a single point in $F(M)$.

Furthermore, in a neighborhood of a parabolic point there are two surfaces of 
non-degenerate singular points which meet at the parabolic point, one
of hyperbolic-regular points and one of elliptic-regular points. 
In the momentum map image this appears as a curve of images of elliptic-regular points
which meets at the image of the parabolic point with a curve of images of hyperbolic-regular points. An example with two parabolic values is sketched in Figure~\ref{fig:flap}.


\begin{definition}
 \label{def:flap}
A \emph{flap} in an integrable system $(M, \omega, F)$ is an open set $S\subset M$ which is a connected component of $F^{-1}(F(S))$ with the following property: there is a connected open region in $F(S)$ bounded by three smooth curves such that
\begin{itemize}
    \item two of the curves have the property that, other than their endpoints, they are each the image of a family of elliptic-regular singular points in $S$,
    \item the other curve, with its endpoints removed, is the image of a family of hyperbolic-regular singular points in $S$,
    \item the point where the two curves of elliptic-regular values meet is the image of an elliptic-elliptic singular point in $S$,
    \item each of the other two corners of the region is the image of a parabolic singular orbit in $S$, 
    \item there are no other singular point in $S$.
\end{itemize}
\end{definition}

Outside of the set $S$ anything can be happening, but this does not interact with the behavior of the flap. In Figure~\ref{fig:flap} we sketch the simplest situation, the one in which $S = F^{-1}(F(S))$ (and therefore the number of components of each fiber is always 1 or 2) and $F(S)$ lies in the interior of $F(M)$.
In such a flap, the fibers above the interior points of the triangle are the disjoint union of 2 Lagrangian tori, outside of the triangle they are a single torus, above the elliptic-regular values they are the disjoint union of a torus and a circle, above the elliptic-elliptic value it is the disjoint union of a torus and a point, above the hyperbolic-regular values they are double tori (cf.\ Figure~\ref{fig:double-torus}), and above the two degenerate values they are cuspidal tori (cf.\ Figure~\ref{fig:cusp-torus}).
Throughout the paper, we sometimes call the image $F(S)$ a flap as well, refering to the flap $S$.
A detailed discussion of the topological properties of parabolic singularities in particular in the case of flaps can be found in Efstathiou $\&$ Giacobbe~\cite{EG-cusps} where they appear under the name \emph{cuspidal} singular points.



\begin{figure}
 \centering
 \includegraphics[width = 280pt]{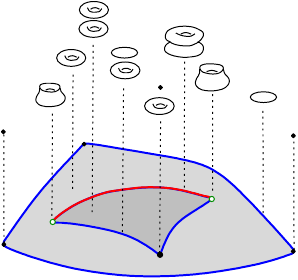}
 \caption{A flap in the momentum map image. Elliptic-elliptic values are marked by black dots and parabolic values by green punctured dots. Hyperbolic-regular values are sketched in red and elliptic-regular ones in blue. Regular values are painted gray. A look at the fibers shows when cases overlap.}
 \label{fig:flap}
\end{figure}

%
%
%
%

The name {\em flaps} is motivated by their topological form in the base space $B$ of the singular Lagrangian fibration induced by $F$ (as discussed in Section~\ref{sec:fibration}).
The region around this triangle in $B$ can be obtained by gluing the triangle to the rest of the base space along the curve of hyperbolic-regular points and parabolic degenerate points, so it is like a flap glued onto the momentum map image. Example~\ref{ex:dullin-pelayo} and Example \ref{ex:LFP} have such flaps,
as does the system described in Example~\ref{ex:CP2_6blowups}.

\begin{figure}
\centering
\begin{subfigure}[t]{.48\textwidth}
\centering
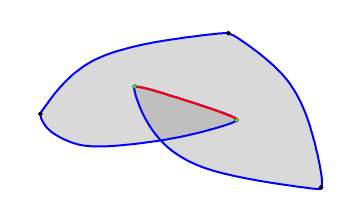
\caption{The types of singular values in the momentum map image.}
\end{subfigure}\quad
\begin{subfigure}[t]{.48\textwidth}
\centering
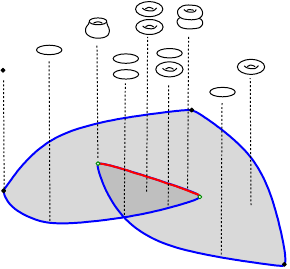
\caption{The fibers over the momentum map image.}
\end{subfigure}
 \caption{A momentum map image of an integrable system which includes a pleat (or `swallow tail') like the system from Le Floch $\&$ Palmer~\cite[Section 6.6]{LFPfamilies}. Elliptic-elliptic values (EE) are marked by black dots and parabolic values (D as in `degenerate') by green punctured dots. Hyperbolic-regular values (HR) are sketched in red and elliptic-regular ones (ER) in blue. Regular values are painted gray. A look at the fibers shows when cases overlap.}
 \label{fig:LFP-system}
\end{figure}

\begin{remark}
Parabolic points also appear in so-called \emph{pleats} (which are also called {\em swallowtails}) as sketched in Figure \ref{fig:LFP-system} (see also Le Floch $\&$ Palmer~\cite[Section 6.6]{LFPfamilies}). The explicit definition can be found in Efstathiou $\&$ Giacobbe~\cite{EG-cusps}. Pleats were studied from a global viewpoint by Efstathiou \& Sugny~\cite{EfstSugny-swallowtail}.
\end{remark}

\begin{remark}
Parabolic points are also very common in physical systems, see for instance Bolsinov $\&$ Rikhter $\&$ Fomenko ~\cite{BolRikFom-Ktop} and the references in Bolsinov $\&$ Guglielmi $\&$ Kudryavtseva~\cite{Bol-parabolic}.
\end{remark}


\subsection{Parabolic points and reduction}
\label{sec:para-reduction}

Parabolic points locally admit a natural $\mbS^1$-action,
and the coordinates from Proposition~\ref{prop:cuspnormalform}
can actually be extended to a tubular neighborhood of the entire
orbit of the parabolic point.
The following result was first obtained in the analytic case by Bolsinov $\&$ Guglielmi $\&$ Kudryavtseva~\cite[Proposition 3.1]{Bol-parabolic}
and later extended to the smooth case by Kudryavtseva \& Martynchuk~\cite[Theorem 3.1]{KudMar-circle}.

\begin{proposition}[{Kudryavtseva \& Martynchuk~\cite[Theorem 3.1]{KudMar-circle}}]
Let $(M,\om,F)$ be a compact integrable system and let $p\in M$ be a parabolic point.
Let $\Lambda_p$ be the connected component of $F^{-1}(F(p))$ which contains $p$.
Then coordinates $(x,y,t,\theta)$ around $p$ from Proposition~\ref{prop:cuspnormalform} can be extended to a tubular
neighborhood of $\Lambda_p$, taking $\theta\in\R/2\pi\Z$, such that in these coordinates
the symplectic form is given by
\[
\om = A(x,y,t)\mathrm{d}x\wedge\mathrm{d}y + \mathrm{d}t\wedge\mathrm{d}\theta +B(x,y,t) \mathrm{d}t\wedge\mathrm{d}x + C(x,y,t) \mathrm{d}t\wedge\mathrm{d}y
\]
for some functions $A(x,y,t), B(x,y,t), C(x,y,t)$ such that $A(x,y,t)>0$.
\end{proposition}


There is a local Hamiltonian $\mbS^1$-action in a neighborhood of $p$ with Hamiltonian $J(x,y,t,\theta)=t$ 
and Hamiltonian vector
field $\frac{\partial}{\partial \theta}$, and if the original integrable system admits a global 
Hamiltonian $\mbS^1$-action it
must equal this one up to sign. Notice that this $\mbS^1$-action is free in a neighborhood
of the parabolic orbit.
In a neighborhood of the orbit of $p$, performing symplectic reduction with respect to this local
$\mbS^1$-action
at some level $t$ yields a disk with coordinates $(x,y)$ and symplectic
form $\om_t := A(x,y,t)\mathrm{d}x\wedge\mathrm{d}y$. 
On this disk the other Hamiltonian reduces to a function
\begin{equation}\label{eqn:birth-death}
 f_t(x,y) = x^3+tx+y^2.
\end{equation}
We can think of this as a family of functions on the disk parameterized by $t$.
The graph of $f_t$ and its level sets for various values of $t$ can be seen in
Figure~\ref{fig:cerf}.
Notice:
\begin{itemize}
 \item if $t<0$ then $f_t$ has two non-degenerate critical points (of index 1 and 0). The point $(x,y)=(\sqrt{-t/3},\ 0)$ is the index 1 critical point and it lies on the level set $f_t^{-1}\left(-2\left(\frac{-t}{3}\right)^{3/2}\right)$ which traces out a curve with a loop as sketched in Figure~\ref{fig:cerf} on the right. We call the region enclosed by this loop the \emph{teardrop region}.
 \item if $t>0$ then $f_t$ has no critical points;
 \item if $t=0$ then $f_t$ has exactly one critical point, which is degenerate.
\end{itemize}
This family parameterizes the process of two non-degenerate critical
points coming together and annihilating as $t$ increases, or being born as $t$ decreases,
and thus it is called a \emph{birth-death singularity}.
In fact, Equation~\eqref{eqn:birth-death} is the typical such bifurcation in Morse theory, cf.\ Cerf~\cite{Cerf}.

\begin{figure}
 \centering
 \includegraphics[width = 320pt]{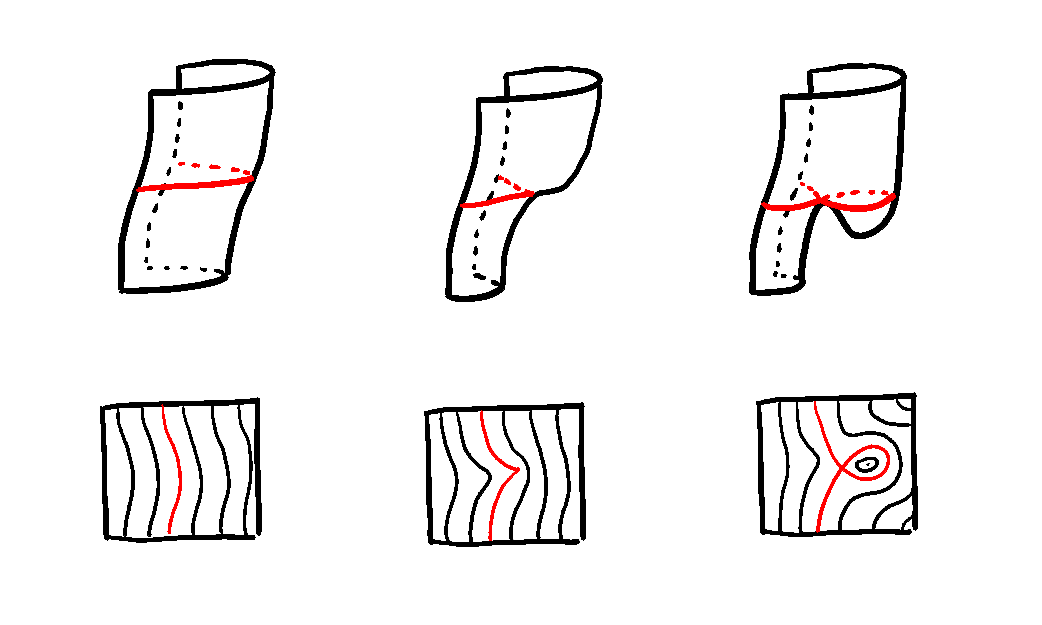}
 \caption{The function $f_t(x) = x^3+tx+y^2$ and its level sets for $t=1$ (left), $t=0$ (middle), and $t=-1$ (right).
 The \emph{teardrop region} is the region enclosed by the red curve in the right figure.}
 \label{fig:cerf}
\end{figure}

\begin{remark}
In an integrable system with a flap, this one-parameter family of functions on the reduced space at level $J^{-1}(j)$ produces a teardrop region which appears, grows, shrinks, and disappears as $j$ increases.
\end{remark}

%


\section{Examples of hypersemitoric systems}
\label{sec:examples}

This section focuses on intuitive but essential examples of hypersemitoric systems (cf.~Definition~\ref{def:hypersemitoric}) with which we
\begin{itemize}
 \item 
 explain intuitively a method how to obtain a hypersemitoric system from a semitoric system by means of replacing a focus-focus point in a semitoric system by a flap;
 \item
 give two explicit examples, Examples~\ref{ex:CP2_6blowups} and~\ref{ex:S2T2}, for the original motivation to introduce the class of hypersemitoric systems, namely being able to extend an effective Hamiltonian $\mbS^1$-action that cannot be extended to a semitoric system to the `next easiest and natural class of integrable system.' The $\mbS^1$-space in Example~\ref{ex:CP2_6blowups} does not extend to a semitoric system because of the presence of three $\Z_2$-spheres in a single level set of $J$, while the $\mbS^1$-space in Example~\ref{ex:S2T2} does not extend since the fixed surfaces have nonzero genus. We explicitly show how each of these examples can be extended to a hypersemitoric system.
\end{itemize}


One way to obtain examples of hypersemitoric systems is the following technique by Dullin \& Pelayo~\cite{dullin-pelayo}:

\begin{figure}
 \centering
 \begin{subfigure}[t]{.45\textwidth}
  \centering
  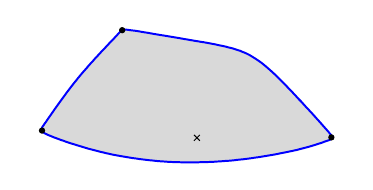
  \caption{}
 \end{subfigure}\qquad
\begin{subfigure}[t]{.45\textwidth}
\centering
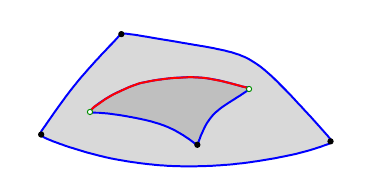
\caption{}
\end{subfigure}
 \caption{Dullin $\&$ Pelayo~\cite{dullin-pelayo} consider a semitoric systems with a focus-focus singular point (sketched in (a)) and turn the focus-focus point into a flap (sketched in (b)). For details on the fibers of a flap, see also Figure~\ref{fig:flap}.}
 \label{fig:triangle}
\end{figure}

\begin{example}
\label{ex:dullin-pelayo}
Dullin \& Pelayo~\cite{dullin-pelayo} take a semitoric system $(M,\om,(J,H))$ that has at least one focus-focus point and then perturb $H$ near a focus-focus point $p\in M$ to produce a new integrable system in which $p$ is no longer focus-focus but instead is the elliptic-elliptic point of a flap, as sketched in Figure~\ref{fig:triangle}. More precisely, they define a function $G\colon M\to\R$ supported in a neighborhood of $p$ working in coordinates of the local normal form. 
Then they replace the original integrable system $(J,H)$ by $(J, \widetilde{H} = H + G)$ and obtain a new system with a flap as desired. Eventually, they give in \cite[Section 8]{dullin-pelayo} the following explicit example on $M=\mbS^2\times\R^2$ with $\omega$ being the product of the standard symplectic forms: Inducing coordinates $(x,y,z)$ on $\mbS^2$ via the inclusion $\mbS^2\subset\R^3$
and using coordinates $(u,v)$ on $\R^2$, they set
\begin{align*}
& J (x,y,z,u,v):= \frac{u^2+v^2}{2}+z,\\
& H (x,y,z,u,v):= \frac{xu+yu}{2}, \\
& G_\gamma (x,y,z,u,v):= \gamma z^2,
\end{align*}
where $0\leq \gamma \leq 1$. Note that they work in fact with a globally defined $G_\gamma$.
Define $\widetilde{H}_\gamma := H + G_\gamma$, i.e., one has $\widetilde{H}_0=H$.

Then $(M,\om,(J,H))$ is the so-called \emph{coupled spin oscillator} as in \vungoc\ \cite{VN2007},
and $(M,\om,(J,\widetilde{H}_\gamma))$ transitions from the coupled spin oscillator at $\gamma=0$ into a system  at $\gamma>\frac{1}{2}$ which has hyperbolic singularities.
\end{example}

Note that Dullin $\&$ Pelayo's~\cite{dullin-pelayo} technique does not change $J$ or $(M,\om)$ and thus does not change the underlying Hamiltonian $\mbS^1$-space.
Moreover, if the original system was a hypersemitoric system then this operation produces a hypersemitoric system, since the $\mbS^1$-action is preserved and the only new singular
points introduced are either non-degenerate or parabolic (see~\cite[Remark 6.4]{dullin-pelayo}).

There is also an example where a system first displays a flap and eventually a pleat:

\begin{example}[Le Floch $\&$ Palmer~\cite{LFPfamilies}]
\label{ex:LFP}
On the first Hirzebruch surface, Le Floch $\&$ Palmer~\cite[Section 6.6]{LFPfamilies} give an explicit example of a parameter-dependent integrable system $(J,H_t)$ with $0\leq t\leq 1$ which transitions from being toric for $t=0$, to having a flap of singular values for $t\approx \frac{1}{2}$, to having a pleat for $t\approx 1$.
\end{example}

\begin{example}
 Gullentops~\cite{yannick-thesis} gives explicit examples of integrable systems with hyperbolic singularities by perturbing a
 toric integrable system on $\mbS^2\times\mbS^2$ blown up four times (whose associated Delzant polygon is an octagon).
\end{example}

\begin{remark}
 Note that Dullin \& Pelayo's~\cite{dullin-pelayo} technique can produce many examples of hypersemitoric systems from semitoric systems, but not all hypersemitoric systems can be formed this way: For instance hypersemitoric systems can have fixed surfaces which are not spheres, in which case they could never come from a semitoric system via this technique.
\end{remark}

Now we give two examples that illustrate the idea of our proof of Theorem~\ref{thm:extending} in Section~\ref{sec:extending}. They focus on systems where the given or underlying Hamiltonian $\mbS^1$-spaces cannot be extended to semitoric systems, thus motivating the introduction of hypersemitoric systems.

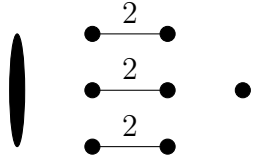
\begin{figure}[h]
\begin{tikzpicture}
 \draw[black,fill=black] (0,0) ellipse (.1cm and .75cm);	
 \draw (1,.75) -- (2,.75) node[midway,above]{2};
 \draw[black,fill=black] (1,.75) circle (.1cm);
 \draw[black,fill=black] (2,.75) circle (.1cm);
 \draw (1,0) -- (2,0) node[midway,above]{2};
 \draw[black,fill=black] (1,0) circle (.1cm);
 \draw[black,fill=black] (2,0) circle (.1cm);
 \draw (1,-.75) -- (2,-.75) node[midway,above]{2};
 \draw[black,fill=black] (1,-.75) circle (.1cm);
 \draw[black,fill=black] (2,-.75) circle (.1cm);
 \draw[black,fill=black] (3,0) circle (.1cm);
\end{tikzpicture}
 \caption{A Hamiltonian $\mbS^1$-space which cannot be obtained from an $\mbS^1$-action underlying a toric or semitoric system.}
 \label{fig:CP2-6blowups}
\end{figure}

\begin{example}
\label{ex:CP2_6blowups}
We will produce a Hamiltonian $\mbS^1$-space on $\CP^2$ blown up six times which cannot be extended
to a semitoric system but can be extended to a hypersemitoric system:
consider the Hamiltonian $\mbS^1$-space on $\mathrm{Bl}^5(\CP^2)$ described in Example~\ref{ex:CP2-45blowups} whose Karshon graph is shown in Figure~\ref{fig:CP2-5blowups-karshon}.
Now, using the notation from Example~\ref{ex:CP2-45blowups}, we perform an additional blowup of size $\lambda_2$ at the isolated fixed point which is not at the maximum value of $J$. This yields the Karshon graph in Figure~\ref{fig:CP2-6blowups} which has three edges at the same time passing through a level set $J^{-1}(j)$. In order to underly a semitoric system, each edge in the Karshon graph must be associated to a unique edge (or chain of edges) from the top or bottom boundary of the semitoric polygon (cf.\ Hohloch $\&$ Sabatini $\&$ Sepe~\cite{HSS}), so there cannot be three edges in the Karshon graph which intersect in a single level set $J^{-1}(j)$. Thus this system cannot be extended to a semitoric system.

We describe how to produce a hypersemitoric system to which this Hamiltonian $\mbS^1$-space can be extended. The idea of the procedure is sketched in Figure~\ref{fig:trick}.
Start with the semitoric system $(J,H)$ on $\mathrm{Bl}^5(\CP^2)$ and use the technique described in Example~\ref{ex:dullin-pelayo} around the unique focus-focus point to create a new integrable system $(J,\widetilde{H})$. The point $p\in\mathrm{Bl}^5(\CP^2)$ that was focus-focus in $(J,H)$ was turned into an elliptic-elliptic singular point with hyperbolic-regular points and two parabolic degenerate points nearby, i.e., the image of $(J,\widetilde{H})$ now contains a flap as sketched in Figure~\ref{fig:triangle}.
Notice that $(J,H)$ and $(J,\widetilde{H})$ have the same underlying Hamiltonian $\mbS^1$-space
since the manifold, symplectic form, and $J$ remain all unchanged.
Since $p$ is an elliptic-elliptic singular point of $(J,\widetilde{H})$ we
can perform a toric blowdown of size $\lambda_1$ on $p$ (assuming $\lambda_1$ is sufficiently small), 
producing a new integrable system on $\mathrm{Bl}^6(\CP^2)$ which is a hypersemitoric system and has the desired underlying Hamiltonian $\mbS^1$-space.

Notice that in the general case (i.e.~in the proof of Theorem~\ref{thm:extending}) one difficulty
that we must deal with is to make sure that the flap can be made sufficiently large to accommodate 
the required blowups. This is non-trivial and is the content of Proposition~\ref{prop:large-flap}.
\end{example}

\begin{figure}
\centering
\begin{subfigure}[t]{.40\textwidth}
\centering
\begin{tikzpicture}
\filldraw[thick, fill = white, color = white] (0,0) node[anchor = north,color = black]{}
  -- (0,3) node[anchor = south,color = black]{}
  -- (3,3) node[anchor = south,color = black]{}
  -- (3,0) node[anchor = north,color = black]{}
  -- cycle;
 \draw[black,fill=black] (0,1.5) ellipse (.15cm and 1.5cm);	
 \draw[black,fill=black] (3,1.5) ellipse (.15cm and 1.5cm);	
\draw (-0.2,1.5) node[anchor=east]{$g=1$};
\draw (3.2,1.5) node[anchor=west]{$g=1$}; 
\end{tikzpicture}
\caption{There are no $\Z_k$-spheres with $k>1$ in the Karshon graph and the existing fixed points occur in the fixed surfaces, which are both tori.}
\label{fig:S2T2-karshon}
\end{subfigure}\quad
\begin{subfigure}[t]{.56\textwidth}
\centering
\begin{tikzpicture}
\filldraw[thick, fill = gray!60] (0,0) node[anchor = north,color = black]{}
  -- (0,3) node[anchor = south,color = black]{}
  -- (3,3) node[anchor = south,color = black]{}
  -- (3,0) node[anchor = north,color = black]{}
  -- cycle;
\draw [thick] (0,2.25) -- (3,2.25);
\draw [thick] (0,.75) -- (3,.75);
\draw (0,.75) node[anchor=east]{HE};
\draw (0,2.25) node[anchor=east]{HE};
\draw (3,.75) node[anchor=west]{HE};
\draw (3,2.25) node[anchor=west]{HE};
\draw (1.5,.75) node[anchor=south]{HR};
\draw (1.5,2.25) node[anchor=south]{HR};
\draw (0,0) node[anchor=north east]{EE};
\draw (0,3) node[anchor=south east]{EE};
\draw (3,3) node[anchor=south west]{EE};
\draw (3,0) node[anchor=north west]{EE};
\draw[black,fill=black] (0,0) circle (.05cm);
\draw[black,fill=black] (0,.75) circle (.05cm);
\draw[black,fill=black] (0,2.25) circle (.05cm);
\draw[black,fill=black] (0,3) circle (.05cm);
\draw[black,fill=black] (3,0) circle (.05cm);
\draw[black,fill=black] (3,.75) circle (.05cm);
\draw[black,fill=black] (3,2.25) circle (.05cm);   
\draw[black,fill=black] (3,3) circle (.05cm);
\end{tikzpicture}
\caption{The momentum map image with labeling HR (hyperbolic-regular), HE (hyperbolic-elliptic), and EE (elliptic-elliptic). The boundary consists entirely of elliptic-regular points, except for the marked rank zero points.}
\label{fig:S2T2-image}
\end{subfigure}
\caption{The Karshon graph and momentum map image for the system from Example~\ref{ex:S2T2} defined on $\mbS^2\times\T^2$.}
\label{fig:S2T2}
\end{figure}
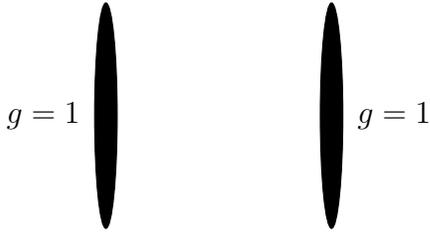
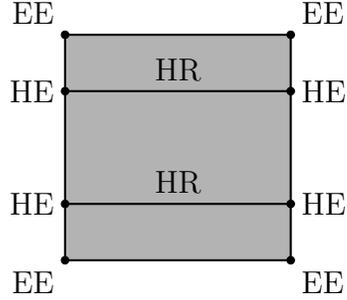

\begin{example}\label{ex:S2T2}
Consider the Hamiltonian $\mbS^1$-space given by $M=\mbS^2\times\T^2$ equipped with the direct sum of the usual symplectic forms and Hamiltonian $J$ given by the usual height function on the sphere. Thus, the $\mbS^1$-action rotates the sphere fixing the poles and does not effect the torus.
Then $M^{\mbS^1}$ is the disjoint union of two copies of $\T^2$. There are no $\Z_k$ spheres with $k>1$, so the Karshon graph is as in Figure~\ref{fig:S2T2-karshon}. 
In toric or semitoric systems the components of the fixed point set of the underlying $\mbS^1$-space are always isolated points or embedded spheres (cf.\ Hohloch $\&$ Sabatini $\&$ Sepe~\cite{HSS}), so this cannot be extended to a semitoric system. 
Consider the torus presented as the surface of revolution of the circle $x^2+(y-2)^2=1$ around the $x$-axis, and consider the function
$h(x,y,z)=z$ restricted to this surface. Then $h$ is the usual Morse function on $\T^2$, i.e.~the (unperturbed) height function when the torus is standing `on its end' which has four critical points: one of index 0, two of index 1, and one of index 2.
Let $\pi_{\T^2}\colon M\to\T^2$ be the projection and define $H := h \circ \pi_{\T^2}\colon M\to\R$.

Then, as already shown in Remark~\ref{rmk:morse}, $(M,\om,(J,H))$ is an integrable system with
no degenerate points.
By considering the index of the critical points of the Morse functions, we deduce that this integrable system has elliptic-elliptic, elliptic-regular,
hyperbolic-regular, and hyperbolic-elliptic points.
Since $J$ is proper and generates an effective $\mbS^1$-action $(J,H)$ is thus a hypersemitoric system. The image of $(J,H)$ with the types of points is shown in Figure~\ref{fig:S2T2-image}.
\end{example}


 \section{Properties of integrable systems with \texorpdfstring{$\mbS^1$}{S1}-actions}
\label{sec:properties}

In this section we will make use of the local normal form theorem (Theorem~\ref{thm:normal-form})
which implies that in a neighborhood $U$ of a rank 0 singular point $p$ of an integrable system $(M,\om,F=(J,H))$ there are 
coordinates $\psi\colon U\to\R^4$, with coordinate functions $\psi = (x,\xi,y,\eta)$, such that there are functions $f_1,f_2\colon U\to\R$ depending on the singularity type of $p$ satisfying $\{f_i,J\}=\{f_i,H\}=0$ for $i\in\{1,2\}$.
Note that the presence of hyperbolic blocks in the singularities prevents the general existence of a local diffeomorphism $g$ of $\R^2$ with $g\circ (J,H) = (f_1,f_2)$, cf.\ item~\eqref{item:g} of Theorem~\ref{thm:normal-form}. Since $J$ Poisson commutes with each $f_i$ (i.e.\ $f_1$ and $f_2$ are invariant under the flow of $\mathcal{X}^J$) the flow of $\mathcal{X}^J$ stays on the level sets of $(f_1,f_2)$ --- which will be sufficient for our proofs in the following.

\subsection{Systems with \texorpdfstring{$\mbS^1$}{S1}-actions}
\label{sec:properties_S1action}

In this section we prove some results about
 integrable systems in which one of the two integrals generates an $\mbS^1$-action.

The following proposition is probably well-known to experts,
but for the convenience of the reader we include a short proof here.
Item~\eqref{item:no_HH} also follows from the work of Zung~\cite{ZungI,ZungII}, in which he
classifies the local symmetries of non-degenerate singular points without depending on the local normal form theorem (Theorem~\ref{thm:normal-form}). There he notes in particular that hyperbolic-hyperbolic
singularities do not admit a local $\mbS^1$-action, see Zung~\cite[Theorem 6.1]{ZungI}.

\begin{proposition}\label{prop:HE-and-HH}
 Let $(M,\om,F=(J,H))$ be an integrable system such that $J$ generates an effective $\mbS^1$-action. 
 Then:
 \begin{enumerate}
  \item\label{item:HE_fixed} If $p\in M$ is a singular point of hyperbolic-elliptic type then $p$ is
   a non-isolated fixed point of the $\mathbb{S}^1$-action, and thus it lies in a fixed surface
   of the $\mbS^1$-action;
  \item\label{item:no_HH} $(M,\om,F)$ has no singular points of hyperbolic-hyperbolic type.
 \end{enumerate}
\end{proposition}

\begin{proof}
Keep in mind that $J\colon M\to\R$ generates an effective global $\mbS^1$-action and let $p\in M$ be a singular point of rank 0 of $(J,H)$. Then the local normal form theorem (cf.\ Theorem \ref{thm:normal-form}) implies that in a neighborhood $U$ of $p$ there exist symplectic coordinates 
  \[
   \psi=(x,\xi, y, \eta)\colon U\to\R^4
  \]
such that $\psi(p)=(0,0,0,0)$ and functions $f_1,f_2\colon U\to\R$ of a certain form (depending on the type of $p$) such that $\{J|_U,f_i\}=0$ for $i\in\{1,2\}$ in $U$.
Since $p$ is a rank zero singular point of $(J,H)$, in particular, we have $dJ(p)=0$ so $p$ is a fixed point of the $\mbS^1$-action.

\eqref{item:HE_fixed} If $p\in M$ is of hyperbolic-elliptic type then, up to reordering the integrals, we may work with the functions
  \[ f_1 = (x^2+\xi^2)/2\quad \textrm{and} \quad f_2= y\eta\]
  defined on $U$.
  Consider the point $\widetilde{p} = \psi^{-1}(0,0,Y,0)\in U$ for some sufficiently small $Y \in \R$ and let $\phi_t^J$ be the flow of $\mathcal{X}^J$ and $\phi^1_t$ the flow of $f_1$.
  Since $\phi_t^J$ preserves $f_1$ and $f_2$ it follows that $\phi^J_t(\widetilde{p}) \in \psi^{-1}\{(0,0,z,0)\mid z\in\R\}$ for small enough $t$. 
  Thus, if the action on $\widetilde{p}$ is non-trivial then either the forward or backward flow
  has to approach the fixed point $p$ contradicting the fact that the flow is
  periodic. Thus, for all sufficiently small $Y$ the point $\psi^{-1}(0,0,Y,0)$ is fixed by the
  $\mbS^1$-action so $p$ is not an isolated fixed point of the $\mbS^1$-action,
  and so by Lemma~\ref{lem:karshon_fixedset} it must lie on a fixed surface.

 \eqref{item:no_HH} If $p\in M$ is of hyperbolic-hyperbolic type then we may take
  \[ f_1 = x\xi \quad \textrm{and} \quad f_2 = y\eta.\]
Let $p'=\psi^{-1}(X,0,0,0)\in U$ for some sufficiently small $X$. Then 
\[
 \mathrm{d}J(p')\in\mathrm{span}\{\mathrm{d}f_1(p'),\mathrm{d}f_2(p')\}=\mathrm{span}\{\mathrm{d}\xi\}
\]
so $\mathcal{X}^J\in\mathrm{span}\{\partial_ x \}$.
If $p'$ is not fixed by the $\mbS^1$-action then either forward or backward flow of $\mathcal{X}^J$ must approach the fixed 
point $p$. This contradicts the fact that the flow of $\mathcal{X}^J$ is periodic, so $p'$ must be a fixed point
of the $\mbS^1$-action.
Thus, following similar reasoning we see that for all sufficiently small $X,\Xi,Y,P\in\R$ the points $\psi^{-1}(X,0,0,0)$, $\psi^{-1}(0,\Xi,0,0)$, 
$\psi^{-1}(0,0,Y,0)$, and $\psi^{-1}(0,0,0,P)$ are all fixed by the $\mbS^1$-action, so $p$ is not an isolated fixed point and
also the component of $M^{\mbS^1}$ containing $p$ is not a surface.
This contradicts Lemma~\ref{lem:karshon_fixedset}.
\end{proof}

Let $M^{\textrm{HR}},M^{\textrm{HE}},M^{\mathrm{D}}\subset M$ be the set
of hyperbolic-regular singular points, hyperbolic-elliptic singular points, 
and degenerate singular points, respectively.

\begin{lemma}\label{lem:no_vert_tangent}
 Let $(M,\om,(J,H))$ be an 
 integrable system such that $J$ generates an effective $\mbS^1$-action
 and let $C\subset M$ be a connected component of $M^{\mathrm{HR}}$.
 Then $F(C)\subset\R^2$ does not have any vertical tangencies.
\end{lemma}

\begin{proof}
 Suppose that $p\in M^{\mathrm{HR}}$ and let $C\subset M^{\mathrm{HR}}$ be
 the connected component of $M^{\mathrm{HR}}$ containing $p$.
 By Lemma~\ref{lem:image_imm_curve}, $F(C)\subset\R^2$ is a one-dimensional 
 immersed submanifold.
 Since $\mathrm{rank}(p)=1$ there are $a, b \in \R$ such that $b\mathcal{X}^J(p)-a\mathcal{X}^H(p)=0$. According to Bolsinov $\&$ Fomenko~\cite[Proposition 1.16]{BolFom}, the tangent vector to the curve $F(C)$ at $F(p)$ is then given by $(a,b)$. 
 Now assume that $F(C)$ has a vertical tangent in $p$. Then we must have $a=0$ and $b\neq 0$, and therefore $\mathcal{X}^J(p)=0$ and thus $\mathrm{d}J(p)=0$. 
Since $p$ is a rank 1 singular point, by Lemma~\ref{lem:dJ-nonzero-interior} we see that $\mathrm{d}J(p)=0$ implies that $F(p)\in \partial(F(M))$,
but the image of hyperbolic-regular points must lie in the interior of the momentum map image according to the local normal form in Theorem \ref{thm:normal-form}. Thus $F(C)$ cannot have a vertical tangent in $p$.
\end{proof}

We know that the image of a connected component of $M^{\mathrm{HR}}$ is a smooth immersed curve by Lemma~\ref{lem:image_imm_curve}, and Lemma~\ref{lem:no_vert_tangent} implies that in integrable systems $(M,\om,(J,H))$ such that $J$ generates an effective
$\mbS^1$-action it is actually embedded (and is a graph over $J$).
The following result is illustrated in Figure~\ref{fig:noloops}.

\begin{corollary}\label{cor:no-loops}
If $C\subset M$ is a connected component of $M^{\mathrm{HR}}$ then
$F(C)$ is homeomorphic to an open interval with distinct endpoints and each endpoint is either an element of $F(M^{\textrm{HE}})$ or $F(M^{\textrm{D}})$.
In particular,  $F(C)$ is not homeomorphic to a loop.
Furthermore, $F(C)$ is not a curve which connects the image of a fixed
surface back to itself.
\end{corollary}

\begin{proof}
By Lemma~\ref{lem:no_vert_tangent}, $F(C)$ does not have any vertical tangent so it cannot 
include a self-intersection and cannot form a loop. Thus, $F(C)$ is homeomorphic to an interval.
Due to the local normal forms result, Theorem~\ref{thm:normal-form}, given any point $p\in M$
which is regular or singular of type elliptic-elliptic, elliptic-regular, or focus-focus, there exists
a neighborhood of $p$ which does not include any hyperbolic-regular points.
Thus, the end points of the interval $F(C)$ must be the image of the only other possible points
in $M$, either hyperbolic-elliptic or degenerate.
Finally, $\pi_1\colon F(C)\to J(C)$ is injective by Lemma~\ref{lem:no_vert_tangent}, 
so $F(C)$ cannot connect the image of a fixed surface back to itself (as in Figure~\ref{fig:noloopsb}).
\end{proof}

\begin{figure}
\centering
\begin{subfigure}[t]{.30\textwidth}
\centering
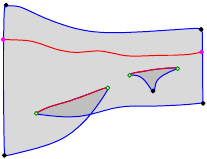
\caption{Possible.}
\end{subfigure}\quad
\begin{subfigure}[t]{.30\textwidth}
\centering
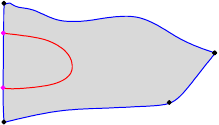
\caption{Impossible.}
\label{fig:noloopsb}
\end{subfigure}\quad
\begin{subfigure}[t]{.30\textwidth}
\centering
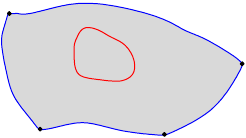
\caption{Impossible.}
\end{subfigure}
 \caption{
The nature of the relationship between the existence of hyperbolic-regular singularities and 
degenerate singularities in the presence of a global $\mbS^1$-action has already been considered
by Dullin $\&$ Pelayo in early versions of~\cite{dullin-pelayo}.  
 (a) is a possible momentum map image of an integrable system with a global $\mbS^1$-action. This example includes a flap, a pleat, and a curve of hyperbolic-regular points whose two endpoints  are both hyperbolic-elliptic points; we have seen such behaviors in the systems from Section~\ref{sec:examples}.
 (b) is not possible for an integrable system with a global $\mbS^1$-action since a family of hyperbolic-regular points cannot connect a fixed surface to itself, cf.\ Corollary \ref{cor:no-loops}. (c) is not possible for an integrable system with a global $\mbS^1$-action since such a system cannot have a loop of hyperbolic-regular values, cf.\ Corollary \ref{cor:no-loops}.}
 \label{fig:noloops}
\end{figure}

\begin{corollary}\label{cor:forced-degen}
 Let $(M,\om,(J,H))$ be an integrable system such that $J$ generates an effective $\mbS^1$-action.
 If the $\mathbb{S}^1$-action has strictly less than two fixed surfaces and $(J,H)$ has a singularity of hyperbolic-regular type, 
 then $(J,H)$ has at least one degenerate singular point. 
\end{corollary}

\begin{proof}
 Recall that hyperbolic-regular values come in one-parameter families.
 By Corollary~\ref{cor:no-loops} this family terminates in two points, each of which are either the image of a degenerate or hyperbolic-elliptic point.
 By Lemma~\ref{prop:HE-and-HH}, hyperbolic-elliptic points always lie in a fixed surface. Thus, if there is at most one fixed surfaces, not more than one end point can be the image of hyperbolic-elliptic points since a single family of hyperbolic-regular points cannot connect the same fixed surface to itself, cf.\ Corollary~\ref{cor:no-loops}.
\end{proof}

Corollary~\ref{cor:forced-degen} is the reason why we need to allow certain degenerate singular points when we want to extend all possible Hamiltonian $\mbS^1$-spaces.
In particular, any $\mathbb{S}^1$-action which has less than two fixed surfaces and cannot be extended to a semitoric system must have at least one degenerate point in any extension:
 
 \begin{corollary}\label{cor:forced-degen-2ndversion}
  Let $(M,\om,J)$ be a Hamiltonian $\mbS^1$-space which has strictly less than two fixed surfaces and has three or more $\Z_k$-spheres 
  passing through a single level set of $J$.
  If $H\colon M\to\R$ is such that $(M,\om,(J,H))$ is an integrable system, then $(M,\om,(J,H))$
 has at least one degenerate singular point.
 \end{corollary}

In particular, this implies that the Hamiltonian $\mbS^1$-space on $\CP^2$ blown up six times from
Example~\ref{ex:CP2_6blowups}, sketched in Figure~\ref{fig:CP2-6blowups}, cannot be extended to an integrable system with no degenerate points.
This justifies why we need to allow degenerate points in hypersemitoric systems, since our
goal is to show that all Hamiltonian $\mbS^1$-spaces can be extended to such a system.

\subsection{Fibers of hypersemitoric systems and isotropy}
\label{sec:fibers_isotropy}
In this section we will discuss how the topology of the fibers which contain hyperbolic-regular
points is related to the isotropy of the $\mbS^1$-action in certain cases.
A common fiber which contains hyperbolic-regular points is the \emph{double torus}, which
is homeomorphic to two tori glued along an $\mbS^1$, 
or, equivalently, to a figure eight (i.e.~an immersion
of $\mbS^1$ with a single transverse self-intersection) crossed with $\mbS^1$, as in Figure~\ref{fig:double-torus}.
This fiber occurs in every example in Section~\ref{sec:examples}, and in particular
this is the type of fiber that contains the hyperbolic-regular points produced
by the technique of Dullin \& Pelayo~\cite{dullin-pelayo}.

The fibers of integrable systems for which one of the integrals
generates an $\mbS^1$-action can be more complicated than this, though.
Another possibility, for instance, is the \emph{curled torus}, see Figure~\ref{fig:curled-torus}.
A curled torus fiber is homeomorphic to a figure eight crossed with the interval
$[0,1]$ modulo the relation $(x,0)\sim(\phi(x),1)$,
where $\phi$ is a map from the figure eight to itself which switches the top
and bottom teardrops (for instance rotation by $\pi$). 

\begin{figure}
\centering
\begin{subfigure}[t]{.30\textwidth}
\centering
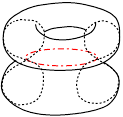
\caption{The double torus.}
\label{fig:double-torus}
\end{subfigure}\quad
\begin{subfigure}[t]{.30\textwidth}
\centering
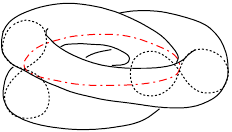
\caption{The curled torus.}
\label{fig:curled-torus}
\end{subfigure}\quad
\begin{subfigure}[t]{.30\textwidth}
\centering
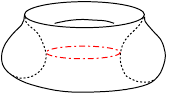
\caption{The cuspidal torus.}
 \label{fig:cusp-torus}
\end{subfigure}
\caption{Some fibers which include hyperbolic-regular singular points.}
\label{fig:hyp-fibers}
\end{figure}

\begin{proposition}
 Let $(M,\om,F=(J,H))$ be an 
 integrable system such that $J$ generates an effective $\mbS^1$-action 
 and
 consider a fiber $F^{-1}(c)$. Then
 \begin{enumerate}
  \item if $F^{-1}(c)$ is a double torus then the $\mbS^1$-action generated by $J$ acts freely
   on $F^{-1}(c)$;
  \item if $F^{-1}(c)$ is a curled torus then the $\mbS^1$-action generated by $J$ acts freely
   on the regular points of $F^{-1}(c)$ and acts with isotropy subgroup $\Z \slash 2 \Z$ on the
   hyperbolic-regular points of $F^{-1}(c)$.
 \end{enumerate}
\end{proposition}

\begin{proof}
First suppose that $F^{-1}(c)$ is a double torus fiber.
Using the local normal form around any of the hyperbolic-regular points we can see that the period of the flow
of $\mathcal{X}^J$ is equal for all points in a neighborhood of the hyperbolic-regular point in question. Similarly,
the flow of a hyperbolic-regular point in a curled torus fiber will have exactly half of the period 
of the flow of the regular points in an open set around that point. Since an effective Hamiltonian $\mbS^1$-action is free
on a dense set this proves the claim.
\end{proof}


More complicated fibers are also possible.
See Gullentops~\cite{yannick-thesis} and the references therein for a detailed discussion of the possible
fibers in integrable systems with hyperbolic-regular singularities. 
In the following Example \ref{ex:Efst}
there are fibers which contain multiple curves of hyperbolic-regular points.
Note that the system also includes degenerate points which
are not of parabolic type, and thus this system is 
not a hypersemitoric system.


\begin{example}[Martynchuk \& Efstathiou~\cite{Mart-Efst-seifert}, Section 4.2]
\label{ex:Efst}
 Consider $M=\mbS^2\times\mbS^2$ with coordinates $((x_1,y_1,z_1),(x_2,y_2,z_2))$ obtained from inclusion
 as the product of unit spheres in $\R^3\times\R^3$, the standard product symplectic form $\om = \om_{\mbS^2}\oplus\om_{\mbS^2}$, and the integrable system
 $F=(J,H)\colon \mbS^2\times\mbS^2\to\R^2$ given by
 \[
  J = z_1 + 2 z_2, \qquad H = \mathrm{Re}\left( (x_1+\mathrm{i}y_1)^2 (x_2-\mathrm{i}y_2) \right).
 \]
 Let $N=(0,0,1)$ denote the north pole and $S=(0,0,-1)$ denote the south pole of each sphere.
 $J$ generates an effective $\mbS^1$-action with four fixed points given by
 \[
  A = (S,S), \quad B = (N,S), \quad C = (S,N), \quad \textrm{ and } \quad D = (N,N),
 \]
 and two $\Z_2$-spheres: one connecting $A$ to $C$, consisting of points of the form $(S,(x_2,y_2,z_2))$, and
 one connecting $B$ to $D$, consisting of points of the form $(N, (x_2,y_2,z_2))$. Thus the Karshon
 graph of $(M,\om,J)$ is as shown in Figure~\ref{fig:Efst-Karshon}.
 The integrable system $(M,\om,(J,H))$ is not a hypersemitoric system since the points $A$ and $D$ are degenerate points
 that are not parabolic since parabolic points are always in the interior of the momentum map. The image of $F$ is shown in Figure~\ref{fig:Efst-image}.
 Given any $c\in\R^2$ on the open interval connecting $F(A)$ to $F(B)$ or $F(C)$ to $F(D)$ the fiber $F^{-1}(c)$
 is a curled torus, and given any $c\in\R^2$ on the open interval connecting $F(B)$ to $F(C)$ the 
 fiber $F^{-1}(c)$ is two curled torus fibers glued along a regular orbit. 
 Comparing Figure~\ref{fig:Efst-Karshon} to Figure~\ref{fig:Efst-image} we see that the two $\Z_2$-spheres get mapped
 to the same points in the image in the region between $F(B)$ and $F(C)$.
\end{example}

\begin{figure}
\begin{subfigure}[t]{.48\textwidth}
\centering\centering\raisebox{20pt}{
\begin{tikzpicture}[scale=.8]
 \draw (-3,.7) -- (1,.7) node[midway,above]{2};
 \draw[black,fill=black] (-3,.7) circle (.1cm);
 \draw[black,fill=black] (1,.7) circle (.1cm);
 \draw (-1,-.7) -- (3,-.7) node[midway,above]{2};
 \draw[black,fill=black] (-1,-.7) circle (.1cm);
 \draw[black,fill=black] (3,-.7) circle (.1cm);
\end{tikzpicture}}
\caption{The Karshon graph.}
\label{fig:Efst-Karshon}
\end{subfigure}
\begin{subfigure}[t]{.48\textwidth}
\centering
\begin{tikzpicture}[scale=.8]
 \draw[bend left=60, fill = gray!40] (-3,0) to (3,0);
 \draw[bend left=60, fill = gray!40] (3,0) to (-3,0);
 \draw[black,fill=black] (-3,0) circle (.1cm) node[left]{F(A)};
 \draw (-3,0) -- (-1,0);
 \draw[black,fill=black] (-1,0) circle (.1cm) node[above]{F(B)};
 \draw (-1,0) -- (1,0);
 \draw[black,fill=black] (1,0) circle (.1cm) node[above]{F(C)};
 \draw (1,0) -- (3,0);
 \draw[black,fill=black] (3,0) circle (.1cm) node[right]{F(D)};
\end{tikzpicture}
\caption{A sketch of the image $F(M)$.}
\label{fig:Efst-image}
\end{subfigure}\quad
\caption{The Karshon graph and momentum map image for the system described
in Example~\ref{ex:Efst}. In the momentum map image, the images of the fixed points
are indicated with dots, and the curves inside of the interior of the image
are the image of hyperbolic-regular points.}
\label{fig:Efst-example}
\end{figure}
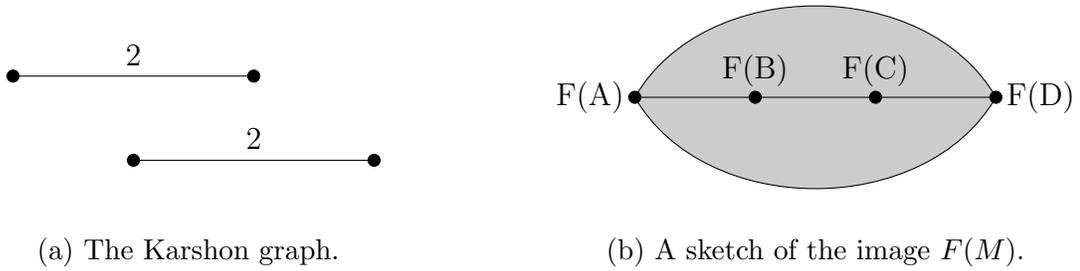


\section{All Hamiltonian \texorpdfstring{$\mathbb{S}^1$}{S1}-spaces can be extended to hypersemitoric systems}
\label{sec:extending}

In this section we prove Theorem~\ref{thm:extending}, which states that any Hamiltonian $\mbS^1$-space can be
extended to a hypersemitoric system.
We prove this making use of Karshon's classification of minimal models.


\subsection{Preparations for the proof}
\label{sec:preparations}
First we will show that all of the minimal Hamiltonian $\mbS^1$-spaces described by Karshon admit an 
extension to a hypersemitoric system.

\begin{proposition}\label{prop:minsystems}
 If $(M,\om,J)$ is a four-dimensional compact Hamiltonian $\mbS^1$-space which does not admit an $\mbS^1$-equivariant
 blowdown then there exists an $H\colon M\to\R$ such that $(M,\om,(J,H))$ is a hypersemitoric system
 with no degenerate singular points.
\end{proposition}

\begin{proof}
 Due to Theorem~\ref{thm:karshon-minimal} there are three classes of minimal Hamiltonian $\mbS^1$-spaces.
 The first two classes are induced by toric actions,
 so by definition these extend to toric systems which are in particular hypersemitoric systems.
 
 The remaining minimal models are spaces $(M,\om,J)$ with two fixed surfaces and no interior points, 
 which are thus $\mbS^2$-bundles over a closed surface $\Sigma$ with an $\mbS^1$-action
 that fixes $\Sigma$ and rotates each fiber (i.e.~ruled surfaces). Denote by $\pi_\Sigma\colon M\to\Sigma$ the projection map. Given such a space, let $f\colon\Sigma\to\R$ be any Morse function on $\Sigma$ and define $H\colon M\to\R$ by $H = f\circ\pi_\Sigma$.
 Notice that $H$ is automatically invariant under the $\mbS^1$-action generated by $\mathcal{X}^J$, and therefore $J$ and $H$ Poisson commute. Since $f$ is Morse we conclude that $H$ and $J$ are linearly independent almost everywhere and thus $(M,\om,F=(J,H))$ is an integrable system.
 
We will now show that $(M,\om,F=(J,H))$ is a hypersemitoric system. First recall from Theorem~\ref{thm:karshon-minimal} that the $\mbS^1$-action on the ruled surfaces is rotation of the $\mbS^2$, which locally is thus generated by the height function on $\mbS^2$. This can also be seen in the explicit realization of these minimal models described in~\cite[proof of Lemma 6.15]{karshon}.
Thus, $J$ is locally of the following form: around any point $p\in M$ there is a
 neighborhood $U$ of the form $U=\mbS^2\times U_{\Sigma}$ where $U_\Sigma\subset\Sigma$
 is an open set and $J|_U = g\circ\pi_{\mbS^2}$, where
 $g$ is the usual height function on $\mbS^2$ and $\pi_{\mbS^2}\colon U\to\mbS^2$ is the projection map.  
 Thus, $F|_U = (g\circ\pi_{\mbS^2},f\circ\pi_\Sigma)$.
 Since $g$ and $f$ are both Morse (cf.\ Remark~\ref{rmk:morse}) this implies that all singular points of $F=(J,H)$ in $U$ are non-degenerate, and thus all singular points of $F$ are non-degenerate, so the system is a hypersemitoric system.
\end{proof}

\begin{remark}
 The image of the hypersemitoric system constructed on a ruled surface in the proof of Proposition~\ref{prop:minsystems} 
 will always be a rectangle, and the images of the hyperbolic-regular points of the system will be horizontal
 lines across the rectangle (as in Figure~\ref{fig:S2T2-image}). So there are two vertical boundary components in the image
 corresponding to the two fixed surfaces in these Hamiltonian $\mbS^1$-spaces, but even though these two edges of the image have the same length this does not mean that the corresponding fixed surfaces
 have the same symplectic area (in fact, they often do not). This is because the hypersemitoric system we construct is not toric so there is no relationship between the length of the edges and the symplectic volume of the corresponding submanifolds.
\end{remark}

\begin{remark}\label{rmk:product-top-bdry} 
In Proposition~\ref{prop:minsystems} we show 
that given a Hamiltonian $\mbS^1$-space with $M=\mbS^2\times\Sigma_g$
where $\Sigma_g$ is a surface of genus $g$ and $J$ rotates the sphere, we can take a Morse function $f$ on $\Sigma$
to produce a hypersemitoric system $(M,\om,(J,H=f\circ \pi_\Sigma))$.
Notice that, making use of action-angle coordinates (and the normal form around elliptic points), we may furthermore choose $f$ so that
$\mathcal{X}^H$ has $2\pi$-periodic flow in a preimage of a neighborhood
of the upper boundary of $F(M)$, and therefore on that neighborhood $(J,H)$
forms an (open) toric system.
This will be useful for us when performing gluing
in the proof of Lemma~\ref{lem:stblowups-hst}.
\end{remark}

Let $\Gamma$ be a Karshon graph. In the following, we call a {\em connected component} of $\Gamma$
briefly a \emph{component of $\Gamma$}.
The \emph{maximum component} 
of $\Gamma$ is the component of $\Gamma$ which contains the vertex corresponding to the maximum
value of the momentum map
and the \emph{minimum component} 
of $\Gamma$ is the component of $\Gamma$ which contains the vertex corresponding to the minimum	
value of the momentum map.

 Recall that the four possible cases of the effect of $\mbS^1$-equivariant blowups on a Karshon graph are described Section~\ref{sec:karshon-blowups} and numbered as type~\ref{case:blowup-surface}-\ref{case:interior-blowup}.

\begin{lemma}\label{lem:stages}
 Let $\Gamma$ be the Karshon graph of a Hamiltonian $\mbS^1$-space $(M,\om,J)$ and 
 let $\Gamma_\mathrm{min}$ denote the minimal Karshon graph from which $\Gamma$ can be obtained 
 by a sequence of $\mbS^1$-equivariant blowups. Then one can in fact obtain $\Gamma$ from $\Gamma_\mathrm{min}$ as follows:
 \begin{itemize}
  \item \textbf{Stage 1:} Perform a sequence of blowups on isolated fixed points in the maximum and minimum components to obtain a new graph $\Gamma'$ from $\Gamma_\mathrm{min}$. All blowups performed in this stage are of types~\ref{case:blowup-min11}, \ref{case:blowup-min}, or~\ref{case:interior-blowup};
  \item \textbf{Stage 2:} Perform a sequence of blowups on points that lie in fixed surfaces to obtain a new graph $\Gamma''$ from $\Gamma'$. All blowups performed in this stage are of type~\ref{case:blowup-surface};
  \item \textbf{Stage 3:} Perform a sequence of blowups on isolated fixed points that correspond to the components of $\Gamma''$ which are not the maximum or minimum components to obtain the desired graph $\Gamma$ from $\Gamma''$. All blowups performed in this stage are of type~\ref{case:interior-blowup}.
 \end{itemize}
\end{lemma}

Before we prove Lemma~\ref{lem:stages}, let us briefly put it into context:

\begin{remark}
 Lemma~\ref{lem:stages} provides the formal guideline for the order that we will perform a sequence of blowups on a given Karshon graph in the proof of Theorem~\ref{thm:extending}. The exact blowups that we perform are described in the proof of Theorem~\ref{thm:extending}. An example of performing blowups in this order will be described in Section~\ref{sec:follow-ex}.
\end{remark}

\begin{proof}[Proof of Lemma~\ref{lem:stages}]
 This lemma is based on the fact that blowups in different components of Karshon graph
 do not interact with each other. As long as this principle is observed one can perform blowups in any order --- thus also the preferred one given in the statement of the lemma we are proving --- without obtaining different spaces in the end.
 
Starting with $\Gamma_\mathrm{min}$ we obtain a new Karshon graph $\Gamma'$ by performing blowups on only the
maximum and minimum components of $\Gamma_\mathrm{min}$ to make them agree with the maximum and minimum components of $\Gamma$
(except possibly for the area labels on the fat vertices, if any).
This is Stage 1.
 For Stage 2, we perform blowups of type~\ref{case:blowup-surface} on the fat vertices 
 of $\Gamma'$ to produce a new graph $\Gamma''$ which now has the same area labels on its fat vertices
 as $\Gamma$ and has the same number of components as $\Gamma$.
 Now the maximum and minimum components of $\Gamma''$ agree with those
 of $\Gamma$, and thus $\Gamma''':=\Gamma$ is obtained from $\Gamma''$ by
 a sequence of blowups of type~\ref{case:interior-blowup} on the components
 of $\Gamma''$ which are not its maximum or minimum components.
 This is Step 3.

 Notice that the only type of blowup from~\ref{case:blowup-surface}--\ref{case:interior-blowup}
 which changes the number of components of the Karshon graph is
 type~\ref{case:blowup-surface}, which produces a new component containing a single vertex.
 Thus, at the start of Step 3, as described in the statement of the lemma,
 all of the non-extremal components of $\Gamma''$ can only be single isolated vertices,
 since all non-extremal components of the minimal models are single vertices
 and the only new components that we produced in Steps 1 and 2 were also 
 components of a single vertex.
\end{proof}

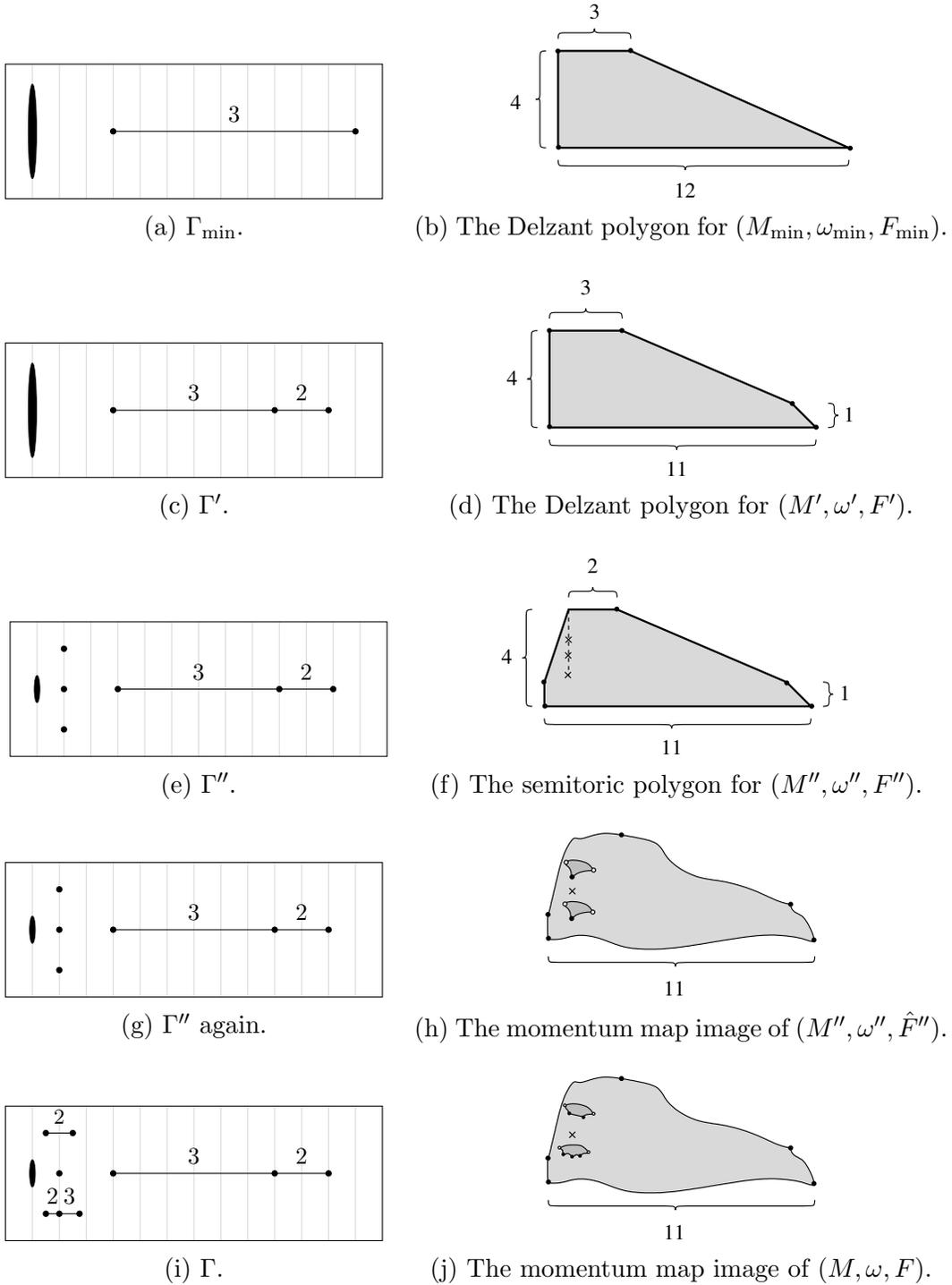
\begin{figure}
\centering
\begin{subfigure}[t]{.38\textwidth}
\centering
\begin{tikzpicture}[scale = .4]
\foreach \x in {0,1,...,12}{
        \draw[black!15!white,black!15!white] (\x,-2.5) -- (\x,2.5);}
 \draw[black,fill=black] (0,0) ellipse (.15cm and 1.75cm);

 \draw (3,0) -- (12,0) node[midway,above]{\footnotesize{3}};

 \draw[black,fill=black] (3,0) circle (.1cm);
 \draw[black,fill=black] (12,0) circle (.1cm);

 \draw (-1,-2.5) rectangle (13,2.5);
\end{tikzpicture}
\caption{$\Gamma_{\mathrm{min}}$.}
\label{fig:Gmin}
\end{subfigure}
\begin{subfigure}[t]{.48\textwidth}
\centering
\scalebox{.9}{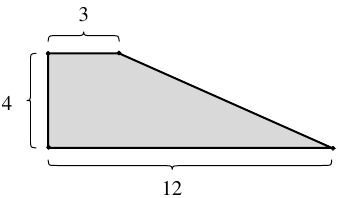}
  \caption{The Delzant polygon for $(M_\mathrm{min},\om_\mathrm{min},F_\mathrm{min})$.}
  \label{fig:hirz-example-toric1}
\end{subfigure}\\[1em]
\begin{subfigure}[t]{.38\textwidth}
\centering
\begin{tikzpicture}[scale=.4]
\foreach \x in {0,1,...,12}{
        \draw[black!15!white,black!15!white] (\x,-2.5) -- (\x,2.5);}

 \draw[black,fill=black] (0,0) ellipse (.15cm and 1.75cm);	

 \draw (3,0) -- (9,0) node[midway,above]{\footnotesize{3}};
 \draw (9,0) -- (11,0) node[midway,above]{\footnotesize{2}};

 \draw[black,fill=black] (3,0) circle (.1cm);
 \draw[black,fill=black] (9,0) circle (.1cm);
 \draw[black,fill=black] (11,0) circle (.1cm);

 \draw (-1,-2.5) rectangle (13,2.5);
\end{tikzpicture}
\caption{$\Gamma'$.}
\label{fig:Gprime}
\end{subfigure}
\begin{subfigure}[t]{.48\textwidth}
\centering
\scalebox{.9}{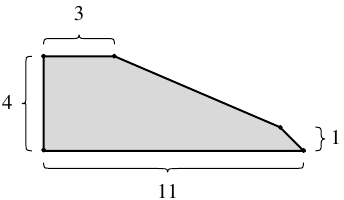}
  \caption{The Delzant polygon for $(M',\om',F')$.}
  \label{fig:hirz-example-toric2}
\end{subfigure}\\[1em]
\begin{subfigure}[t]{.38\textwidth}
\centering
\begin{tikzpicture}[scale=.4]
\foreach \x in {0,1,...,12}{
        \draw[black!15!white,black!15!white] (\x,-2.5) -- (\x,2.5);}
 \draw[black,fill=black] (0,0) ellipse (.1cm and .5cm);	

 \draw[black,fill=black] (1,1.5) circle (.1cm);
 \draw[black,fill=black] (1,0) circle (.1cm);
 \draw[black,fill=black] (1,-1.5) circle (.1cm);

 \draw (3,0) -- (9,0) node[midway,above]{\footnotesize{3}};
 \draw (9,0) -- (11,0) node[midway,above]{\footnotesize{2}};

 \draw[black,fill=black] (3,0) circle (.1cm);
 \draw[black,fill=black] (9,0) circle (.1cm);
 \draw[black,fill=black] (11,0) circle (.1cm);
 \draw (-1,-2.5) rectangle (13,2.5);
 \end{tikzpicture}
\caption{$\Gamma''$.}
\label{fig:Gprimeprime}
\end{subfigure}
\begin{subfigure}[t]{.48\textwidth}
\centering
\scalebox{.9}{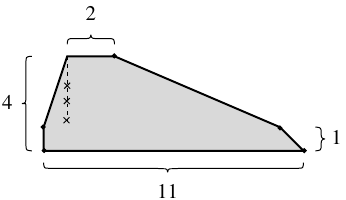}
  \caption{The semitoric polygon for $(M'',\om'',F'')$. }
  \label{fig:hirz-example-semitoric}
\end{subfigure}\\[1em]
\begin{subfigure}[t]{.38\textwidth}
\centering
\begin{tikzpicture}[scale=.4]
\foreach \x in {0,1,...,12}{
        \draw[black!15!white,black!15!white] (\x,-2.5) -- (\x,2.5);}
 \draw[black,fill=black] (0,0) ellipse (.1cm and .5cm);	

 \draw[black,fill=black] (1,1.5) circle (.1cm);
 \draw[black,fill=black] (1,0) circle (.1cm);
 \draw[black,fill=black] (1,-1.5) circle (.1cm);

 \draw (3,0) -- (9,0) node[midway,above]{\footnotesize{3}};
 \draw (9,0) -- (11,0) node[midway,above]{\footnotesize{2}};

 \draw[black,fill=black] (3,0) circle (.1cm);
 \draw[black,fill=black] (9,0) circle (.1cm);
 \draw[black,fill=black] (11,0) circle (.1cm);
 \draw (-1,-2.5) rectangle (13,2.5);
 \end{tikzpicture}
\caption{$\Gamma''$ again.}
\label{fig:Gprimeprimeagain}
\end{subfigure}
\begin{subfigure}[t]{.48\textwidth}
\centering
\scalebox{.9}{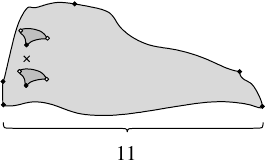}
\caption{The momentum map image of $(M'',\om'',\hat{F}'')$.}
 \label{fig:hirz-example-hypersemi1}
\end{subfigure}\\[1em]
\begin{subfigure}[t]{.38\textwidth}
\centering
\begin{tikzpicture}[scale = .4]
\foreach \x in {0,1,...,12}{
        \draw[black!15!white,black!15!white] (\x,-2.5) -- (\x,2.5);}
 \draw[black,fill=black] (0,0) ellipse (.1cm and .5cm);	

 \draw[black,fill=black] (.5,1.5) circle (.1cm);
 \draw[black,fill=black] (1.5,1.5) circle (.1cm);
 \draw (.5,1.5) -- (1.5,1.5) node[midway,above]{\scriptsize{2}};

 \draw[black,fill=black] (1,0) circle (.1cm);

 \draw[black,fill=black] (.5,-1.5) circle (.1cm);
 \draw[black,fill=black] (1,-1.5) circle (.1cm);
 \draw[black,fill=black] (1.75,-1.5) circle (.1cm);
 \draw (.5,-1.5) -- (1,-1.5) node[midway,above]{\footnotesize{2}};
 \draw (1,-1.5) -- (1.75,-1.5) node[midway,above]{\footnotesize{3}};

 \draw (3,0) -- (9,0) node[midway,above]{\footnotesize{3}};
 \draw (9,0) -- (11,0) node[midway,above]{\footnotesize{2}};

 \draw[black,fill=black] (3,0) circle (.1cm);
 \draw[black,fill=black] (9,0) circle (.1cm);
 \draw[black,fill=black] (11,0) circle (.1cm);

 \draw (-1,-2.5) rectangle (13,2.5);
\end{tikzpicture}
\caption{$\Gamma$.}
\label{fig:G}
\end{subfigure}
 \begin{subfigure}[t]{.48\textwidth}
\centering
\scalebox{.9}{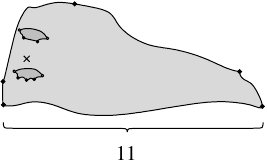}
  \caption{The momentum map image of $(M,\om,F)$.}
  \label{fig:hirz-example-hypersemi2}
\end{subfigure}
\caption{The Karshon graphs (left) and corresponding Delzant polygons, semitoric polygons, and momentum map images (right) for the systems discussed in the example
 from Section~\ref{sec:follow-ex}.
 The Karshon graphs are drawn over lines representing
integer values of the momentum map so the momentum map labels can
be easily read off of the graph.}
\label{fig:big-example}
\end{figure}


\subsection{Following the proof on an example}
\label{sec:follow-ex}

In this section we will sketch the idea of the proof of Theorem~\ref{thm:extending} by illustrating it on a specific example. Let $\Gamma$ denote the graph sketched in Figure~\ref{fig:G}.
We will demonstrate how a hypersemitoric system can be constructed which has $\Gamma$ as
the Karshon graph of its underlying Hamiltonian $\mbS^1$-space. 
The entire process is sketched in Figure~\ref{fig:big-example}.
The Karshon graphs in 
Figures~\ref{fig:Gmin}, \ref{fig:Gprime}, \ref{fig:Gprimeprime}, \ref{fig:Gprimeprimeagain}, and~\ref{fig:G}
correspond to the systems shown in Figures~\ref{fig:hirz-example-toric1},
 \ref{fig:hirz-example-toric2}, \ref{fig:hirz-example-semitoric}, \ref{fig:hirz-example-hypersemi1}, and~\ref{fig:hirz-example-hypersemi2}, respectively.
The eventual proof of Theorem~\ref{thm:extending} in Section~\ref{sec:proof} will follow the same idea
in the general setting, showing that this
process works for all possible Karshon graphs $\Gamma$.

\subsubsection*{Let us now have a closer look at the various steps of the procedure:}
Given $\Gamma$, we know that there is a minimal graph $\Gamma_{\mathrm{min}}$, shown in Figure~\ref{fig:Gmin}, from which $\Gamma$ can be obtained by a sequence of blowups.
The graph $\Gamma_{\mathrm{min}}$ here is the Karshon graph of
the third Hirzebruch surface with the standard $\mbS^1$-action (with scaling such that it has the Delzant polygon as shown in Figure~\ref{fig:hirz-example-toric1}), and can thus
be extended to a toric system.
The idea is to perform the sequence of blowups on $\Gamma_\mathrm{min}$ in three
stages as discussed in Lemma~\ref{lem:stages}, while showing that the property
of being able to be extended to a hypersemitoric system is preserved.

Let $(M_\mathrm{min},\om_\mathrm{min},F_\mathrm{min})$ be the toric
system associated with the minimal model $\Gamma_{\mathrm{min}}$ of $\Gamma$, so the image of $F$ is the polygon
shown in Figure~\ref{fig:hirz-example-toric1} with vertices at $(0,0)$,
$(0,4)$, $(3,4)$, and $(12,0)$.

\emph{Stage 1:} Performing a toric blowup of size $\lambda_1=1$ at the far right fixed point produces the toric system corresponding to the Delzant polygon shown in Figure~\ref{fig:hirz-example-toric2}.
Denote this system by $(M',\om',F')$.
The underlying Hamiltonian $\mbS^1$-space has Karshon graph $\Gamma'$ as
shown in Figure~\ref{fig:Gprime}.

\emph{Stage 2:} Next we perform three wall chops (corresponding to semitoric blowups)
all of the same size $\lambda_2= \lambda_3= \lambda_4=1$ on the fixed 
surface at the minimum of the momentum map. This produces three new focus-focus
points. Let $(M'',\om'',F'')$ be a semitoric system associated to the resulting
marked polygon invariant. A representative
of the semitoric polygon of this system is shown on the left in Figure~\ref{fig:hirz-example-semitoric}
and the Karshon graph for the underlying Hamiltonian $\mbS^1$-space is the graph $\Gamma''$ shown
in Figure~\ref{fig:Gprimeprime}.

\emph{Stage 3:} In order to perform a toric blowup on the three new focus-focus points
we will now use the technique described in Dullin \& Pelayo~\cite{dullin-pelayo}
to transform them into elliptic-elliptic points on small flaps without changing the underlying 
$\mbS^1$\-space (i.e.~performing a supercritical Hamiltonian-Hopf bifurcation).  
The momentum map image of the resulting system $(M'',\om'',\hat{F}''=(J'',\hat{H}''))$ is shown
on the right in Figure~\ref{fig:hirz-example-semitoric}. 
Since the technique of Dullin \& Pelayo 
does not change 
the underlying Hamiltonian $\mbS^1$-space $(M'',\om'',J'')$ (c.f.~\cite[Theorem 1.2]{dullin-pelayo}), the associated
Karshon graph shown in Figure~\ref{fig:Gprimeprimeagain} is still equal to
the graph $\Gamma''$ in Figure~\ref{fig:Gprimeprime}.
Now two of the focus-focus points have been replaced by elliptic-elliptic points, at
which we can perform toric blowups.

Now we perform a toric blowup of size $\lambda_5=1/2$ on the top flap
and a toric blowup of size $\lambda_6=1/2$ on the bottom flap. Finally, we perform
one more blowup of size $\lambda_7=1/4$ on one of the two new elliptic-elliptic points formed
on the bottom flap to produce a hypersemitoric system $(M,\om,F)$. 
Throughout this process we have now performed the sequence of blowups
used to obtain $\Gamma$ from $\Gamma_\mathrm{min}$ on this integrable system,
so the Karshon graph associated to $(M,\om,J)$ is $\Gamma$, as desired.
The proof that these flaps can always be made large enough to admit blowups of the desired
size is the content of Section~\ref{sec:large-flap}.
This completes the example.

\subsection{Semitoric blowups of hypersemitoric systems}
\label{sec:stblowups-hst}

Operations like blowups and blowdowns on toric and semitoric systems can be done by mainly manipulating the associated momentum polytopes since these systems are wholly or sufficiently determined by this invariant according to the classifications of Delzant and Pelayo $\&$ \vungoc. One of the difficulties with hypersemitoric systems is that there is not yet any similar classification. Thus, when we would like to perform a semitoric blowup on a minimal hypersemitoric system in the proof of Theorem~\ref{thm:extending} we cannot work with a polygon invariant to define the
semitoric blowup as we did in Section~\ref{sec:semitoric-blowup}, since there is not yet any such invariant established.

The strategy we will employ follows from the observation that the minimal
Hamiltonian $\mbS^1$-spaces on ruled surfaces can be obtained by gluing a Hirzebruch
surface with a product manifold $\mbS^2\times\Sigma_g$, where $\Sigma_g$ is a surface
of genus $g$.
Thus, we may perform the wall chops (corresponding to semitoric blowups) on the Hirzebruch surface, which
is a toric (and hence also semitoric) integrable system, before gluing it with
$\mbS^2\times\Sigma_g$, to obtain the desired Hamiltonian $\mbS^1$-space. 
This technique allows us to be sure that the hyperbolic-regular points do not interact
with the wall chops.

Roughly, the following lemma explains that a wall chop can be performed without changing the
structure of the preimage of a neighborhood of the bottom boundary of the momentum map image,
which is the region we will use for the gluing, cf.\ Figure~\ref{fig:gluing}.

\begin{lemma}\label{lem:bottom-boundary}
 Let $(M,\om,F)$ be a semitoric system and let $[\De,\vec{c}_i,(+1)_{j=1}^s]$ be its marked
 polygon invariant. 
 Assume that $\De$ has no vertices along the interior of its bottom boundary.
 Suppose that $[\De',(c_1,\ldots,c_{\ell-1},c,c_\ell),(+1)_{j=1}^{s+1}]$ is a
 wall chop of $[\De,\vec{c}_i,(+1)_{j=1}^s]$ and let $(M',\om',F')$ be a semitoric system
 having $[\De',(c_1,\ldots,c_{\ell-1},c,c_\ell),(+1)_{j=1}^{s+1}]$ as its marked polygon invariant.
 Let $g$ and $g'$ be the straightening maps such that $\De = g\circ F$ and $\De' = g'\circ F'$, as in Section~\ref{sec:semitoric}.
 Then there exist open sets $U\subset M$ and $U'\subset M'$ and a symplectomorphism $\phi\colon U\to U'$
 such that
  \begin{itemize}
   \item $F(U)$ is an open neighborhood of $\partial^-(F(M))$ and $F'(U')$ is an open neighborhood
    of $\partial^-(F'(M'))$, as subsets of $F(M)$ and $F'(M')$ respectively,
   \item $g\circ F(U)$ and $g'\circ F'(U')$ do not intersect the cuts or marked points 
    in $\De=g\circ F(M)$ and $\De'=g'\circ F'(M')$,
   \item $\phi$ is equivariant with respect to the $\T^2$-action induced by the toric momentum
    maps $g\circ F|_{U}$ and $g'\circ F'|_{U'}$.
  \end{itemize}
\end{lemma}

\begin{proof}
 By the description of wall chops of polygons in Section~\ref{sec:semitoric-blowup}, notice that $\De$ and
 $\De'$ are equal as sets in a neighborhood of their common bottom boundary.
 Let $\tilde{V}\subset \R^2$ be a convex open neighborhood of the bottom boundary
 (for instance, $\tilde{V}$ could be
 an $\varepsilon$-neighborhood of the bottom boundary in $\R^2$ for sufficiently small $\varepsilon>0$),
 sufficiently small such that $\tilde{V}\cap \De = \tilde{V}\cap \De'$.
 Since we have chosen representatives where all cuts are upwards
 by taking $\tilde{V}$ small enough we may assume that $V$ does not
 intersect any cuts or marked points in $\De$ or $\De'$ .
 Let $U = (g\circ F)^{-1}(V)$ and let $U' = (g'\circ F')^{-1}(V)$.
 Now $(U,\om|_U,(g\circ F)|_U)$ and $(U',\om'|_{U'},(g'\circ F')|_{U'})$ 
 are open toric systems, since
 $(g\circ F)|_U$ and $(g'\circ F')|_{U'}$ each induce a $\T^2$-action.
 By Karshon \& Lerman~\cite[Proposition 6.5]{karshon-lerman} 
 open toric systems are classified
 by their momentum map image if the momentum map is proper onto a convex open set. 
 Since $(g\circ F)|_U$ and $(g'\circ F')|_{U'}$ are proper onto $\tilde{V}$
 they are isomorphic as open toric systems, which means that there exists 
 a $\T^2$-equivariant  symplectomorphism $\phi$ between them, as in the statement.
 The other two points of the theorem are automatic from the choice of $V$ and the construction
 of $U$ and $U'$.	
\end{proof}

The next lemma is of importance for the proof of Theorem~\ref{thm:extending}.

\begin{lemma}\label{lem:stblowups-hst}
 Let $(M_{\mathrm{min}}, \om_{\mathrm{min}}, J_{\mathrm{min}})$ be a minimal
 Hamiltonian $\mbS^1$-space which is a ruled surface as
 in Theorem~\ref{thm:karshon-minimal}, and let $(M,\om,J)$ be a Hamiltonian $\mbS^1$-space obtained by taking a 
 sequence of $k\in\Z_{>0}$ $\mbS^1$-equivariant blowups of the $\mbS^1$-space $(M_{\mathrm{min}}, \om_{\mathrm{min}}, J_{\mathrm{min}})$
 on fixed surfaces (i.e.~those of type~\ref{case:blowup-surface} in Lemma \ref{lem:casesGraphChange}).
 Then there exists $H\colon M\to \R$ such that $(M,\om,(J,H))$ is a hypersemitoric system
 with exactly $k$ focus-focus points and no degenerate points.
\end{lemma}

\begin{figure}[h]
 \centering
 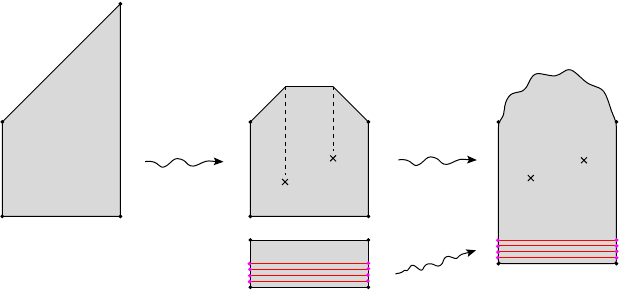
 \caption{The strategy to obtain a hypersemitoric system with the desired underlying Hamiltonian $\mbS^1$-space in the proof
  of Lemma~\ref{lem:stblowups-hst}. We start with a Hirzebruch surface, perform the desired wall chops on the Hirzebruch surface, and then glue it to a system on $\mbS^2\times\Sigma_g$, where $\Sigma_g$ is a surface of genus $g$, to obtain the desired genus of the fixed surfaces.}
 \label{fig:gluing}
\end{figure}

\begin{proof}
Let $\Gamma$ denote the Karshon graph of $(M,\om,J)$.
We will construct a hypersemitoric system whose underlying
Hamiltonian $\mbS^1$-space also has $\Gamma$ as its Karshon graph, and the result
will follow.

The Karshon graph for $(M_{\mathrm{min}}, \om_{\mathrm{min}}, J_{\mathrm{min}})$ consists of two fat vertices a distance of $s>0$ apart (for which we assume without loss of generality that the $J$-values are $0$ and $s$), both with genus $g > 0$, and area label $a>0$ for the vertex at $J=0$ and $a+ns>0$ for the vertex at $J=s$, for some $n\in\Z$.

On the other hand, for some small $\varepsilon>0$, consider the $n^{\mathrm{th}}$ Hirzebruch surface with Delzant polygon given by the convex hull of $(0,0)$, $(s,0)$, $(s,a-\varepsilon+ns)$, $(0,a-\varepsilon)$.
The Karshon graph of the underlying Hamiltonian $\mbS^1$-space of this toric system has two
vertices at $J$-values $0$ and $s$, which have area labels $a-\varepsilon$ and $a-\varepsilon +ns$,
and are both labeled with genus $0$.

By assumption, $(M,\om, J)$ can be obtained from $(M_{\mathrm{min}}, \om_{\mathrm{min}}, J_{\mathrm{min}})$
by a sequence of $\mbS^1$-equivariant blowups of type~\ref{case:blowup-surface} of sizes $\lambda_1^{\mathrm{left}},\ldots,\lambda_{k_{\mathrm{left}}}^{\mathrm{left}}$
on the left fixed surface and of sizes
$\lambda_1^{\mathrm{right}},\ldots,\lambda_{k_{\mathrm{right}}}^{\mathrm{right}}$ on the right fixed surface.
Notice that this implies that $\sum_{i=1}^{k_{\mathrm{left}}}\lambda_i^{\mathrm{left}}-a>0$
and $\sum_{i=1}^{k_{\mathrm{right}}}\lambda_i^{\mathrm{right}}-(a+ns)>0$. Assume $\varepsilon>0$ to be smaller than both of these values.
Now, let $(M',\om',F')$ be a semitoric system whose polygon invariant is obtained from the polygon invariant of the Hirzebruch surface by performing wall chops of sizes $\lambda_1^{\mathrm{left}},\ldots,\lambda_{k_{\mathrm{left}}}^{\mathrm{left}}$
on the left vertical wall and sizes
$\lambda_1^{\mathrm{right}},\ldots,\lambda_{k_{\mathrm{right}}}^{\mathrm{right}}$
on the right vertical wall.
Then the Karshon graph $\Gamma'$ of the underlying $\mbS^1$-space of $(M',\om',F')$ is exactly the same as $\Gamma$ except 
for the labels on the fat vertices: the area labels on $\Gamma'$ are smaller than those of $\Gamma$ by $\varepsilon$,
and both have genus $0$ instead of genus $g$.

We will now change this system by gluing another system along its bottom to obtain the desired Karshon graph. Notice that by Lemma~\ref{lem:bottom-boundary} this system is still toric in a neighborhood of the bottom boundary.

By Proposition~\ref{prop:minsystems}, there exists a hypersemitoric system
on $(M_\varepsilon,\om_\varepsilon, F_{\varepsilon})$
with $M_{\varepsilon}=\mbS^2\times\Sigma_g$, where $\Sigma_g$ is a surface of genus $g$, whose Karshon graph consists of two fat vertices
at $J$-values $0$ and $s$, which both are labeled
with area $\varepsilon$ and genus $g$.
Moreover, as in Remark~\ref{rmk:product-top-bdry} we may assume that the
system is toric in the preimage under $F_{\varepsilon}$ of a neighborhood 
of the top boundary of $F(M)$.
To complete the proof, we will glue this system to the bottom
of the system described above, as sketched in Figure~\ref{fig:gluing}.
This is essentially the simplest case of symplectic gluing, we
will describe it briefly here, but the full details and the general
situation can be found in Gompf~\cite{Gompf-gluing}.

The system $(M',\om',F')$ is toric on the preimage $U'\subset M'$ of a neighborhood of the bottom boundary of $F'(M')$, 
and in particular we can choose 
\[
U'=F^{-1}(\{(x,y)\in\R^2\mid 0\leq x \leq s, 0 \leq y<R\})
\]
for a suitable $R>0$.
Denote by $\mathbb{D}_R\subset \C$ the open disk of radius $R$ and consider $\mbS^2\times\mathbb{D}_{R}$ with integrable system $F'':=(z,\rho)$ where $z$ is the height coordinate on $\mbS^2$ (appropriately scaled)
and $\rho$ is the radial coordinate on $\mathbb{D}_R$.
Then this system $F''$ on $\mbS^2\times\mathbb{D}_{R}$ has the same moment map image as
the open toric system on $U'$ inherited from $(M',\om',F')$,
and thus by Karshon \& Lerman~\cite[Proposition 6.5]{karshon-lerman} they are $\mathbb{T}^2$-equivariantly symplectomorphic.
Similarly, the local model for the preimage of the top boundary of $F_{\varepsilon}(M_{\varepsilon})$
is $U_{\varepsilon} =  \mbS^2\times\mathbb{D}_{r}$ with $F_{\varepsilon}=(z_{\varepsilon},-\rho_{\varepsilon})$
where $z_{\varepsilon}$ is the height coordinate on $\mbS^2$ (appropriately scaled)
and $\rho_{\varepsilon}$ is the radial coordinate on $\mathbb{D}_{r}$.
These two regions, after removing the center of each disk, may be 
embedded into a region
modeled by $\mbS^1\times \,\, ]-r,R[ \, \times\mbS^2$ with momentum map
$F=(z,\rho)$ where $z$ is as before and $\rho\in \, ]-r,R[$.
This can be used to glue the regions together, and therefore glue
together the integrable systems $(M',\om',F')$ and $(M_{\varepsilon},\om_{\varepsilon},F_{\varepsilon})$,
to obtain a new integrable system which is also a hypersemitoric system.
This process does not introduce any new $\Z_k$-spheres, and the fixed surfaces in the new
system are the connected sum of the fixed surfaces of the original two systems,
and thus they are surfaces of genus $g$ with areas $a$ and $a+ns$. Therefore,
this system has the desired Karshon graph.
\end{proof}


\subsection{The size of flaps related to Hamiltonian-Hopf bifurcations}
\label{sec:large-flap}

For the `model proof' in Section~\ref{sec:follow-ex} to work in full generality we need to show that the necessary flaps (cf.\ Definition \ref{def:flap}) can be made large enough to contain the required blowups as in {\em Stage 3} in Section \ref{sec:follow-ex}. The idea worked out in this section is as follows: first, we produce small flaps using Dullin $\&$ Pelayo's technique (cf.\ Example \ref{ex:dullin-pelayo}) and then we prove that we can enlarge each small flap to the necessary size.



\subsubsection{Some notation on $\mbS^1$-equivariant blowups of $\mbS^1$-spaces}
\label{sec:blowupNotation}

A vertex of a Karshon graph is \emph{isolated} if it is a non-fat vertex
which is not connected to any edges. Note that in this case the absolute value of both weights of the fixed point associated with this vertex is one. For instance, vertices associated to focus-focus points are always isolated.

Let $(M,\om,J)$ be a Hamiltonian $\mbS^1$-space, $\Gamma$ the associated Karshon graph, and let $v_0$ be an isolated vertex in $\Gamma$ which is not at the maximum or minimum value of $J$.
Now set $\Gamma_0 := \Gamma$ and suppose that $\Gamma_1$ is a Karshon graph obtained from $\Gamma_0$ by performing an $\mbS^1$-equivariant blowup at $v_0$. Similarly, let $\Gamma_2,\ldots,\Gamma_k$ be Karshon graphs such that for each $1\leq\ell\leq k$, $\Gamma_\ell$ is obtained from $\Gamma_{\ell-1}$ by an $\mbS^1$-equivariant blowup. Further suppose that, at each step, the $\mbS^1$-equivariant blowup is performed on one of the {\em new} vertices created by this process. Note that therefore all blowups will effect the same connected component of the graph. Thus, $\Gamma_k$ is equal to $\Gamma_0$ except that the vertex $v_0$ is replaced by a more complicated component. 
In such a situation we say that $\Gamma_k$ is obtained from $\Gamma$ by a \emph{sequence of blowups at $v_0$}.

For $\ell\in\{0,\ldots,k-1\}$ let $(M_\ell,\om_\ell,J_\ell)$ be the Hamiltonian $\mbS^1$-space associated to $\Gamma_\ell$.
In the remainder of this section, we will introduce an open 
set $S\subset M$ which we will call a \emph{blowup set} 
of this sequence of blowups. 
The blowup set will be important for us because when performing the sequence
of $\mbS^1$-equivariant blowups at $v_0$ the symplectic manifold $M$ is only changed on
 the set $\overline{S}$, the closure of the set $S$.
That is, there is a natural symplectomophism from an open dense subset of $M_k$
onto $M\setminus \overline{S}$.
Thus, any structure (such
as an integrable system) on $M$ can be passed automatically to $M_k$ away from $S$, and it is
only around $S$ that we must worry about adapting the structure to the new manifold.

Now we will define the set $S$ which can roughly be thought of in the following way: for each blowup an equivariantly embedded ball is removed and $S$ is the union of the images of these balls mapped back to $M$. There is an example in Figure~\ref{fig:blowup-set}.
Performing an $\mbS^1$-equivariant blowup on $\Gamma_\ell$ corresponds to an $\mbS^1$-equivariant blowup on $(M_\ell,\om_\ell, J_\ell)$ which (as discussed in Section~\ref{sec:S1-blowups}) is performed by removing an $\mbS^1$-equivariantly embedded ball from $(M_\ell,\om_\ell,J_\ell)$ and collapsing the boundary. 
To make this more precise, for any $r>0$ let $B^4(r)$ denote the ball of radius $r$, and let $\phi_\ell\colon B^4(r_\ell)\hookrightarrow M_{\ell}$ be a choice of such
an $\mbS^1$-equivariant embedding (for some $r_\ell>0$) and let 
\[
 S_\ell := \phi_\ell(B^4(r_\ell))\subset M_\ell.
\]
Then, $M_{\ell+1}$ is defined as the quotient space
\[
 M_{\ell+1} = \left(M_\ell \setminus S_\ell \right) / \sim_\ell,
\]
where $\sim_\ell$ is the identity relation on $M_\ell \setminus \overline{S_\ell}$
and identifies two points of $\partial(S_\ell)\cong \mbS^3$
if and only if they are in the same fiber of a choice of Hopf fibration.
Thus, for each $0 < \ell \leq k$ the quotient map 
\begin{equation}\label{eqn:blowupmap}
 \psi_\ell\colon M_{\ell-1}\setminus S_{\ell-1} \to M_\ell
\end{equation}
is a continuous map and $\psi_1^{-1}\circ\ldots\circ\psi_\ell^{-1}(S_\ell)\subset M$.
Now, define $S\subset M$ by 
 \begin{equation}\label{eqn:blowupset}
  S := S_{0, \dots, k}:= S_0 \cup \left(\bigcup_{\ell=1}^k (\psi^{-1}_1\circ\ldots\circ\psi^{-1}_\ell)(S_\ell)\right).
 \end{equation}
\begin{definition}\label{def:blowupset}
The set $S= S_{0, \dots, k}\subset M$ from the above construction is called a 
\emph{blowup set at $p$} associated to the sequence of blowups used to obtain $\Gamma_k$ from $\Gamma = \Gamma_0$,
and for each $\ell$ the set $\Sigma_\ell := \left(\partial(S_{\ell-1})/\sim_{\ell-1} \right)\subset M_\ell$ is called the
\emph{exceptional divisor}.
\end{definition}
Note that the blowup set depends on the choice of $\mbS^1$-equivariant embeddings
$\phi_\ell$ used to perform the blowups, and there are different $\mbS^1$-equivariant embeddings
which produce isomorphic $\mbS^1$-spaces when using those embeddings to
perform $\mbS^1$-equivariant blowups.

\begin{figure}
\begin{center}
\begin{subfigure}[t]{.90\textwidth}

\begin{center}
\begin{tikzpicture}[scale = .5]
\filldraw[draw=black, fill=gray!60] (0,0) node[anchor=north,color=black]{}
  -- (0,3) node[anchor=south,color=black]{}
  -- (3,3) node[anchor=north,color=black]{}
  -- (6,0) node[anchor=north,color=black]{}
  -- cycle;
  
\draw[black,fill=black] (0,0) circle (.07cm);
\draw[black,fill=black] (0,3) circle (.07cm);
\draw[black,fill=black] (3,3) circle (.07cm);
\draw[black,fill=black] (6,0) circle (.07cm);

\draw [->] (6.15,1) -- (6.75,1);

\begin{scope}[xshift = 200pt]

\filldraw[draw=black, fill=gray!60] (0,0) node[anchor=north,color=black]{}
  -- (0,3) node[anchor=south,color=black]{}
  -- (2,3) node[anchor=north,color=black]{}
  -- (4,2) node[anchor=north,color=black]{}
  -- (6,0) node[anchor=north,color=black]{}
  -- cycle;

\draw[black,fill=black] (0,0) circle (.07cm);
\draw[black,fill=black] (0,3) circle (.07cm);
\draw[black,fill=black] (2,3) circle (.07cm);
\draw[black,fill=black] (4,2) circle (.07cm);
\draw[black,fill=black] (6,0) circle (.07cm);

\draw [->] (6.15,1) -- (6.75,1);

\begin{scope}[xshift = 200pt]
  
\filldraw[draw=black, fill=gray!60] (0,0) node[anchor=north,color=black]{}
  -- (0,3) node[anchor=south,color=black]{}
  -- (2,3) node[anchor=north,color=black]{}
  -- (3,2.5) node[anchor=north,color=black]{}
  -- (4.5,1.5) node[anchor=north,color=black]{}
  -- (6,0) node[anchor=north,color=black]{}
  -- cycle;

\draw[black,fill=black] (0,0) circle (.07cm);
\draw[black,fill=black] (0,3) circle (.07cm);
\draw[black,fill=black] (2,3) circle (.07cm);
\draw[black,fill=black] (3,2.5) circle (.07cm);
\draw[black,fill=black] (4.5,1.5) circle (.07cm);
\draw[black,fill=black] (6,0) circle (.07cm);

\end{scope}
\end{scope}
\end{tikzpicture}
\end{center}
\caption{A sequence of blowups on the first Hirzebruch surface represented by corner
chops on the associated Delzant polygon. The images of the elliptic-elliptic
points (i.e.~the vertices of the polygon) are marked by black dots for better visibility.}
\label{fig:blowup_seta}
\end{subfigure}\quad
\begin{subfigure}[t]{.90\textwidth}
\begin{center}
\begin{tikzpicture}

\filldraw[draw=none, fill=gray!60] (0,0) node[anchor=north,color=black]{}
  -- (0,3) node[anchor=south,color=black]{}
  -- (3,3) node[anchor=north,color=black]{}
  -- (6,0) node[anchor=north,color=black]{}
  -- cycle;

\filldraw[draw=gray!90, fill=gray!90] (2,3) node[anchor=north,color=black]{}
  -- (3,3) node[anchor=south,color=black]{}
  -- (4.5,1.5) node[anchor=south,color=black]{}
  -- (3,2.5) node[anchor=north,color=black]{}
  -- cycle;  

\filldraw[draw=black, fill=none] (0,0) node[anchor=north,color=black]{}
  -- (0,3) node[anchor=south,color=black]{}
  -- (3,3) node[anchor=north,color=black]{}
  -- (6,0) node[anchor=north,color=black]{}
  -- cycle;
  
\draw (2,3) -- (3,2.5);
\draw (3,2.5) -- (4.5,1.5);

\draw[black,fill=black] (2,3) circle (.03cm);
\draw[black,fill=black] (3,3) circle (.03cm);
\draw[black,fill=black] (4.5,1.5) circle (.03cm);
\draw[black,fill=black] (3,2.5) circle (.03cm);

\draw [dashed] (2,-.5) -- (2,3.5);
\draw [dashed] (4.5,-.5) -- (4.5,3.5);


\draw (2, -.5) node[anchor=north]{$j=a$};
\draw (4.5, -.5) node[anchor=north]{$j=b$};

\end{tikzpicture}
\end{center}
\caption{The Delzant polygon for the first Hirzebruch surface with the image of the blowup set
for the blowups in Figure~\ref{fig:blowup_seta} shaded in. Notice that the shaded region is \emph{not}
a triangle, it is a four-sided polygon (we have marked its vertices with black dots).
The $j$ parameter from Lemma~\ref{lem:reduced-blowup-set} parameterizes the image of the blowup
set in the reduced space, and for each fixed value of $j$ the image in the reduced space is 
homeomorphic to a disk.}
\label{fig:blowup_setb}
\end{subfigure}
\end{center}
\caption{The blowup set of a sequence of blowups is essentially the set
of points removed during the blowing up process, as is shown in this example on
the first Hirzebruch surface.}
\label{fig:blowup-set}
\end{figure}


\subsubsection{Two technical lemmas}
\label{sec:technicalLemmas}

The following two statements formulate and prove technical results that will be needed in the proof of Proposition~\ref{prop:large-flap}.
As in Section~\ref{sec:reduction}, we denote the quotient of $M$ by the $\mbS^1$-action by $\hat{M}=M/\mbS^1$ with
quotient map $\pi\colon M\to \hat{M}$, and we denote the
symplectic quotient at level $J=j$ by $\hat{M}_j=J^{-1}(j)/\mbS^1$. Furthermore, we let $\mathrm{smooth}(\hat{M})$ be those points
in $\hat{M}$ which correspond to orbits in $M$ on which $\mbS^1$ acts freely (and thus $\mathrm{smooth}(\hat{M})$ inherits a smooth structure) and we
set $\mathrm{sing}(\hat{M}) = \hat{M}\setminus\mathrm{smooth}(\hat{M})$.

Given an integrable system $(M,\om,(J,H))$ where $J$ generates an
$\mbS^1$-action, the intuition behind the following lemma is that, as long as one stays away from the
non-free orbits of the $\mbS^1$-action,
it is not difficult to change $H$ by prescribing its behavior on the quotient $\hat{M}$.



\begin{lemma}\label{lem:reduction-extending}
Let $(M,\om,(J,H))$ be a completely integrable system such that $J$ generates an effective $\mbS^1$-action. 
Then 
\begin{enumerate}
 \item 
 There exists a function $\hat{H}\colon \hat{M}\to\R$ such that $\hat{H}\circ \pi = H$ and such that the restriction $\hat{H}|_{\mathrm{smooth}(\hat{M})}: \mathrm{smooth}(\hat{M}) \to \R$ is smooth.
 \item
 Let $\hat{U}\subset \hat{M}$ be an open neighborhood of $\mathrm{sing}(\hat{M})$. Furthermore, let $\hat{H}'\colon \hat{M}\to\R$ be a function which is smooth on $\mathrm{smooth}(\hat{M})$ and equal to $\hat{H}$ on $\hat{U}$, and assume that $\mathrm{d}\hat{H}'$ is non-zero almost everywhere on $\hat{M}$. 
 Then, setting $H':=\hat{H}'\circ\pi$, the tuple $(M,\om, (J,H'))$ forms a completely integrable system.
\end{enumerate}
\end{lemma}

\begin{proof}
The integral $H$ descends to a map $\hat{H}$ on $\hat{M}$ since $H$ is constant under the flow of $J$, which generates the $\mbS^1$-action, and
$\hat{H}$ is smooth on $\mathrm{smooth}(\hat{M})$ since the action of $\mbS^1$ is free on $\mathrm{smooth}(\hat{M})$.

Next suppose that $\hat{U}$ and $\hat{H}'$ are as in the statement
and let $H':=\hat{H}'\circ\pi$. The conditions for the function $H'$ to form an integrable system with $J$ are all local, so there is nothing to check in $\pi^{-1}(\hat{U})$
since $H=H'$ in that set. 
By definition, the function $H':=\hat{H}'\circ\pi$ is constant on the orbits of the $\mbS^1$-action since they are the orbits of the flow of $J$, so $H'$ and $J$ Poisson commute.
For any $p\in M$, note that $\mathrm{d}H'\wedge \mathrm{d}J(p) = 0$ implies that $\mathrm{d}\hat{H}'(\pi(p))=0$. By assumption, $\mathrm{d}\hat{H}'$ is non-zero almost everywhere in $\hat{M}$, so $\mathrm{d}H'$ and $\mathrm{d}J$ are linearly independent almost everywhere in $M$.
\end{proof}

Denote by $\mathbb{D}_r := \{z\in\C\mid |z|^2<r^2\}$ the standard disk of radius $r>0$. 
The purpose of the following statement is to show that the blowup set of a sequence
of blowups descends in the reduced space to a set which can be parameterized as a disk
times an open interval, in which the radius of the disk varies along the interval
and shrinks to a point at each endpoint of the blowup set,
see Figure~\ref{fig:blowup-set}.
This result will enable us to work locally on the disk when editing integrable systems
near a blowup set in Proposition~\ref{prop:large-flap}.

\begin{lemma}\label{lem:reduced-blowup-set}
Let $(M,\om,J)$ be a Hamiltonian $\mbS^1$-space and let $p\in M$ be a fixed point with weights $m$ and $-n$ (where $m,n \in\Z_{> 0}$) and set $j_p:=J(p)$.
Let $S\subset M$ be a blowup set at $p$ as in Definition~\ref{def:blowupset}.
Let $\hat{S}:=S/\mbS^1$ and $\hat{M}:=M/\mbS^1$. 
Then there exist $a,b\in\R$ such that $J(S)= \ ]a,b[$. Moreover, for each $j\in\,]a,b[$ there exists $R_j>0$ so that $\hat{S} = \rho(\mathcal{A})$ 
  where
 \[
  \mathcal{A} := \mathcal{A}(a,b,\{R_j\}_{j\in]a,b[}) := \{(z,j)\in\C\times\, ] a,b [\,\,\, \mid z\in\mathbb{D}_{R_j}\}
 \]
 and $\rho\colon \mathcal{A} \to \hat{M}$ is a continuous map satisfying:
 \begin{enumerate}
  \item $\rho(z,j)\in J^{-1}(j)$ for all $(z,j)\in\mathcal{A}$;
  \item $\hat{S}\,\cap\,\mathrm{smooth}(\hat{M}) = \hat{S}\setminus \{\pi(p)\}$
   and the restriction of $\rho$ to $\mathcal{A}\setminus\{(0,j_p)\}$ is a diffeomorphism onto $\hat{S}\setminus\{\pi(p)\}$;
  \item the assignment $j\mapsto R_j$ is continuous and piecewise smooth, and  
  \[
   \lim_{\substack{j\to a\\ j>a}}R_j = \lim_{\substack{j\to b\\ j<b}} R_j = 0.
  \]
    \end{enumerate}
\end{lemma}

\begin{proof}
We proceed by induction on the number of blowups.

{\em Base case:} Suppose that $S$ is the blowup set of a \textit{single} blowup of size $r$ at a fixed point of the $\mbS^1$-action with weights $m$ and $-n$ where $m,n\in\Z_{>0}$. 
Then $S$ is the image of an $\mbS^1$-equivariantly embedded ball. This
embedding induces local coordinates $(z,w)$ on $S$ which can be extended to some $U\subset M$ with $S\subset U$, such that the $\mbS^1$-action is given by $\lambda\cdot (z,w) = (\lambda^m z, \lambda^{-n} w)$ for $\lambda\in\mbS^1$ with momentum map $J(z,w) = \frac{1}{2}(m|z|^2-n|w|^2)$. 
Without loss of generality, we may assume the coordinates $(z,w)$ to be centered at this ball, meaning $S = \{(z,w)\in\C^2 \mid |z|^2+|w|^2 < r^2\}$.

Choosing $H(z,w):=\frac{1}{2}|w|^2$ yields a toric system $(J,H)$ on $U$. Take some level set $J^{-1}(j)$.
For the remainder of the proof we assume without loss of generality that $j_p=0$.
We will now show that $\hat{S}_j := (J^{-1}(j)\cap S)/\mbS^1\subset J^{-1}(j)/\mbS^1$ is diffeomorphic
to a disk when $j\neq 0 = j_p$.

It is sufficient to consider the case $j > 0$ since the case $j < 0$ goes analogously.
By the definition of $J$ and $j$, we have $m|z|^2-n|w|^2=2j>0$ implying $z\neq 0$, since $j\neq 0$.
Using the $\mbS^1$-action we may assume that $z\in\R^{>0}$, which yields coordinates $z\in\R^{>0}$ and $w\in\C$ on $J^{-1}(j)/\mbS^1$.
Solving for $z$ under the assumption that $J(z,w)=j$ yields $z = \sqrt{\frac{2j+n|w|^2}{m}}$. From $|z|^2+|w|^2<r^2$ we obtain eventually
 \[0\leq |w|^2 < R_j \quad \textrm{ where } \quad R_j:=\frac{r^2-\frac{2j}{m}}{\frac{n}{m}+1}.\]
Notice that $R_j$ varies continuously with $j$ and $R_j\to 0$ as $j \to \frac{1}{2} mr^2$. 
Let $\mathcal{A}_{j>0} = \{(z,j)\in\mathcal{A} \mid j>0\}$.
We can now partially define $\rho$ by defining its restriction $\rho|_{\mathcal{A}_{j>0}}\colon \mathcal{A}_{j>0}\to\hat{S}$ 
by 
\[
 \rho(w,j) = \left( \sqrt{\frac{2j+n|w|^2}{m}}, w\right),
\]
which is surjective onto $\hat{S}\cap J^{-1}(\R_{>0})$.
The same technique can be used to define $\rho$ for $j<0$. The same argument works when $j=0$ except that the
choice of coordinates will not be smooth at $z=0$ when $j=0$.
In these coordinates, $z=0$ and $j=0$ correspond to $\pi(p)$.
Since $\rho$ varies smoothly with $j$, we conclude that $\rho|_{\mathcal{A}\setminus\{(0,j_p)\}}$
is a diffeomorphism onto $\hat{S}\setminus\{\pi(p)\}$.
Therefore, the lemma is proved in the case that $S$ is the blowup set of a single blowup.

{\em Inductive step:}  Now suppose that $\tilde{S}=\tilde{S}_{0, \dots, k}$ is a blowup set obtained from taking $k>0$ blowups and has the properties \emph{(1)-(3)} described in the statement of this lemma.

Let $S=S_{0, \dots, k, k+1}$ be a blowup set obtained by taking the same sequence of blowups as for $\tilde{S}$ plus one additional blowup. 
We will prove that for each $j\in J(\tilde{S}_{k+1})$ the image of $\tilde{S}_{k+1}$ in the reduced space at level $j$ is diffeomorphic to a disk in a way that varies smoothly with $j$, except at $(z,j)=(0,j_p)$.
After performing the first $k$ blowups, the $(k+1)^{st}$ blowup is done by removing an $\mbS^1$-equivariantly embedded ball, whose image $S_{k+1}$ has the desired properties \emph{(1)-(3)} as in the statement from the result of the base case.
Let 
\[
 \psi := \psi_{k+1} \circ \ldots \circ \psi_1
\]
be a composition of the maps from Equation~\eqref{eqn:blowupmap}, so that $\psi$ maps $M$ to $M$ blown up $k+1$ times,
and set $\tilde{S}_{k+1}:=\psi^{-1}(S_{k+1})$. Then, following
Equation~\eqref{eqn:blowupset}, the blowup set $S$ is obtained as the union $S=\tilde{S} \cup\ \tilde{S}_{k+1}$.
For any value of $j$ such that $\tilde{S}\cap J^{-1}(j)=\varnothing$ or $\tilde{S}_{k+1}\cap J^{-1}(j)=\varnothing$ there is
nothing to prove, since from the base case (respectively, by the inductive hypothesis) we already know that such a set is a disk which varies smoothly with $j$ (i.e.~it has properties \emph{(1)-(3)} from the statement
of the present lemma).

Now suppose that $J^{-1}(j)$ intersects both $\tilde{S}$ and $\tilde{S}_{k+1}$. Since this sequence of blowups started at a focus-focus point, both weights at a focus-focus point have absolute value 1. Furthermore,
$\mbS^1$-equivariant blowups on points with non-zero weights (which are necessarily of type~\ref{case:interior-blowup}) always produce new fixed points which have one
weight with a strictly higher absolute value.
Thus, we conclude that
the $(k+1)^{\mathrm{st}}$ blowup must have occurred at a point for which one of the weights has absolute value strictly greater than 1, call it $\eta\in\Z$.
The $\Z_\eta$-sphere emanating from that point is the exceptional divisor from one of the earlier blowups, and since the
$\Z_k$-spheres (for $|k|>1$) emanating from a point are always included in the interior of a blowup of that point, the set $S_{k+1}$ includes this exceptional divisor in its interior at the level $J=j$.
Taking the image in the reduced space at level $j$, this implies that the image of $\tilde{S}_{k+1}$ in the reduced space at level $j$
 is an annulus which is perfectly filled in by the image of $\tilde{S}$ in the reduced space (which corresponds to the $\Z_\eta$-sphere),
 and thus it is a disk, as desired.
\end{proof}

\subsubsection{Resizing flaps}

The aim of the following statement is to turn blowups happening around a focus-focus point into blowups happening around an elliptic-elliptic point by transforming the original system into a system which includes a flap. The main difficulty is making sure that the flap is large enough to include the entire set on which the blowups occur.

\begin{definition}\label{def:on-flap}
 Suppose that $(M,\om,F)$ is an integrable system which includes a flap (as described in Definition~\ref{def:flap}) 
 such that there is exactly one elliptic-elliptic singular point $p\in M$ on the flap.
Let $\tau\colon M\to B$ be the singular Lagrangian fibration induced by $F$ as defined in Section~\ref{sec:fibration}. 
 Then an open connected set $S\subset M$ is \emph{on the flap} if
 \begin{enumerate}
	  \item $S$ does not contain any hyperbolic-regular or degenerate singular points, and 
  \item $\tau(S)$ and $p$ are in the same connected component of $B\setminus \overline{\tau (M^{\mathrm{HR}})}$,
  where $M^{\mathrm{HR}}\subset M$ denotes the set of hyperbolic-regular points of $M$.
 \end{enumerate}
\end{definition}
This is illustrated in Figure \ref{fig:trick} on the very right and Figure \ref{fig:hirz-example-hypersemi2}, which show systems for which
blowups have been performed `on the flaps'.

 
\begin{proposition}\label{prop:large-flap}
Let $(M,\om,F=(J,H))$ be a compact integrable system such that $J$ generates an effective $\mbS^1$-action and let $\Gamma$ be the Karshon graph of the Hamiltonian $\mbS^1$-space $(M,\om,J)$. 
Let $p\in M$ be a focus-focus singular point of $(M,\om,F)$ and let $v_p$ be the corresponding isolated vertex in $\Gamma$. 
Suppose that a Karshon graph $\Gamma_k$ is obtained from $\Gamma=:\Gamma_0$ by a sequence of $k> 0$ blowups at $v_p$. Let $S:= S_{0, \dots, k}\subset M$ be a choice of blowup set (as in Definition~\ref{def:blowupset})
for this sequence of blowups, and further suppose that an open neighborhood of $\overline{S}$ contains no degenerate singular points of the integrable system. Then for any open neighborhood $U$ of $\overline{S}$ there exists an integrable system $(M,\om,\widetilde{F}=(J,\widetilde{H}))$ satisfying the following items:
 \begin{enumerate}
  \item $\widetilde{H}$ and $H$ coincide outside of $U$,
  \item $p$ is an elliptic-elliptic singular point for $(M,\om,\widetilde{F})$,
  \item\label{item:on-the-flap} compared to $(M,\om,F)$, the system $(M,\om,\widetilde{F})$ contains a new flap around $p$, and 
  $S$ lies entirely on this new flap as defined in Definition~\ref{def:on-flap}.
 \end{enumerate}
 Furthermore, if $(M,\om,F)$ was a hypersemitoric system, then $(M,\om,\widetilde{F})$ is also a hypersemitoric system.
\end{proposition}


\begin{proof}

First we outline the structure of the proof.
After setting up some notation and arranging the system in a suitable way in \emph{Step 0} and \emph{Step 1}, in \emph{Step 2} we use the technique of Dullin \& Pelayo  to change the function $H$ so that $p$ is no longer a focus-focus point but is instead an elliptic-elliptic point on a small flap.
Then, in \emph{Step 3}, we view the integrable system in the reduced space of the $\mbS^1$-action locally around $p$ as a family of Morse functions on the disk (using Lemma~\ref{lem:reduced-blowup-set}). In \emph{Step 4} we edit these Morse functions in such a way that that when we lift them back up to an integrable system in \emph{Step 5}, the flap has been enlarged to include the entire blowup set. 
An important part of the strategy of the proof is that after creating the small flap in \emph{Step 2}, we only edit the system in regions where the $\mbS^1$-action is free (in particular avoiding a neighborhood of $p$). 
This is because where the $\mbS^1$-action is free, by Lemma~\ref{lem:reduction-type}, there is a straightforward relationship between the behavior at the singular points of the integrable system and the behavior of the critical points of the induced function on the reduced space.

\textbf{Step 0:} \emph{Set up and notation.}
Let $p\in M$ be a focus-focus point, set $j_p:=J(p)$,
let $S=S_{0, \dots, k}\subset M$ be a blowup set at $p$, and let $\overline{S}$ denote the closure of $S$.
Let $U$ be an open neighborhood of $\overline{S}$.
Suppose that the blowups used to produce the set $S$ have sizes $\lambda_1,\ldots,\lambda_k$. Then there exists an $\varepsilon>0$ such that the blowup set $S_\varepsilon$ for a sequence of blowups of sizes $\lambda_1+\varepsilon, \ldots,\lambda_k+\varepsilon$ is still contained in the set $U$. Without loss of generality we may shrink $U$ so that $U=S_\varepsilon$, and note that this means that $U$ is closed under the $\mbS^1$-action.

As in Lemma~\ref{lem:reduction-extending}, the function $H\colon M\to\R$ induces a function $\hat{H}_j\colon \hat{M}_j\to\R$ on each reduced space $\hat{M}_j:=(M/\!\!/\mbS^1)_j$ satisfying $\hat{H}_j\circ\pi_j = H$.  Set $\hat{U}_j = (J^{-1}(j)\cap U)/\mbS^1\subset \hat{M}_j$. 
The only $\Z_k$-spheres (with $k>1$) which can occur in a blowup set at a point $q$ are those which emanate from $q$. Thus, since $U$ is a blowup set around a focus-focus point, there are no $\Z_k$-spheres (with $k>1$) which intersect $U$ and there are no fixed points of the $\mbS^1$-action in $U$, with the exception of the focus-focus point $p$. 
Thus, for $j\neq j_p$ the action of $\mbS^1$ on each point of $J^{-1}(j)\cap U$ is free 
and by assumption $U$ can be chosen to be small enough so that there are no degenerate points of the
integrable system in $U$ so, by Lemma~\ref{lem:reduction-type}, $\hat{U}_j\subset \mathrm{smooth}(\hat{M}_j)$ and the functions $\hat{H}_j|_{\hat{U}_j}$ are Morse.

Recall the projection $\pi\colon M\to \hat{M} = M/\mbS^1$ and let $[p]=\pi(p)\in \hat{M}_{j_p}$ be the image of the focus-focus point $p$ in the reduced space. Then $[p]\notin \mathrm{smooth}(\hat{M}_j)$, but $\hat{U}_{j_p}\setminus \{[p]\}\subset\mathrm{smooth}(\hat{M}_{j_p})$, and $\hat{H}_{j_p}|_{\hat{U}_{j_p}\setminus\{[p]\}}$ is Morse.

\textbf{Step 1:} \emph{We show that we may assume that the Morse functions have no critical points when restricted to $\hat{U}_j$, for $j\neq j_p$, or $\hat{U}_{j_p}\setminus\{[p]\}$, for $j=j_p$.}
Recall that $U$ is a blowup set. Thus, by Lemma~\ref{lem:reduced-blowup-set}, for $j\neq j_p$ the set $\hat{U}_j$ is diffeomorphic to an open disk.
Moreover, $\hat{U}_{j_p}$ is homeomorphic to an open disk by a homeomorphism which is a diffeomorphism on $\hat{U}_{j_p}\setminus \{[p]\}$.
So we may view $\{\hat{H}_j\}_j$ as a one parameter family of functions on a disk, which are Morse except for at the point $[p]$.
Away from $[p]$, this family depends smoothly on the parameter.
If $U$ contained any elliptic-regular or hyperbolic-regular points of $F$ then, for some values of $j$, there will be a non-degenerate critical point of $\hat{H}_j$ inside $\hat{U}_j$. Nevertheless, 
since focus-focus points admit a neighborhood in which they are the only singular point of the integrable
system (see Theorem~\ref{thm:normal-form}),
there is a sufficiently small neighborhood of $[p]$ in $\hat{M}$ such that each intersection
of this neighborhood with $\hat{M}_j$ does not contain any critical points of $\hat{H}_j$, not counting the point $[p]$ at which $\hat{H}_j$ is not smooth.

Let $\hat{S}:=\pi(S)$ and let $\hat{S}_j:=\pi_j(S)$. We will now change each $\hat{H}_j$ to move the critical points of $\hat{H}_j$ (not counting $[p]$) into $\hat{U}_j\setminus \hat{S}_j$ without changing anything in a neighborhood of $[p]$: let $\Phi_j: \hat{U}_j \to \hat{U}_j$ be a smooth diffeomorphism which is the identity in a neighborhood of $[p]$ and in a neighborhood of the boundary, and which moves all critical points close to the boundary. Denote the composition by $\hat{H}'_j:= \hat{H}_j \circ \Phi_j$. Then there exists a set $V$ with $S\subset V\subset U$ which, like $U$, can be assumed to be a blowup set such that, for each $j$, the function $\hat{H}_j'|_{\hat{V}_j}$
is a Morse function with zero critical points for all $j\in J(V)\setminus \{j_p\}$,
where $\hat{V}_j = (J^{-1}(j)\cap V)/\mbS^1$. Moreover, $\hat{H}_j'|_{\hat{V}_{j_p}\setminus \{[p]\}}$ is a Morse function with no critical points.

Now let $H'(q) := \hat{H}_{J(q)}'([q])$ for $q\in M$. 
By Lemma~\ref{lem:reduction-extending}, 
we obtain an integrable
system $F'=(J,H')$ which still has a focus-focus point at $p$ since $H$ and $H'$ 
agree in a neighborhood of $p$.

\textbf{Step 2:} \emph{Creating a flap with Dullin $\&$ Pelayo's \cite{dullin-pelayo} technique.}
In Example \ref{ex:dullin-pelayo}, we outlined Dullin $\&$ Pelayo's \cite{dullin-pelayo} method how to turn a focus-focus point into an elliptic-elliptic point at the `cost' of creating a flap. 

We now apply this technique to the focus-focus point $p$ in the integrable system $F'$ constructed in {\em Step 1} and obtain a new integrable system $F'' = (J,H'')$ in which $p$ is an elliptic-elliptic point. Since $p$ is an elliptic-elliptic point, it is a consequence of the local normal form (Theorem~\ref{thm:normal-form} that the point $[p]$ is either a local maximum or local minimum of $H''$ on the reduced space.

Without loss of generality, assume $[p]$ to be a local minimum. Furthermore, we may assume
that, outside of the set $V$, the system $F''$ coincides with $F'$ and hence with $F$.
This new integrable system $F''$ satisfies all of the desired properties except that the new flap produced around the image of $p$ is not necessarily large enough to contain all of $S$.

\textbf{Step 3:} \emph{A family of Morse functions on the disk.}
Recall $\mathbb{D}_{r} := \{z\in\C\mid \abs{z}^2<r^2\}$ for $r>0$.
Recall the set $V$ constructed in \emph{Step 1}, which is a blowup set,
and let $\hat{V} = \pi (V)\subset \hat{M}$.
Applying Lemma~\ref{lem:reduced-blowup-set} to $V$ and the image of the blowup set $S$ in each reduced space,
we conclude that each are diffeomorphic to a disk except at the focus-focus point where it is only homeomorphic to a disk but not diffeomorphic. So we can parameterize $\hat{V}\subset\hat{M}$ as a disk bundle over an interval, where $\hat{S}$ is a smaller (sub-)disk bundle over a smaller (sub-)interval.
More precisely, there exist 
\begin{itemize}
 \item intervals $]a,b[$ and $]\tilde{a},\tilde{b}[$ with
$]\tilde{a},\tilde{b}[ \ \subset\ ]a,b[\,\,$,
 \item real numbers $R_j>0$ for each $j\in\ ]a,b[$ and $r_j>0$ for each
$j\in\ ]\tilde{a},\tilde{b}[$ such that $j\mapsto R_j$ and $j\mapsto r_j$ are continuous
 and $r_j<R_j$ for all $j\in\ ]\tilde{a},\tilde{b}[\,\,$,
 \item sets $\mathcal{A}$ and $\mathcal{B}$ given by
\[
 \mathcal{A} = \{(z,j)\in\C\times\, ] a,b [\,\,\, \mid z\in\mathbb{D}_{R_j}\},
 \qquad
 \mathcal{B}= \{(z,j)\in\C\times\, ] \tilde{a},\tilde{b} [\,\,\, \mid z\in\mathbb{D}_{r_j}\}\] which thus satisfy $\mathcal{A}\subset\mathcal{B}$ and 
 \item a homeomorphism $\rho\colon \mathcal{A} \to \hat{V}$
 which is a diffeomorphism when restricted to $\mathcal{A}\setminus\{(0,j_p)\}$
      and which satisfies $\rho(z,j)\in J^{-1}(j)$ for all $(z,j)\in\mathcal{A}$ and 
 $\rho(\mathcal{B}) = \hat{S}$.
 \end{itemize}
 


\textbf{Step 4:} \emph{Enlarging the flap.}
Recall from Section~\ref{sec:para-reduction} and in particular Figure~\ref{fig:cerf} that, on the reduced space, a flap will correspond to the connected region below (or above) the level set of a Morse function at the level of an index 1 singular point. This level set containing the index 1 point has the shape of a curve with one self-intersection point, forming a loop which encloses a teardrop shaped region.
Our goal in this step is to edit this function to
expand the values of $j$ for which the teardrop region exists (\emph{Step 4a}),
and increase the size of the
teardrop region (\emph{Step 4b}), 
so that the image of the blowup set $S$
in the reduced space is contained in this teardrop region.
Then the set $S$ will lie on the new flap of the integrable system, as desired.
A sketch of this process on a single reduced space is shown in Figure~\ref{fig:blowupSetFlap}.


\begin{figure}[h]
 \centering
 \includegraphics[width = 300pt]{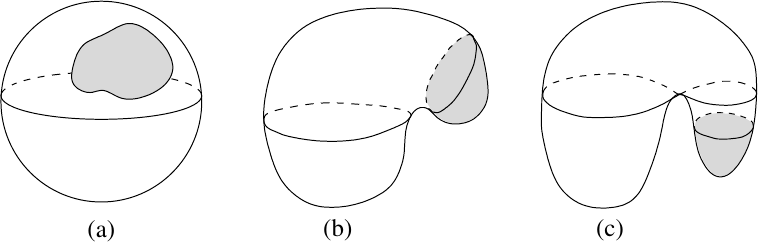}
 \caption{Note that a flap corresponds to a small `upward or downward pointing nose' of the reduced space which, being two dimensional, is a surface of some genus $g$ (with $g=0$ in this figure). From left to right: (a) is the original situation, (b) is obtained by creating
 a flap with the technique of Dullin \& Pelayo, and (c) is obtained by enlarging the flap to
 contain the blowup set as described in \emph{Step 4} of the proof of Proposition~\ref{prop:large-flap}.
 }
 \label{fig:blowupSetFlap}
\end{figure}

In {\em Step 2}, we have produced a flap around $[p]$ using Dullin $\&$ Pelayo's technique outlined in Example \ref{ex:dullin-pelayo}.
Recall $H''$ from {\em Step 2} and define $\hat{H}''$ via $\hat{H}''\circ\pi = H''$. Moreover, recall the map $\rho\colon\mathcal{A}\to\hat{V}\subset \hat{M}$ from {\em Step 3} and define $f\colon \mathcal{A} \to \R$ via $f=\hat{H}''\circ\rho$ and set $f_j:=f(\cdot,j)$.

Note that the function $z\mapsto f_j(z)$ is Morse with two critical points (of index 1 and 0) for all $j$ in an
open interval around $j_p$, {\em not} including $j=j_p$, and that $z\mapsto f_j(z)$ is Morse with no critical points outside of the closure of that interval. At the end points of the interval, $z\mapsto f_j(z)$ is not
Morse since it includes one degenerate point, see Figure~\ref{fig:cerf}.

\textbf{Step 4a:} \emph{Enlarging the flap to exist for a larger interval of $j$-values.}
First, we want to adapt $j\mapsto f_j$ so that the function has precisely two critical points for a larger interval of $j$-values that includes the interval $ ]\tilde{a}, \tilde{b}[\,$, therefore producing a longer flap in the integrable system. 
The idea is to reparameterize the $j$-parameter of $f_j$, except that we
have to be careful not to change any of the $f_j$ in a neighborhood of the boundary of the disk
so that it can still be glued back into the global function.
We will see that this is possible because of the general behavior of this function outside of the teardrop region, see Figure~\ref{fig:cerf}.


Let $\mathbb{D}:=\mathbb{D}_1$ denote the unit disk. Let $\chi\colon\mathbb{D}\to [0,1]$ be a smooth function which is zero in a neighborhood of the boundary $\partial \mathbb{D}$, but is identically one in a neighborhood of the origin of the disk large enough so that for each j, the set $\{R_jz \mid \chi(z)=1 \}\subset\mathbb{D}_{R_j}$ contains the teardrop shaped region.
Let $\Psi\colon  ]a,b[ \, \to \, ]a,b[$ be a smooth bijection 
such that the $j$-values for which $f_j$ has two critical points have been reparameterized to contain the interval $]\tilde{a},\tilde{b}[\,$ and such that $\Psi$ is the identity in a neighborhood of $j_p$. Let $\mu_j\colon \mathbb{D}_{R_j}\to\mathbb{D}_{R_{\Psi(j)}}$ be the scaling map $\mu_j(z) = \left(R_{\Psi(j)}/R_j\right)z$, and let $\tilde{g}_j\colon\mathbb{D}_{R_j}\to\R$ and $\tilde{g}\colon \mathcal{A}\to\R$ be given by
 \[\tilde{g}_j := (1-\chi) f_j + \chi (f_{\Psi(j)}\circ\mu_j), \quad \tilde{g} (\cdot, j ) := \tilde{g}_j(\cdot).\]
Denote by $]\alpha,\beta[$ the interval of $j$-values for which $\tilde{g}_j$ has two critical points.
Notice that $\tilde{g}_j=f_j$ near the boundary of the disk. Moreover, since $\Psi$ is the identity in a neighborhood of $j_p$, we see that $\tilde{g}=f$ in a neighborhood of $(0,j_p)\in \mathcal{A}$ as well. In the intermediate region $\mathcal{R} = \{\chi\neq 1\}\cap \{\chi\neq 0\}$ the function $\tilde{g}_j$ may have additional critical points. 
By adding a smooth perturbation $G_j\colon \mathcal{A} \to \R$ supported in $\{\chi \neq 0\}$ and depending smoothly on $j$, we may eliminate all critical points in $\mathcal{R}$. 
We denote the resulting function by $\tilde{f}_j\colon\mathbb{D}_{R_j}\to\R$, that is \[ \tilde{f}_j:= (1-\chi) f_j + \chi (f_{\Psi(j)}\circ\mu_j+G_j). \]
Finally, we define $\tilde{f}\colon \mathcal{A}\to\R$ by $\tilde{f}(\cdot,j):=\tilde{f}_j(\cdot)$.





Altogether, starting with $f$ we have now produced the function $\tilde{f}$ which satisfies:
  \begin{itemize}
   \item $\tilde{f}_j$ is a Morse function with no critical points for $a<j<\alpha$ or $\beta<j<b$;
   \item $\tilde{f}_j$ is a Morse function with one critical point of index 1 and one critical point of index 0 for $\alpha<j<\beta$ with $j\neq j_p$. By performing an affine transformation of coordinates if necessary, we may, and do, assume that the 
rank zero critical point is the origin for all $j\in \, ]\alpha, \beta [ \, \setminus \{j_p\}$;
   \item $\tilde{f}_\alpha$ and $\tilde{f}_\beta$ are smooth functions with exactly one critical point (which 
   is degenerate);
   \item $\tilde{f}_{j_p}|_{\mathbb{D}_{R_{j_p}}\setminus\{0\}}$ is Morse with one critical point
    of index 1. Moreover, $0\in\mathbb{D}_{R_{j_p}}$ is a local minimum of $z\mapsto f_{j_p}(z)$;
   \item for all $j\in \ ]a,b[$ there is a neighborhood $\mathcal{O}$ of the boundary $\partial\mathbb{D}_{R_j}$  of $\mathbb{D}_{R_j}$ such that
    $\tilde{f}_j|_{\mathcal{O}}$
    is Morse with zero critical points.
  \end{itemize}
Thus, $f$ is of the form discussed in Section~\ref{sec:para-reduction} and has a graph
as shown in Figure~\ref{fig:cerf}.

\textbf{Step 4b:} \emph{Enlarging the flap to contain a larger set for each fixed $j$.}
Now that the teardrop region exists for all desired values of $j$, we want to expand the teardrop
region to contain the blowup set.
Recall from Step 3 that $r_j$ is the radius of the sub-disk corresponding to $\hat{S}$.
For $j\in\   ]a,b[$, let $\psi_j: \mathbb{D}_{R_j} \to \mathbb{D}_{R_j}$ be a diffeomorphism shrinking $\mathbb{D}_{(2r_j+1)/3}$ and acting as the identity outside of $\mathbb{D}_{(r_j+2)/3}$
such that $(z,j)\mapsto \psi_j(z)$ is smooth.
For $\alpha<j<\beta$, the level set of the index 1 point of $\tilde{f}_j\circ\psi_j$ 
defines a teardrop region in $\mathbb{D}_{R_j}$ as in Figure~\ref{fig:cerf}.
Choose $\psi_j$ to shrink $\mathbb{D}_{r_j}$ enough so that $\mathbb{D}_{r_j}$ is contained in this teardrop region.

One of the simplest cases in the work of Cerf~\cite{Cerf}
states that the birth and death of critical points in families of Morse functions
in two dimensions
is generically of the form given in Equation~\eqref{eqn:birth-death} from Section~\ref{sec:para-reduction}.
Thus, we may assume that 
the degenerate points of $\tilde{f}_j\circ\psi_j$ are of the form $(x,y)\mapsto x^3+jx+y^2$ for some local coordinates $(x,y)$ on the disk while retaining all of the other properties of $\tilde{f}_j\circ\psi_j$
listed above.

Now, set \[g(z,j) := \tilde{f}(\psi_j(z),j)\colon \mathcal{A} \to\R\]
and abbreviate $g_j:=g(\cdot,j)$.  
  
Then $g$ has the following properties:
   \begin{itemize}
    \item $g=\tilde{f}=f$ in a neighborhood of $(0,j_p)\in \mathcal{A}$;
    \item $g=\tilde{f}=f$ in a neighborhood of the boundary $\partial \mathcal{A}$;
    \item $g_j$ is a Morse function with no critical points for $j$ with $a<j<\alpha$ or $\beta<j<b$;
    \item $g_j$ is a Morse function with one critical point of index 1 and one critical point of index 0 for $j$ with $\alpha<j<\beta$ with $j\neq j_p$;
    \item all degenerate points can be locally modeled by $(x,y,j)\mapsto x^3 \pm (j-\gamma) x+y^2$;
    \item for each $j$ satisfying $\tilde{a}<j<\tilde{b}$ the function $g_j$ has exactly one critical
     point $q$ of index 1, and $g_j^{-1}(q)\cap \mathbb{D}_{r_j} = \varnothing$ (i.e.~the set $\mathbb{D}_{r_j}$ is in the teardrop region defined by $g_j$).
   \end{itemize}

Thus, $g$ has all of the properties that we desired our integrable system to have
in the reduced space around the point $p\in M$. Most importantly, it defines a teardrop shaped region (which
will lift to a flap of the integrable system) which entirely contains the blowup set.

 \textbf{Step 5:} \emph{Completing the proof.}
Let $g\colon \mathcal{A} \to\R$ be as described in {\em Step 4}, 
let $\rho\colon \mathcal{A} \to \hat{V}$ be as defined in {\em Step 3},
and let $\hat{H}''\colon \hat{M}\to\R$ be as defined in {\em Step 4}.
Define 
$\hat{H}'''\colon \hat{M}\to\R$ by
\[
 \hat{H}'''(q) := \begin{cases} g\circ \rho^{-1}(q), &q\in\hat{V},\\ \hat{H}''(q), & q\notin\hat{V}.\end{cases}
\]
Notice that $\hat{H}''$ and $\hat{H}'''$ are equal in a neighborhood of $[p]\in \hat{M}$
and outside of $\hat{V}$. So in particular $\hat{H}''$ and $\hat{H}'''$ only differ at points of the reduced spaces which correspond to free orbits of the $\mbS^1$-action,
and $H'''$ is smooth at such points since it is the composition of
smooth functions (note the smoothness of $\rho^{-1}$ 
at such points is guaranteed by Lemma~\ref{lem:reduced-blowup-set}). Therefore,
taking $\widetilde{H} :=\hat{H}'''\circ\pi$, the resulting pair $(J,\widetilde{H})$ 
forms an integrable system on $M$ by Lemma~\ref{lem:reduction-extending}.

Since the only degenerate points of $\hat{H}_j'''$ for any $j$ are of the form $x^3+tx\pm y^2$, the degenerate points of $(J,\widetilde{H})$ are parabolic as in Proposition~\ref{prop:cuspnormalform}.
By Lemma~\ref{lem:reduction-type}, all of the non-degenerate singular points of $(J,\widetilde{H})$ are of the desired type for the statement of the proposition we are proving. 
Finally, from the previous step we know that $\mathbb{D}_{r_j}$ lies in the teardrop
region of $g_j$, and therefore $\hat{S}_j = \rho_j(\mathbb{D}_{r_j})$ lies below the value of the index 1 critical point of $\hat{H}_j'''$ (i.e.~in the teardrop shaped region) for all relevant $j$. Thus,
$S$ lies entirely in the new flap and so item~\eqref{item:on-the-flap} of the statement of the proposition is also satisfied.
\end{proof}
   

Proposition~\ref{prop:large-flap} can also be applied to multiple focus-focus points,
producing an integrable system which has one flap corresponding to each focus-focus point in a way
that the new flaps do not interfere with each other:
 
\begin{corollary}\label{cor:simultaneous}
 Let $(M,\om,F=(J,H))$ be an integrable system such that $J$ generates an
 effective $\mbS^1$-action and such that all singular points of $(J,H)$ are non-degenerate,
 and let $\Gamma$ be the Karshon graph of the Hamiltonian $\mbS^1$-space $(M,\om,J)$. 
 Suppose that $p_1,\ldots,p_m$ are focus-focus points of $(M,\om,F)$
 and let $v_1,\ldots,v_m$ be the corresponding isolated vertices of $\Gamma$.
 Suppose that $\Gamma'$ is a Karshon graph which can be obtained from
 $\Gamma$ by performing a finite sequence of blowups at each of
 $v_1,\ldots,v_m$.
 Then, for each $i=1,\ldots,m$ there exists a blowup set (as in Definition~\ref{def:blowupset}) $S^i\subset M$ for the sequence of blowups
 at $v_i$ such that $\overline{S}^1,\ldots,\overline{S}^m$ are disjoint
 and for any open neighborhood $U$ of $\cup_{i=1}^m \overline{S}^i$ there exists
 an integrable system $(M,\om,\widetilde{F}=(J,\widetilde{H}))$ such
 that
 \begin{enumerate}
  \item $\widetilde{H}$ and $H$ coincide outside of $U$,
  \item $p_1,\ldots,p_m$ are each an elliptic-elliptic singular point of $(M,\om,\widetilde{F})$,
  \item $(M,\om,\widetilde{F})$ has $m$ flaps and the images of $S^1,\ldots, S^m$ each lie on a single flap.
 \end{enumerate}
 Moreover, $(M,\om,\widetilde{F})$ is a hypersemitoric system.
\end{corollary}
 
\begin{proof}
 Since $(M,\om,J)$ admits the blowups in question, there must exist blowup sets
 $S^1,\ldots,S^m$ such that $\overline{S}^1,\ldots,\overline{S}^m$ are disjoint
 and thus they admit open neighborhoods $U^i \supset \overline{S}^i$ such that
 $U^1,\ldots,U^m$ are disjoint.
 Let $U$ be any open neighborhood of $\cup_{i=1}^m \overline{S}^i$, and we may assume that
 $U = \cup_{i=1}^m U^i$.
 Using induction, we will now show that we can apply Proposition~\ref{prop:large-flap} 
 to each of $U^1,\ldots,U^m$ in turn.
 
 Since the original system $(M,\om,F)$ has no degenerate points,
 we now apply the proposition in $U^1$. 
 Supposing that the proposition has already been applied
 to $U^1,\ldots,U^{i-1}$, 
 the only degenerate points of the new system will 
 lie in $U^1\cup\ldots\cup U^{i-1}$, which is disjoint from $U^i$.
 Thus, $U^i$ contains no degenerate points of the new system and we may apply Proposition~\ref{prop:large-flap}.
 Therefore, by induction, we conclude that we 
 can apply Proposition~\ref{prop:large-flap} to all of $U^1,\ldots,U^m$.
 At the end of this process, we are left with a hypersemitoric system $(M,\om,\widetilde{F})$
 which has the desired properties.
\end{proof}

\subsection{Proof of Theorem~\ref{thm:extending}}
\label{sec:proof}
Now we are prepared to prove Theorem~\ref{thm:extending}.

\begin{proof}[Proof of Theorem~\ref{thm:extending}]
Let $(M,\om,J)$ be a Hamiltonian $\mbS^1$-space and let $\Gamma:=\Gamma(J)$ be the associated
Karshon graph. 
Suppose that there exists a hypersemitoric system $(\mathcal{M},\Omega,(\mathcal{J},\mathcal{H}))$ where the Karshon graph $\Gamma(\mathcal{J})$ of the underlying Hamiltonian $\mbS^1$-space $(\mathcal{M}, \Omega ,\mathcal{J})$ satisfies $\Gamma(\mathcal{J}) =\Gamma$. Then, by Karshon~\cite{karshon}, there exists an $\mbS^1$-equivariant symplectomorphism $\Phi: (M,\om,J) \to (\mathcal{M},\Omega,\mathcal{J})$. Now define $H:= \mathcal{H} \circ \Phi$, and then $(M, \omega, (J,H))$ is a hypersemitoric system that extends the Hamiltonian $\mbS^1$-space $(M,\om,J)$, thus completing the proof of Theorem~\ref{thm:extending}.
Thus, the remainder of the proof is devoted to finding a hypersemitoric system whose underlying $\mbS^1$-space
has $\Gamma$ as its Karshon graph.

{\bf Case 1: $\Gamma$ has no fat vertex.}
If $\Gamma$ has no fat vertex then all fixed points of $J$ are isolated. By Lemma~\ref{lem:isolated-points-extend}, the Hamiltonian $\mbS^1$-space $(M,\om,J)$ then can be extended to a toric system, which is in particular a hypersemitoric system, thus proving Theorem~\ref{thm:extending} in this case.

{\bf Case 2: $\Gamma$ has at least one fat vertex.}
Due to Theorem~\ref{thm:karshon-minimal}, we know that $\Gamma$ can be obtained from a minimal Karshon graph $\Gamma_\mathrm{min}$ by a finite sequence of $\mbS^1$-equivariant blowups which, according to Lemma~\ref{lem:stages}, can be performed in three separate stages. By Karshon's classification of minimal Hamiltonian $\mbS^1$-spaces (cf.\ Theorem~\ref{thm:karshon-minimal}), $\Gamma_\mathrm{min}$ either corresponds to an $\mbS^1$-space which extends to a toric system (that lives on $\CP^2$ or on a Hirzebruch surface) or corresponds to a ruled surface.

\textbf{Case 2a: $\Gamma_\mathrm{min}$ extends to a toric system.} 
If $\Gamma_\mathrm{min}$ extends to a toric system\\ $(M_\mathrm{min},\om_\mathrm{min},F_\mathrm{min})$ on $\CP^2$ or on a Hirzebruch surface, we will show how to employ Lemma~\ref{lem:stages} to perform the necessary blowups to obtain a hypersemitoric system with the required Karshon graph $\Gamma$. Before we explain the details below, let us briefly outline the steps we will take. First,
we will perform blowups on the components of $\Gamma$ containing the vertices with maximal and minimal $J$-value. Recall that, by assumption, there is at least one fat vertex, i.e.\ at least one fixed surface. 
Next we will use a fixed surface to perform blowups which produce, for each remaining component of the graph $\Gamma$, one isolated fixed point. This point will correspond to a focus-focus point in the associated integrable system. Finally, we perform a sequence of blowups on each of these isolated points to obtain the remaining components of $\Gamma$,
carrying along the integrable system by replacing the focus-focus points by elliptic-elliptic points on flaps
and performing blowups on these elliptic-elliptic points.
For an example of this process, see
Section~\ref{sec:follow-ex} in which we apply this technique to a
specific $\mbS^1$-space.

\textbf{Stage 1:} \emph{Adjusting the connected components of the vertices with maximal and minimal $J$-value.}  As
in the first stage of Lemma~\ref{lem:stages}, we perform blowups on the graph $\Gamma_\mathrm{min}$ to obtain
a new graph $\Gamma'$ such that the connected components of the vertices with maximal and minimal $J$-value of $\Gamma'$ are equal to those of $\Gamma$, with the possible exception of the normalized area labels on the fat vertices. 
Since at most one of the maximal or minimal components is not a fat vertex in this case, this produces at most two non-trivial chains of $\Z_k$-spheres and thus, by Lemma~\ref{lem:toric-extending}, the graph $\Gamma'$ is the Karshon graph of a Hamiltonian $\mbS^1$-space which can be extended to a toric system $(M',\om',F')$. This toric system has at least one
fixed surface since the graph $\Gamma$ has at least one fat vertex.

\textbf{Stage 2:} \emph{Producing the focus-focus points.} 
After achieving the same minimal and maximal components for $\Gamma'$ and $\Gamma$ in the previous stage, as
in Stage 2 of Lemma~\ref{lem:stages}, we now want to adjust the remaining connected components, i.e., the components of $\Gamma\setminus \Gamma'$.
Notice that any minimal Karshon graph with at least one fat vertex has exactly two components, since otherwise the additional component would have to be an isolated fixed point with weights $+1$ and $-1$ which could be removed by a blowdown of type~\ref{case:blowup-surface}. Thus, the set of non-extremal
connected components of $\Gamma$ is exactly the same as the set of connected components
of $\Gamma\setminus\Gamma'$.

Let $l\in\Z$ be the number of connected components of $\Gamma\setminus \Gamma'$.
Each component of $\Gamma$ which does not contain a vertex with the maximal or minimal
possible $J$-value (i.e.~each element of $\Gamma\setminus\Gamma'$) is called an \emph{island}.
The aim of this stage is to perform blowups on the fixed surface(s) of $(M',\om',F')$, that is, to perform wall chops on the vertical walls of the polygon of $(M',\om',F')$, to produce one focus-focus point corresponding to each island\footnote{semitoric blowups will also be used to produce focus-focus points in this way
will also be used in the upcoming Hohloch \& Sabatini \& Sepe \& Symington~\cite{HSSS2}, first announced at Poisson 2014 by Daniele Sepe.}.
In the next stage, we will then perform a sequence of blowups on those focus-focus points to produce the missing components. But first we have to work backwards to determine the $J$-value that these focus-focus points should have.

Note that any minimal model which has at least one fixed surface does not have any islands. Thus, 
by performing a sequence of blowdowns on $\Gamma$, each
island of $\Gamma$ can be reduced to a single isolated vertex and eventually removed.
For each island, perform as many blowdowns as necessary until it is an isolated vertex and denote the resulting graph by $\Gamma''$.


We will now show that, on the other hand, $\Gamma''$ can also be obtained from $\Gamma'$ by performing $\mbS^1$-equivariant blowups on the fixed surface(s) of the Hamiltonian $\mbS^1$-space $(M',\om',J')$ from Stage 1 as described in Case~\ref{case:blowup-surface} from Section~\ref{sec:karshon-blowups}:

Consider first the situation that there is only one fixed surface $\Sigma$ and assume without loss of generality $\Sigma  = J^{-1}(j_\mathrm{min})$. Denote the $J$-values of the isolated points in $\Gamma''$ by $j_\mathrm{min}+\lambda_\ell$ for $\ell = 1,\ldots, l$. Then $\Gamma''$ is obtained from $\Gamma'$ by performing $l$ blowups of sizes $\lambda_1,\ldots, \lambda_l$ on points in the fixed surface $\Sigma$. Note that $\Sigma$ is large enough to admit such blowups because $\Gamma$ can be obtained from $\Gamma_\mathrm{min}$
by a sequence of blowups performed in the order specified in Lemma~\ref{lem:stages}. In case $\Sigma = J^{-1}(j_\mathrm{max})$, denote the $J$-values by $j_\mathrm{max}-\lambda_\ell$ and proceed analogously.
 
Otherwise, $(M',\om',J')$ has two fixed surfaces $\Sigma_\mathrm{min} = J^{-1}(j_\mathrm{min})$ and $\Sigma_\mathrm{max}=J^{-1}(j_\mathrm{max})$.
Since $\Gamma$ can be obtained from $\Gamma_\mathrm{min}$
by a sequence of blowups performed in the order specified in Lemma~\ref{lem:stages}, we conclude
that there exists an $m$
such that blowups of size $\lambda_1,\ldots,\lambda_m$ on $\Sigma_\mathrm{min}$ and blowups of size $\lambda_{m+1},\ldots,\lambda_l$ on $\Sigma_\mathrm{max}$ produce $\Gamma''$ from $\Gamma$. Here the $J$-values of the isolated points in $\Gamma''$ are 
$$j_\mathrm{min}+\lambda_1,\ldots,j_\mathrm{min}+\lambda_m, \ j_\mathrm{max}-\lambda_{m+1},\ldots,j_\mathrm{max}-\lambda_l.$$

Thus, we conclude that in both cases $\Gamma''$ can be obtained from $\Gamma'$ by performing blowups on the fixed surface(s) of $(M',\om',J')$.
Since the integrable system $(M',\om',F')$ is toric, and in particular semitoric, Lemma~\ref{lem:semitoric-blowup} implies that the blowups used to obtain $\Gamma''$ from $\Gamma'$ can be realized by performing wall chops on the marked polygon invariant of $(M',\om',F')$. 
Let $(M'',\om'',F'')$ be a semitoric system associated to the resulting marked polygon (see Remark~\ref{rmk:stblowup-system}), 
and note that $\Gamma''$ is the Karshon graph of the underlying $\mbS^1$-space of $(M'',\om'',F'')$.
Let $p_1,\ldots, p_l\in M''$ be the resulting focus-focus points of this system.
Note that we may, and do, choose to perform the semitoric blowups in such
a manner that each level set of $F''\colon M''\to\R^2$ contains at most
one focus-focus point (so the fibers of $F''$ which contain focus-focus points
are all single-pinched tori).

\textbf{Stage 3:} \emph{Constructing the islands.} Now we follow Stage 3 of Lemma~\ref{lem:stages}. For $\ell=1,\ldots, l$, we will start with the focus-focus point $p_\ell$, corresponding to a vertex $v_\ell$ of $\Gamma''$, and construct the $\ell^{\mathrm{th}}$ island of $\Gamma$ by performing a sequence of blowups.
Since each island is connected in $\Gamma$, 
they can be obtained from the corresponding focus-focus point (i.e.\ an isolated vertex in $\Gamma''$) by a finite sequence of blowups on that point and on the new fixed points produced by this process (these are the inverses of the finitely-many blowdowns discussed
in the previous stage).
This is exactly the setting of Proposition~\ref{prop:large-flap}. 

By Corollary~\ref{cor:simultaneous}, we may now apply the technique of
 Proposition~\ref{prop:large-flap} on each focus-focus
point $p_1,\ldots,p_k$ simultaneously and thus conclude that
there exists a function $\overline{H}''$ such that the
system $(M'',\om'', \overline{F}''=(J'', \overline{H}''))$ is a hypersemitoric system,
and the underlying $\mbS^1$-space is still $(M'',\om'', J'')$ which has
Karshon graph $\Gamma''$.
%
%
%
%

Each $p_\ell$ is now an elliptic-elliptic singular point of $(M'',\om'', \overline{F}'')$. By Proposition~\ref{prop:large-flap}, the flap containing $p_\ell$ is large enough to admit the desired sequence of $\mbS^1$-equivariant blowups (of the desired sizes) by performing toric blowups on the flap.
After performing this sequence of toric blowups on the flaps we are left with a hypersemitoric system of which the
underlying $\mbS^1$-space has the desired Karshon graph $\Gamma$.
This completes the proof in this case.


\textbf{Case 2b: The minimal model is a ruled surface.}
Recall that a ruled surface is a sphere bundle over a surface $\Sigma$ of genus $g$. Consider its graph $\Gamma_\mathrm{min}$. We may assume $g>0$ since the minimal model in the case $g=0$ extends to a toric system which we already treated in {\em Case 2a}.
Let $\Gamma_{\mathrm{min}}$ denote the Karshon graph of this minimal model.
We will follow nearly the same stages as in {\em Case 2a}. 
Keep in mind that this minimal model has two fixed surfaces corresponding to the north and south poles of the sphere (see Definition~\ref{def:ruledMfd}).
Following the description of the effect of a blowup on a Karshon graph from
Lemma~\ref{lem:casesGraphChange}, we see that 
$\mbS^1$-equivariant 
blowups cannot remove fixed
surfaces 
of the $\mbS^1$-action
completely (they can only make them smaller), so $\Gamma$ necessarily also has two fixed fat vertices
corresponding to fixed surfaces of the related $\mbS^1$-space. Thus, there are no edges which connect to the maximal or
minimal vertices of $\Gamma$, so \textit{Stage 1} is trivial in this case, implying that the Karshon graph $\Gamma'$ produced in {\em Stage 1} satisfies $\Gamma' = \Gamma_{\mathrm{min}}$.

For \textit{Stage 2}, as in the previous case (and following Stage 2 of Lemma~\ref{lem:stages}) we perform a series of
blowups of type~\ref{case:blowup-surface} on $\Gamma' = \Gamma_{\mathrm{min}}$ to produce a new Karshon
graph $\Gamma''$. 
By Lemma~\ref{lem:stblowups-hst},
given any $\Gamma''$ produced in that way from a ruled surface, there exists a hypersemitoric system
$(M'',\om'',F'')$ which has $\Gamma''$ as its Karshon graph.

Finally, \textit{Stage 3} works exactly the same as in the previous case, and the proof is complete.
\end{proof}


From the algorithm in the above proof, we automatically have the following
refined version of Theorem~\ref{thm:extending}.

\begin{corollary}\label{cor:refined-extending}
 Let $(M,\om,J)$ be a Hamiltonian $\mbS^1$-space. Then there exists a smooth function $H\colon M\to\R^2$
 such that $(M,\om,F=(J,H))$ is a hypersemitoric system such that:
 \begin{enumerate}
   \item every degenerate orbit of $(M,\om,F)$ lies in a cuspidal torus (see Figure~\ref{fig:cusp-torus}), and every hyperbolic-regular point lies in a double torus (see Figure~\ref{fig:double-torus});
  \item there is at most one focus-focus point in each fiber of $F$;
  \item as long as $k\notin \{0,\pm 1\}$, all $\Z_k$-spheres of $(M,\om,J)$ consist entirely of elliptic-regular points of $(M,\om,F)$;
  \item if $(M,\om,J)$ has less than two fixed surfaces, or if it has two fixed surfaces which
   are both diffeomorphic to spheres, then $(M,\om,F)$ has no singular points of hyperbolic-elliptic type.
  \item $(M,\om,F)$ has no swallowtails (see Section~\ref{sec:parabolic_points}).
 \end{enumerate}
\end{corollary}

\bibliographystyle{alpha}
\bibliography{ref}

{\small
  \noindent
  \\
  Sonja Hohloch\\
  University of Antwerp\\
  Department of Mathematics\\
  Middelheimlaan 1\\
  B-2020 Antwerp, Belgium\\
  {\em E\--mail}: \texttt{sonja.hohloch@uantwerpen.be}
}

 {\small
   \noindent
   \\
   Joseph Palmer\\
   Department of Mathematics,\\
   Seeley Mudd Building,\\
   Amherst College,\\ Amherst, MA, USA, 01002.\\ 
   {\em E-mail:} \texttt{jpalmer@amherst.edu}
 }

\end{document}